\documentclass[10pt,a4paper]{amsart}

\usepackage{mathtools}
\usepackage{amsmath}
\usepackage{amssymb}
\usepackage{amsthm}
\usepackage{amscd}
\usepackage[initials,msc-links]{amsrefs}
\usepackage[ansinew]{inputenc}
\usepackage[english=american]{csquotes}
\usepackage[font=footnotesize, justification=centering]{caption}
\usepackage{calc}
\usepackage{color}
\usepackage[dvipsnames]{xcolor}
\usepackage{bm}
\usepackage{enumerate}
\usepackage{mathrsfs}
\usepackage{dsfont}
\usepackage{hyperref}
\hypersetup{
	bookmarks=true,         % show bookmarks bar?
	unicode=false,          % non-Latin characters in Acrobatâs bookmarks
	pdftoolbar=true,        % show Acrobatâs toolbar?
	pdfmenubar=true,        % show Acrobatâs menu?
	pdffitwindow=false,     % window fit to page when opened
	pdfstartview={FitH},    % fits the width of the page to the window
	pdftitle={Multiscale Young measures},    % title
	pdfauthor={Author},     % author
	pdfsubject={Subject},   % subject of the document
	pdfcreator={Creator},   % creator of the document
	pdfproducer={Producer}, % producer of the document
	pdfkeywords={keyword1, key2, key3}, % list of keywords
	pdfnewwindow=true,      % links in new PDF window
	colorlinks=true,       % false: boxed links; true: colored links
	linkcolor=orange,          % color of internal links (change box color with linkbordercolor)
	citecolor=blue,        % color of links to bibliography
	filecolor=magenta,      % color of file links
	urlcolor=cyan           % color of external links
}
\usepackage{tikz}
\usepackage{tikz-cd} 

\makeatletter
\DeclareFontFamily{OMX}{MnSymbolE}{}
\DeclareSymbolFont{MnLargeSymbols}{OMX}{MnSymbolE}{m}{n}
\SetSymbolFont{MnLargeSymbols}{bold}{OMX}{MnSymbolE}{b}{n}
\DeclareFontShape{OMX}{MnSymbolE}{m}{n}{
	<-6>  MnSymbolE5
	<6-7>  MnSymbolE6
	<7-8>  MnSymbolE7
	<8-9>  MnSymbolE8
	<9-10> MnSymbolE9
	<10-12> MnSymbolE10
	<12->   MnSymbolE12
}{}
\DeclareFontShape{OMX}{MnSymbolE}{b}{n}{
	<-6>  MnSymbolE-Bold5
	<6-7>  MnSymbolE-Bold6
	<7-8>  MnSymbolE-Bold7
	<8-9>  MnSymbolE-Bold8
	<9-10> MnSymbolE-Bold9
	<10-12> MnSymbolE-Bold10
	<12->   MnSymbolE-Bold12
}{}
\let\llangle\@undefined
\let\rrangle\@undefined
\DeclareMathDelimiter{\llangle}{\mathopen}%
{MnLargeSymbols}{'164}{MnLargeSymbols}{'164}
\DeclareMathDelimiter{\rrangle}{\mathclose}%
{MnLargeSymbols}{'171}{MnLargeSymbols}{'171}
\makeatother

\newtheorem{theorem}{Theorem}[section]
\newtheorem{lemma}[theorem]{Lemma}
\newtheorem{proposition}[theorem]{Proposition}
\newtheorem{corollary}[theorem]{Corollary}
\newtheorem{conjecture}[theorem]{Conjecture}

\theoremstyle{definition}
\newtheorem{definition}[theorem]{Definition}
\newtheorem{remark}[theorem]{Remark}
\newtheorem{example}{Example}

\usepackage{hyperref}

\newcommand{\Crm}{\mathrm{C}}
\newcommand{\Lrm}{\mathrm{L}}
\newcommand{\Wrm}{\mathrm{W}}

\newcommand{\Acal}{\mathcal{A}}

\newcommand{\Dcal}{\mathcal{D}}

\newcommand{\Mcal}{\mathcal{M}}

\newcommand{\Ycal}{\mathcal{Y}}

\newcommand{\Bfrak}{\mathfrak{B}}

\newcommand{\Mbf}{\mathbf{M}}
\newcommand{\Rbf}{\mathbf{R}}

\newcommand{\Ybf}{\mathbf{Y}}

\newcommand{\Abb}{\mathbb{A}}
\renewcommand{\Bbb}{\mathbb{B}}

\newcommand{\Nbb}{\mathbb{N}}

\newcommand{\Pbb}{\mathbb{P}}

\newcommand{\Sbb}{\mathbb{S}}
\newcommand{\Tbb}{\mathbb{T}}

\newcommand{\Zbb}{\mathbb{Z}}

\DeclareMathOperator{\esssup}{ess\,sup}

\DeclareMathOperator{\id}{id}

\DeclareMathOperator{\graph}{graph}

\DeclareMathOperator{\curl}{curl}

\DeclareMathOperator{\spn}{span}

\DeclareMathOperator{\supp}{supp}
\newcommand{\Leb}{\mathscr L}

\newcommand{\area}[1]{\langle\, #1 \, \rangle}
\newcommand{\areaB}[1]{\Big\langle\, #1 \, \Big\rangle}

\newcommand{\set}[2]{\{\, #1 \;\textup{\textbf{:}} \; #2 \,\}}
\newcommand{\setn}[2]{\{#1 \, \textup{\textbf{:}} \, #2\}}
\newcommand{\setb}[2]{\bigl\{\, #1 \, \textup{\textbf{:}}\, \, #2 \,\bigr\}}
\newcommand{\setB}[2]{\Bigl\{\,\, #1 \ \textup{\textbf{:}}\ \, #2 \,\Bigr\}}
\newcommand{\setBB}[2]{\biggl\{\,\, #1 \ \textup{\textbf{:}}\ \, #2 \,\,\biggr\}}

\newcommand{\bpr}[1]{[ #1 ]}	
\newcommand{\bbpr}[1]{[\![ #1 ]\!]}	
\newcommand{\dpr}[1]{\langle #1 \rangle}

\newcommand{\dprb}[1]{\bigl\langle #1 \bigr\rangle}

\newcommand{\dprBB}[1]{\biggl\langle #1 \biggr\rangle}

\newcommand{\ddprb}[1]{\llangle#1\rrangle}

\newcommand{\cl}[1]{\overline{#1}}

\newcommand{\dd}{\;\mathrm{d}}

\newcommand{\N}{\mathbb{N}}
\newcommand{\R}{\mathbb{R}}

\newcommand{\loc}{\mathrm{loc}}

\newcommand{\sym}{\mathrm{sym}}

\newcommand{\per}{\mathrm{per}}
\newcommand{\reg}{\mathrm{Reg}}
\newcommand{\sing}{\mathrm{Sing}}

\newcommand{\toweak}{\rightharpoonup}
\newcommand{\toweakstar}{\overset{*}\rightharpoonup}
\newcommand{\toweakstarY}{\overset{*}\rightharpoondown}

\newcommand{\todown}{\downarrow}

\newcommand{\embed}{\hookrightarrow}
\newcommand{\cembed}{\overset{c}{\embed}}

\newcommand{\BigO}{\mathrm{\textup{O}}}

\newcommand{\sbullet}{\begin{picture}(1,1)(-0.5,-2)\circle*{1.5}\end{picture}}
\newcommand{\frarg}{\,\sbullet\,}

\newcommand{\BD}{\mathrm{BD}}

\newcommand{\BDY}{\mathbf{BDY}}

\newcommand{\toY}{\overset{\Ybf}{\to}}
\newcommand{\toYY}{\overset{\Ybf^2}{\to}}

\newcommand{\eps}{\epsilon}

\DeclareMathOperator{\Tan}{Tan}

\newcommand{\proofstep}[1]{\textit{#1}}

\def\thetant#1{\mathchoice 
	{\XXint\displaystyle\textstyle{#1}}% 
	{\XXint\textstyle\scriptstyle{#1}}% 
	{\XXint\scriptstyle\scriptscriptstyle{#1}}% 
	{\XXint\scriptscriptstyle\scriptscriptstyle{#1}}% 
	\!\int} 
\def\XXint#1#2#3{{\setbox0=\hbox{$#1{#2#3}{\int}$} 
		\vcenter{\hbox{$#2#3$}}\kern-.5\wd0}} 
\def\dashint{\,\thetant-}

\newcommand{\restrict}{\mathbin{\vrule height 1.35ex depth 0pt width
		0.14ex\vrule height 0.16ex depth 0pt width 0.6ex}}
	
\makeatletter
\def\@tocline#1#2#3#4#5#6#7{\relax
	\ifnum #1>\c@tocdepth % then omit
	\else
	\par \addpenalty\@secpenalty\addvspace{#2}%
	\begingroup \hyphenpenalty\@M
	\@ifempty{#4}{%
		\@tempdima\csname r@tocindent\number#1\endcsname\relax
	}{%
		\@tempdima#4\relax
	}%
	\parindent\z@ \leftskip#3\relax \advance\leftskip\@tempdima\relax
	\rightskip\@pnumwidth plus4em \parfillskip-\@pnumwidth
	#5\leavevmode\hskip-\@tempdima
	\ifcase #1
	\or\or \hskip 1em \or \hskip 2em \else \hskip 3em \fi%
	#6\nobreak\relax
	\dotfill\hbox to\@pnumwidth{\@tocpagenum{#7}}\par
	\nobreak
	\endgroup
	\fi}
\makeatother

\renewcommand{\eps}{\varepsilon}
\renewcommand{\phi}{\varphi}
\newcommand{\M}{\mathcal M}

\DeclareMathOperator{\Y}{\bf Y}
\DeclareMathOperator{\E}{\bf E}

\newcommand{\A}{\mathcal A}
\DeclareMathOperator{\B}{\mathcal B}

\newcommand{\Ahom}{{\A-\mathrm{hom}}}

\title[]{Generalized multi-scale Young measures} 
\author[A.~Arroyo-Rabasa]{Adolfo Arroyo-Rabasa}
\date{\today}
\address{Mathematics Institute, The University of Warwick, United Kingdom.}
\email{\href{mailto:adolfo.arroyo-rabasa@warwick.ac.uk}{adolfo.arroyo-rabasa@warwick.ac.uk}}
\author[J.~Diermeier]{Johannes Diermeier}
\date{\today}
\address{Institute of Applied Mathematics, University of Bonn, Bonn, Germany.}
\email{\href{mailto:diermeier@iam.uni-bonn.de}{diermeier@iam.uni-bonn.de}}
\thanks{A.~A.-R. has been supported by a scholarship from the \emph{Bonn International Graduate School}, the \emph{German Science Foundation} through the Emmy Noether junior research group DFG-grant {BE 5922/1-1}, and has also received funding from the European Research Council (ERC) under the European Union's Horizon 2020 research and innovation programme under grant agreement No 757254 (SINGULARITY). %(after he has finished his studies). 
	J. D. has been  supported by the \textit{Deutsche Forschungsgemeinschaft} through CRC 1060: The Mathematics of Emergent Effects.}
\begin{document}
	\maketitle
	\begin{abstract}
		This paper is devoted to the construction of generalized multi-scale Young measures, which are the extension of Pedregal's multi-scale Young measures [Trans. Amer. Math. Soc. 358 (2006), pp. 591--602]  to the setting of generalized Young 
		measures introduced by DiPerna and Majda [Comm. Math. Phys. 108 (1987), pp. 667--689]. As a tool for variational problems, these are well-suited objects for the study (at different length-scales) of oscillation and concentration effects 
		of convergent sequences of measures.
		Important properties of multi-scale Young measures such as compactness, representation of non-linear compositions, localization principles, and differential constraints are extensively developed in the second part of this paper.  As an application, we use this framework to address the $\Gamma$-limit characterization of the homogenized limit of convex integrals defined on spaces of measures satisfying a general linear PDE constraint. 
	\end{abstract} 
	
	\tableofcontents

	\section{Introduction}

	The notion of \emph{generalized surfaces} introduced by Young~\cite{young1937generalized-cur,young1942generalized-sur,young1942generalized-surII}, and known today as Young's measures, rests on the fundamental idea to consider functions as graphs.
	%
	%Young's pioneering idea to study graphs of functions and their notions of convergence rather than those of functions gave rise to what he introduces as
	%\emph{generalized surfaces}, today known as (classical) Young measures. 
	Young realized that the weak convergence of the graphs of a sequence of functions carries 
	substantially more information than the weak convergence of the functions themselves. In fact, this was the cornerstone leading to the 
	following fundamental principle of Young measures:
	Let $(u_k)_{k\in\N} \subset \Lrm^1(\Omega;\R^N)$ be a uniformly bounded sequence. Then (up to passing to a subsequence) there exists a family of probability measures $\{\nu_x\} \subset \mathrm{Prob}(\R^N)$, parameterized by $x$ in $\Omega$, such that 
	\begin{equation}\label{eq:Y1}
	g(u_{k}) \toweak u_g \quad \text{for all $\Lrm^1(\Omega)$}\quad \text{for some $g \in \Crm_c(\R^N)$}\,,
	\end{equation}
	where
	\begin{equation}\label{eq:Y2}
	u_g = \dashint_{\Omega} g(z) \dd \nu_x(z)\,.
	\end{equation}
	%and the family $\{\nu_x\}_{x \in \Omega} \subset \mathrm{Prob}(\R^N)$ is called the Young measure associated to $(u_{k_j})_{j \in \Nbb}$. 

	In the field of applications, the seminal work of M\"uller \cite{muller1999variational-mod} (among many others) adds up to a fair amount of applications in optimal design where the framework of
	Young measures plays a fundamental role in their development. 
	However, as pointed out by Tartar~\cite{tartar1995beyond-young-me} and Pedregal~\cite{pedregal2006multi-scale-you}, Young measures have their own drawbacks and 
	limitations. The \emph{first} important limitation is their incapability to keep track of concentration of mass (mass carried by the sequence may escape 
	to infinity, in $\R^N$, while leaving the limit in~\eqref{eq:Y1} unchanged). To solve this issue, DiPerna \& Majda adapted Young's ideas using a compactification of $\R^N$. They   
	extended classical Young measures to what today is known as generalized Young measures \cite{diperna1987oscillations-an} (see also \cite{alibert1997non-uniform-int}), which are capable of representing more general limits than~\eqref{eq:Y1}. More precisely, extending~\eqref{eq:Y1} to the representation of limits $\mu_f$ of the form 
	\begin{equation}\label{eq:Y3}
	f(u_k) \toweakstar \mu_f \quad \text{in $\M(\Omega)$ \quad among integrands satisfying $\tilde f \in \Crm(\beta \R^N)$},
	\end{equation}
	with $\tilde f(z) \coloneqq f(z)/(1 + |z|)$, and where $\beta \R^N$ is the Stone--\v Cech compactification of $\R^N$. Here, we omit the representation 
	formula for generalized Young measures as it is substantially more involved than~\eqref{eq:Y2}; we shall postpone this to the Appendix where for the convenience of the reader we give a 
	brief sketch of the construction and properties of (the different notions of) Young measures. The \emph{second} main drawback of Young's 
	construction is the failure to record patterns such as the direction or speed where oscillation (and/or concentration) occurs. This is easily illustrated by the following 
	one dimensional example. Fix a positive real $\alpha$ and consider the purely oscillatory sequence
	\[
	u_k(x) = \sin(k^\alpha x), \qquad k = 1,2,\dots
	\]
	Clearly, the choice of $\alpha$ significantly changes the length-scale period at which oscillations occur as $k$ tends to $\infty$. However, regardless of the 
	choice of $\alpha$, a change of variables argument shows the associated Young measure to this sequence is given by the family $\{\nu_x\}_{x \in \R}$ where $\nu_x = \nu_0$ for 
	almost every $x  \in \Omega$, and where $\nu_0$ is the probability measure satisfying
	\[
	\dpr{\nu_0,g} = \frac{1}{2\pi}\int_0^{2\pi} g(\sin y) \dd y \quad \text{for all $g \in \Crm_c(\R)$}\,.
	\]
	Hence, there is a need to extend the notion of Young measure to one that incorporates the dependence of the parameter $\alpha$. To record this information, Pedregal 
	considered the joint Young measure $\upsilon$ associated to the sequences of pairs 
	\[
	w_k(x) \coloneqq (\area{\chi_k},u_k) : x \mapsto (\area{kx},u_k(x))\,,
	\]   
	where $\area{x}$ denotes the equivalence class of the vector $x \in \R^d$ in the $d$-dimensional flat torus $Z \coloneqq \R^d/\Zbb^d$. {After performing a slicing argument, $\upsilon_x$ decomposes into} $\upsilon_x = \pi_\# \nu_x \otimes \nu_{x,\xi}$, where $\pi : Z \times \R^N \to Z$ is the canonical projection on the torus. {Pedregal} introduced the resulting family $\{\nu_{x,\xi}\}$ of probability measures on $\R^N$ as the associated \emph{two-scale Young measure}, in turn, designed to represent weak-limits of the form
	\begin{equation}
	g(w_k) \toweak U_g \quad \text{in $\Lrm^1(\Omega)$}\quad \text{for integrands $g \in \Crm_c(Z \times \R^N)$}\,.
	\end{equation}
	Similarly to the case of generalized Young measures, a sketch of Pedregal's construction is further discussed in the Appendix.   
	
	We are now in a position to give a rough description of the content and goals of this work. 
	
	\subsection{Main results}

	The \emph{first} goal of this paper is to introduce \emph{generalized multi-scale} Young measures (\enquote{multi-scale*} throughout the text for short) in the following sense. We combine the ideas of 
	DiPerna \& Majda with the approach from Pedregal to construct a {new type of Young measures}, capable of dealing with oscillation-concentration effects while also quantifying the length-scales where these phenomena occur. 
	For simplicity we shall restrict only to two-scale* Young measures. Effectively, we introduce a measure-theoretic tool to represent weak-$*$ limits of the form
	\begin{equation}\label{eq:Y5}
	f(\chi_k,u_k) \toweakstar \mu_f \quad \text{in $\M(\Omega)$ \quad amongst integrands $f$ satisfying $\tilde f \in \Crm(Z \times \beta \R^N)$}\,,
	\end{equation}
	where $\tilde f(\xi,z) \coloneqq f(\xi,z)/(1 + |z|)$. Next, we give the rigorous definition
	and state some of the main properties of two-scale* Young measures.
	
	\begin{definition}[two-scale* Young measures]
		A four-tuple $\bm{\nu} = (\nu,\lambda,\rho,\nu^\infty)$ is called a two-scale* Young measure on $\Omega$ with values in $E$ (a finite dimensional euclidean space) provided that
		\begin{enumerate}[(i)]
			\item $\nu$ is a weak-$*$ measurable from $\cl \Omega \times Z$ into the set of probability measures over $E$ such that the map
			$(x,\xi) \mapsto \dpr{\nu_{x, \xi}, |\frarg|}$ belongs to $\Lrm^1(\cl\Omega \times Z)$, 
			\item $\lambda$ is a bounded positive measure on $\cl \Omega$,
			\item $\rho$ is a weak-$*$ $\lambda$-measurable map from $\cl \Omega$ to the
			set of probability measures on the $d$-dimensional torus,
			\item $\nu^\infty$ is a weak-$*$ ($\lambda \otimes \rho_x)$-measurable from 
			$\cl \Omega \times Z$ into the set of probability measures over $\Sbb_E$ (the sphere of radius one in $E$).\footnote{The semi-product $\lambda \otimes \nu_x$, between a positive measure $\lambda \in \M(\Omega)$ and a weak-$*$ $\lambda$-measurable map $\nu : \Omega \to \mathrm{Prob}(K) : x \mapsto \nu_x$, is the measure of  $\Omega \times K$ defined as
				\[
				(\mu \otimes \nu_x )(U) \coloneqq \int_\Omega \int_K \chi_U(x,z) \dd \nu_x(z) \dd \mu(x), \qquad \text{for all $U \in \B(\Omega \times K)$}.
				\]
			}
		\end{enumerate}
		We denote the set of two-scale* Young measures by $\Y^2(\Omega;E)$. 
	\end{definition}

	\subsubsection{Representation via two-scale* Young measures}
	Here and in what follows \enquote{$\eps \searrow 0$} will denote a sequence of positive reals converging to zero, which heuristically shall represent a microscopic length-scale.  %, and $E$ is a finite dimensional euclidean space
	%whose open unit ball will be denoted by $\Bbb_E$. For example, $E = \R^N$ or $E = \Mbb_{m \times d}(\R)$ the space of $m \times d$ real-valued matrices. 
	%    Given a continuous integrand 
	%    $f \in \Crm(\cl\Omega \times \R^d \times \R^N)$, we define the \emph{recession function} of $f$ by
	%    \[
	%    f^\infty(x,\xi,\hat z) \coloneqq \lim_{\substack{x' \to x\\ {\xi}' \to {\xi}\\z' \to \hat z\\t \to \infty}} \frac{f(x',\xi',tz')}{t} \qquad (x,\xi,\hat z) \, \in \,\cl\Omega \times \Tbb^d \times \Bbb^N,
	%    \] 
	%    provided the limit in the right-hand side exists. 
	%\textcolor{blue}{My style would put the sequence in the theorem envoirment. If zou dont want to refer to the bounded total variation there, just copz the first sentence. The point is: the proof begins with let $\mu_\eps$ be as in the assumtions of the theorem, than you dont start looking for them before the beging of the theorem ;)}

	\begin{theorem}[representation]\label{thm:representation} Let $(\mu_\eps)_\eps\subset \M(\Omega;E)$ be a sequence of vector-valued measures with uniformly bounded total variation, i.e., 
		\[
		\sup_\eps |\mu_\eps|(\Omega) < \infty\,.
		\]
		The following representation result holds up to taking a subsequence of $(\mu_\eps)_\eps$. 
		%    	Let $\eps \searrow 0$ be a sequence of positive reals converging to zero and let $(\mu_\eps)_\eps$ be a uniformly bounded sequence of measures in 
		%    	$\M(\Omega;\R^N)$. 
		There exists a two-scale* Young measure $\bm{\nu} = (\nu,\lambda,\rho,\nu^\infty) \in \Y^2(\Omega;E)$ satisfying 
		the following fundamental property. Let $f : \Omega \times \Tbb^d \times E \to \R$ be a continuous integrand for which the recession function 
		\[
		f^\infty(x,\xi,z) \coloneqq \lim_{\substack{(x',\xi',z') \to (x,\xi,z)\\t \to \infty}} \frac{f(x',\xi',tz')}{t} \qquad (x,\xi,z) \in \cl \Omega \times \Tbb^d \times E,
		\]
		exists. Then, there exists a Radon measure $\mu_f$ on $\cl \Omega$ such that 
		\[
		f(\,\frarg\,, \,{\frarg/\eps}\,,\, u_\eps\,) \, \Leb^d \restrict \Omega \; \toweakstar \; \mu_f \quad \text{as measures on $\cl \Omega$}\,,
		\]
		and $\mu_f$ is characterized by the values
		\begin{align*}
		\mu_f(\omega) & = \int_{\omega} \Big( \int_{\Tbb^d} \Big( \int_E f(x,\xi , z) \dd \nu_x(z) \Big) \dd \xi \Big)  \dd x \\
		& 	\qquad + \; \int_{\cl\omega} \Big(\int_{\Tbb^d} \Big(\int_{\Bbb_E} f^\infty(x,\xi,z) \dd \nu_x^\infty(z) \Big) \dd \rho_x(\xi) \Big) \dd \lambda(x)\,,
		\end{align*}
		where $\omega$ ranges among all Borel subsets of $\cl \Omega$. In this case we say that $(\mu_\eps)_\eps$ generates the two-scale*  Young measure $(\nu,\lambda,\rho,\nu^\infty)$, in symbols
		\[
		\mu_\eps \toYY \bm{\nu}\,.
		\]
	\end{theorem}

	The \emph{second} objective, is to endow two-scale* Young measures with a measure-theoretic toolbox tailored for applications in the calculus of variations. Our hope is to 
	lay a transparent framework which casts the geometrical meaning of the blow-up methods (introduced by Fonseca \& M\"uller~\cite{fonseca1992quasi-convex-in}) into the context of two-scale analysis~\cite{bensoussan2011asymptotic-anal,nguetseng1989a-general-conve,allaire1992homogenization-}. 
	Based on a localization principle, this comprehends the representation of integral functionals arising from {$\Gamma$}-convergence (see~\cite{de-giorgi1975sulla-convergen,de-giorgi1984g-operators-and}) in the context of homogenization of PDE-constrained structures~\cite{braides1996homogenization-,muller1987homogenization-,braides2000mathcal-a-quasi,matias2015homogenization-}.
	
	Formally, this toolbox consists of establishing the following properties:
	
	\subsubsection{Fundamental properties of two-scale* Young measures}
	\begin{enumerate}
		\item Compactness. In a natural way, two-scale* Young measures are elements of due dual of the class of integrands $\E(\Omega;E)$ (see~Section~\ref{integrands}). In Proposition~\ref{cor:Compactnessofyoungmeasures} 
		we show the sequential \emph{ weak-$*$ compactness} (with respect to the weak-$*$ topology of $\E(\Omega;E)^*$) of {uniformly bounded} subsets $\Ycal$ of $\Y^2(\Omega;\E)$, that is, for sets such that
		\[
		\sup \; \setBB{ \int_{\cl \Omega} \int_{Z}\dpr{\nu_{\xi,x}, |\frarg|} \dd \xi \dd x \; + \; \lambda(\cl \Omega) }{(\nu,\lambda,\rho,\nu^\infty) \in \Ycal} < \infty\,.
		\] 
		This result is a fundamental step towards the proof of Theorem~\ref{thm:representation}.
		\item Localization. The relevance of Propositions~\ref{lem:locatregular} and~\ref{prop:loc_singular} is briefly explained as follows.  If a sequence $(\mu_\eps)_\eps$ generates a Young measure $(\nu,\lambda,\rho,\nu^\infty)$,
		then at $(\Leb^d + \lambda)$-almost every $x_0 \in \Omega$ we may find a \emph{blow-up sequence} of the original sequence (at $x_0$) that generates a (global) \emph{tangent two-scale* Young measure} 
		\[
		\bm{D\nu(x_0)} = (\nu_{x_0},D\lambda,\Gamma_\#^{\xi_0}\rho_{x_0},\nu^\infty_{x_0}) \quad \text{for some $\xi_0 \in Z$, and $D\lambda \in \Tan_1(\lambda,x_0)$}\,.
		\]
		Hence, extending the concept of {tangent} Young measure  introduced in~\cite{rindler2011lower-semiconti} by Rindler; see also \cite{rindler2012lower-semiconti}.   
	\end{enumerate}
	%To achieve this we cast the original ideas of DiPerna \& Majda~\cite{diperna1987oscillations-an} into the more functional scheme of Kristensen 
	%\& Rindler \cite{kristensen2010characterizatio}. 
	
	\subsubsection{Barycenter measures and second-scale convergence}
	
	Given a sequence $(u_j)_{k \in \Nbb}$ that generates a two-scale* Young measure $\bm \nu$, the representation Theorem~\ref{thm:representation} yields  that  
	\[
	u_j \toweakstar [\bm{\nu}]  \quad \text{as measures on $\cl \Omega$}\,,
	\]
	where $[\bm \nu]$ is the barycenter of $\bm{\nu}$ (see Definition~\ref{def:bary}). In this sense our notion of barycenter of a two-scale* Young measure coincides with 
	the one for generalized Young measures. We then define a weak-$*$ $(\Leb^d \restrict \cl \Omega + \lambda^s)$-measurable map $\bbpr{\bm{\nu}} : \cl \Omega \to \M(Z;E)$, called the \emph{second-scale barycenter} of $\bm{\nu}$, which has the property that
	\[
	\text{$\mu_\eps$ two-scale converges to $(\Leb^d + \lambda^s) \otimes \bbpr{\bm{\nu}}_x$}\,,
	\] 
	where the convergence above shall be understood as an extension of Nguetseng's concept of two-scale convergence~\cite{nguetseng1989a-general-conve,allaire1992homogenization-} (see Definitions~\ref{def:tsc} and~\ref{def:bary2}).
	In fact, via the compactness of two-scale* Young measures, we give a fairly short proof of the compactness of uniformly bounded sequences with respect to two-scale convergence; 
	see Corollary~\ref{cor:tsc}.

	\subsubsection{Structure of PDE-constrained two-scale* Young measures}
	
	Consider a homogeneous linear partial differential operator of order $k$ on $\R^d$ (with constant coefficients)  of the form 
	\[
	\A = \sum_{\substack{|\alpha| = k\\\alpha \in \N^d}}  A_\alpha\,\partial^\alpha, \qquad A_\alpha \in \textrm{Lin}(E;F).
	\]
	A vector-valued measure $\mu \in \M(\Omega;E)$ is called $\A$-free provided that 
	\[
	\A\mu = 0 \quad \text{in the sense of distributions on $\Omega$}.
	\]
	We say that two-scale* Young measure $\bm{\nu} \in \Y^2(\Omega;E)$ is $\A$-free if it is generated by a sequence of (asymptotically) $\A$-free measures. The set of such Young measures is denoted by $\Y^2_\A(\Omega;E)$. We establish the following rigidity properties of $\A$-free two-scale* Young measures:
	%\begin{definition} Let $1 < p < d/(d-1)$. A two-scale* Young measure $\bm{\nu} \in \Y^2(\Omega;E)$ is called $\A$-free if there exists a sequence $\{\mu_\eps\} \subset \M(\Omega;E)$ such that
	%	\[
	%	\|\A\mu_\eps\|_{\Wrm^{-k,p}(\Omega)} \to 0 \quad \text{and \quad $\mu_\eps \toYY \nu$}.
	%	\]
	%	The set of such Young measures is denoted by $\Y^2_\A(\Omega;E)$.
	%\end{definition}
	
	\begin{enumerate}
		\item The second-scale inherits the $\A$-free constraint. In Proposition~\ref{prop:secondbarycenterisafree} we show that the PDE-constraint is inherited also to the second-scale barycenter of any $\A$-free two-scale* Young measure (the 
		constraint holds trivially for the barycenter), that is, 
		\[
		\A\bbpr{\nu}_{x} = 0 \quad \text{in the sense of distributions on $Z$},
		\]
		for $(\Leb^d + \lambda^s)$-almost evert $x \in \cl \Omega$. The corresponding version of this result in terms of two-scale convergence is contained in Corollary~\ref{cor:Afree}.
		\item Based on the recent {developments~\cite{de-philippis2016on-the-structur} concerning} the structure of PDE-constrained measures, it further holds (see Corollary~\ref{cor:GuidoFilip}) 
		at $\lambda^s$-almost every $x \in \Omega$ that
		\[
		\frac{\dd \bbpr{\bm{\nu}}_{x}}{\dd |\bbpr{\bm{\nu}}_{x}^s|}(\xi) \in \Lambda_{\A} \quad \text{for $|\bbpr{\bm{\nu}}_{x}^s|$-almost every $\xi \in Z$},
		\]
		Here $\Lambda_\A$ is the so-called \emph{wave-cone} associated to $\A$ defined as
		\[
		\Lambda_{\A} \coloneqq \bigcup_{|\eta| =1} \ker \Abb^k(\eta) \subset E,
		\]
		and which consists of all Fourier amplitudes (vectors $z \in E$) where $\A$ is not elliptic with respect to one-directional oscillations.\footnote{Following standard notation, the $k$-homogeneous map
			\[
			\eta \mapsto \Abb^k(\eta) \coloneqq \sum_{|\alpha|=k} \eta^\alpha A_\alpha \in \textrm{Lin}(E,F), \qquad \eta \in \R^d,
			\]
			is the \emph{principal symbol associated to $\A$}. Using the Fourier transform it is immediate to verify that a vector $v \in E$ belongs to $\ker \Abb^k(\eta)$ --- for some $\eta \in \R^d \setminus \{0\}$ --- if and only if the one-directional function
			\[
			x \mapsto v \exp(2\pi\textrm{i}x \cdot \eta) \quad \text{is $\A$-free on $\R^d$}.
			\]}
		\item The support of the non-biting part of $\bm \nu$ (in the sense of Chacon; see Lemma~\ref{chacon}) is constrained by the differential constraint of $\A$. More specifically, in Lemma~\ref{lem:support} we show that
		\[
		\supp \, (\nu^\infty_{x, \xi}) \subset \spn \{\Lambda_{\A}\} \cap \partial \Bbb_E \quad \text{for $(\lambda^s \otimes \rho_x)$-almost every $(x,\xi) \in {\Omega} \times Z$}\,. 
		\] 
	\end{enumerate}
	%This results however remain insufficient to establish a chartacterization for $\A$-free two-scale* Young measures for PDO's
	%of higher order $k \ge 2$; with the exception of symmetric gradient measures which satsify a second-order PDE-constriant (see~\cite{developing}).    

	\subsubsection{Applications to homogenization} 
	
	%Given a measure $\mu$ in $\M(\Omega;\R^N)$, we write $\mu = \mu^{ac} \Leb^d \restrict \Omega  + \mu_s |\mu^s|$ to denote its Radon--Nykodym--Lebesgue decomposition. Here, as usual, $\mu^{ac}$ is an integrable function on $\Omega$ and $\mu_s$ is a $|\mu^s|$-measurable  
	
	We conclude our exposition solving a particular case of homogenization for PDE-constrained measures. We start by considering a family $\{I^\eps\}_{\eps >0}$ of functionals of the form
	\begin{equation}\label{eq:hom1}
	I^\eps(\mu) \coloneqq \int_\Omega f (x,x/\eps,\mu^{ac}(x)) \dd x 
	+ \int_\Omega f^\infty(x,x/\eps,\mu_s(x)) \dd |\mu^s|(x),
	\end{equation}
	defined for measures $\mu = \mu^{ac} \Leb^d \restrict \Omega + \mu_s |\mu^s|\in \M(\Omega;\R^N)$. 
	
	The integrand $f : \Omega \times \Tbb^d \times \R^N \to \R$ is assumed to be a continuous integrand with linear-growth at infinity. Further we will require that $f(x,\xi,\frarg)$ is convex for every $x,\xi \in \Omega \times \Tbb^d$. The candidate measures $\mu \in \M(\Omega;\R^N)$ are assumed to satisfy the PDE-constraint 
	\begin{equation}\label{eq:A}
	\A\mu  = 0 \quad \text{in the sense of distributions on $\Omega$}.
	\end{equation}
	%where the coefficients $A_\alpha \in \R^{n}\otimes \R^N$ are assumed to be constant and where we write $\partial^\alpha = \partial^{\alpha_1}_1 \dots \partial^{\alpha_d}_d$ for every multi-index $\alpha = (\alpha_1,\dots,\alpha_d) \in (\N \cup \{0\})^d$ with $\abs{\alpha} := \abs{\alpha_1} + \cdots + \abs{\alpha_d} \leq k$. The action of $\A$ on the space $\M(\Omega;\R^N)$ is to be understood in the sense of distributions, i.e.,
	%\[
	%\A \mu = \sum_{|\alpha| = k} A_\alpha \partial^\alpha \mu \in \Dcal'(\Omega;\R^n).
	%\]
	%In this context we say that a measure $\mu \in \Mcal(\Omega;\R^N)$ is \emph{$\Acal$-free}, or equivalently, $\mu \in \ker(\A)$ if and only if $\Acal \mu = 0$. 
	Our goal is to show that as $\eps \searrow 0$, the rapidly oscillating variable ${x}/{\eps}$ averages out and the functionals $I^\eps$ \emph{converge} (in the context of $\Gamma$-convergence, detailed below) to an \enquote{homogenized} integral  
	\begin{equation}\label{eq:hom2}
	I^\text{hom}(\mu) = \int_\Omega f_{*\Acal}(x,\mu^{ac}(x)) \dd x 
	+ \int_\Omega (f_{*\Acal})^\infty
	(x,\mu_s(x)) \dd |\mu^s|(x),
	\end{equation}
	where the integrand $f_{*\Acal}$ is characterized by means of the cell minimization problem
	%problem described below.
	%
	%The process of passing from~\eqref{eq:hom1} to~\eqref{eq:hom2} is know as \emph{homogenization}. Originally, homogenization provided a rigorous mathematical approach to \emph{composite materials} (see \cite{Lions,S-p}). Since Particularly, the case when the PDE-constraint~\eqref{eq:A} is satisfied for the operator $\A = \curl$. This problem has been widely studied for integrals defined on gradients, that is, when the PDE-constraint~\eqref{eq:A} is satisfied for the operator $\A = \curl$ (see \cite{Marcellini}). 
	%
	\begin{equation}\label{eq:defhom2}
	\begin{split}
	f_{*\A}(x,z) \coloneqq \inf\setBB{\int_{Q}& f(x,y,z + w(y)) \dd y}{ \\
		& w \in \Crm^\infty(Q;E) \cap \ker \A, \int_Q w(y) \dd y = 0}.
	\end{split}
	\end{equation}

	Let us briefly recall the notions of $\Gamma$-convergence and homogeneous envelope which will be required to give sense to our problem. 
	Let $\{\eps_j\}_j$ be a sequence of positive numbers converging to zero. The \emph{$\Gamma$-limit inferior} of the \emph{sequence} of functionals $\{I^{\eps_j}\}_j$ with respect to the weak-$*$ convergence of measures is defined as
	\[
	\Gamma-\liminf  I^{\eps_j}{(\mu)}\coloneqq \inf\setB{\liminf_{j \to \infty}I^{\eps_j}(\mu_j)}{\text{$\mu_j \toweakstar \mu$ in $\M(\Omega;\R^N)$}}.
	\]
	We say that {a functional }$I$ is the \emph{$\Gamma$-limit inferior} of the \emph{family} of functionals $\{I^\eps\}_{\eps > 0}$ if 
	\[
	I = \Gamma-\liminf_{j \to \infty} I^{\eps_j} \quad \text{for every sequence $\eps_j \searrow 0$}.
	\]
	In this case, we write
	\[
	I = \Gamma-\liminf_{\eps \searrow 0} I^\eps.
	\]
	
	%\subsubsection{A brief history of the problem}
	%
	%In \cite{}

	%The \emph{$\A$-free homogeneous envelope} of an integrand $h \in \Crm(\Tbb^d \times \R^N)$ is defined by the formula \begin{equation}\label{eq:defhom}
	%\begin{split}
	%h_{\Hom,\A}(z) \coloneqq \inf\setBB{& \frac{1}{R^d}\int_{Q_R}  h(y,z + w(y))  \dd y}{R \in \Nbb, \\
	%	& w \in  \Crm^\infty_\per(Q;E) \cap \ker \A= 0, \, \text{and}  
	%	\int_{Q} w(y) \dd y = 0}. 
	%\end{split}
	%\end{equation}
	%
	%
	%\begin{remark}\label{rem:cases}The integrands $f_{*\Acal}$(in~\eqref{eq:defhom2}) and $f_\Ahom$ (in~\eqref{eq:defhom})  coincide whenever the integrand is $z$-convex, that is, the functions $h(y,\frarg)$ are assumed to be convex. In the light of our assumptions, we shall restrict the definition of homogeneous envelope to~\eqref{eq:defhom}. 
	%\end{remark}
	
	The main homogenization result is contained in the following Theorem.
	
	\begin{theorem}[$\Gamma$-$\liminf$ of $I^\eps$]\label{thm:hom} Let $f : \Omega \times \Tbb^d \times E \to \R$ be a continuous integrand with linear-growth at infinity. Further assume that $f(x,\xi,\frarg)$ is convex for all $x \in \Omega$ and all $\xi \in \Tbb^d$. Then, the $\Gamma$-$\liminf$ of the family of functionals
		\begin{align*}
		\mu \mapsto I^\eps(\mu) =  
		\int_\Omega & f(x,x/{\eps},\mu^{ac}(x)) \dd x \\
		&  + \int_\Omega f^\infty(x,{x}/{\eps},\mu_s(x)) \dd |\mu^s|(x)\,, 
		\qquad \mu \in \M(\Omega;E) \cap \ker \A,
		\end{align*}
		with respect to the weak-$*$ convergence in $\M(\Omega;E)$, is given by the homogenized functional
		\[
		I^\textnormal{hom}(\mu) = \int_\Omega f_{*\Acal}(x,\mu^{ac}(x)) \dd x + \int_\Omega (f_{*\Acal})^\infty(x,\mu_s(x)) \dd |\mu^s|(x),
		\]
		defined for measures $\mu \in \M(\Omega;E) \cap \ker \A$.
	\end{theorem}

	\section{Preliminaries and Notation}\label{sec:preliminaries}
	%\mnote{Global notational ToDos: weak vs. weakly (closed/ measurable /convergent/ ...); $z=A$?.}
	%\mnote{A: what about $f^\sharp$ instead of $f^\#$ in order to clash the push-forward notation $T_\#$? Alternatively, we may define $T_*$}
	
	Here and in what follows $\Omega \subset \R^d$ is an open and bounded set with $\mathscr L^d(\partial \Omega) = 0$ (where $\mathscr L^d$ denotes the $d$-dimensional Lebesgue measure). 
	To avoid cumbersome definitions we shall simply write $Z$ to denote the $d$-dimensional torus $\Tbb^d$, and $Q$ to denote the closed $d$-dimensional unit cube $[0,1]^d$. To distinguish the $d$-dimensional Lebesgue measure between two locally $d$-dimensional euclidean spaces $E$ and $F$ we will often write
	$\Leb^d_E$ and $\Leb^d_F$ respectively. We denote the indicator function of a set $A$ by $\chi_A$. If $E,F$ are two Banach spaces, we denote by $\textrm{Lin}(E;F)$ the space of linear maps from $E$ to $F$.

	\subsection{Geometric measure theory} 
	Let $X$ be a locally convex space. We denote by $\Crm_c(X)$ the space of compactly supported and continuous functions on $X$, and by $\Crm_0(X)$ we denote its completion with respect to the $\|\frarg\|_\infty$ norm. 
	The space $\Crm_c(X)$ is not a complete normed space in the usual sense, however, it is a complete metric space as the inductive union of Banach spaces $\Crm_0(K_m)$ where $K_m \subset X$ are compact and $K_m \nearrow X$. 
	By the Riesz representation theorem, the space $\M(X)$ of bounded signed Radon measures on $X$ is the dual of $\Crm_0(X)$; a local argument of the same theorem states that the space $\M_\loc(X)$ of signed Radon measures on $X$ is the dual of $\Crm_c(X)$.
	%\Jnote{Here I am a little bit confused. The dual of $C_0$ are \emph{bounded} Radon measures while the Radon measures $\Mcal_\loc$ are not the dual to anything, right?\\
	%$\M_\loc \cong C_c^*$, you see this by localization}
	We notate by $\M_\loc^+(X)$ the subset of non-negative measures. 
	Since $\Crm_0(X)$ is a Banach space, the Banach--Alaoglu theorem and its characterizations hold and in particular bounded sets of $\M(X)$ are weak-$\ast$ metrizable. 
	On the other hand, the local compactness of $\Crm_c(X)$ permits the existence of a complete and separable metric on $\M_\loc(X)$ with the property that convergence with respect to that metric is equivalent to the weak-$\ast$ convergence in $\M_\loc(X)$  (see Remark 14.15 in \cite{mattila1995geometry-of-set}). In a similar manner, for a finite dimensional euclidean space $E$, $\M(X;E)$ and $\M_\loc(X;E)$ will denote the spaces of {$E$-valued} bounded Radon measures and {$E$-valued} Radon measures respectively.
	
	The space $\M(X)$ is a normed space endowed with the \emph{total variation} norm
	\[
	|\mu|(X;E) \coloneqq \sup\setBB{\int_X \varphi \dd \mu}{\phi \in \Crm_0(X;E), \|\phi\|_\infty \le 1}.
	\]
	The set of all positive Radon measures on $X$ with total variation equal to one is denoted by 
	\[
	\mathrm{Prob}(X) \coloneqq \setB{\nu \in \M^+(X)}{\nu(X) = 1}; 
	\]
	the set of \emph{probability measures} on $X$.%\Jnote{ I think we consequently use $\mathrm{Prob}$ in the following. I do also prefer $\mathrm{Prob}$.}
	
	%Let $F$ be a finite dimensional Euclidean space. 

	The \emph{push-forward} of a measure $\mu \in \M(\Omega;E)$, with respect to a Borel map $T : \Omega \to \Omega'$, is formally defined through the \emph{change of variables} formula \enquote{$T_\# \mu = \mu \circ T^{-1}$} as follows. For every $\phi \in \Crm(\Omega';E)$, we define the measure $T_\#\mu \in \M(\Omega;E)$ via
	\[
	\int_{\Omega'} \phi \dd(T_\#\mu) \coloneqq \int_\Omega \phi \circ T \dd \mu.
	\] 
	
	We define the linear action of a measure $\mu \in \M(X)$ on a function {$\phi \in \Crm_0(X)$}
	%\mnote{A: notice that we have to put $C_0$, the other pairing (for functions with linear growth, has to be defined more carefully through a compactification; we should check this.}  
	by the paring $\langle \phi,\mu \rangle =  \int_X \phi \dd \mu$.
	If $X = Z$ is the $d$-dimensional torus (which is a compact manifold), then any map {$g\in \Crm_c(Z) =  \Crm_0(Z)$}
	can be represented by a periodic and continuous function on the $d$-dimensional semi-closed unit cube; in this case we write  
	\[
	\dprb{ g,\mu} = \int_{Z}g \dd \mu \coloneqq \int_{[0,1[^d} g \dd \mu \quad \text{for $\mu \in \M(Z)$}.
	\]
	
	For a positive measure $\lambda \in \M(X)$ we write $\Lrm_\lambda(X;F)$ to denote the set (space provided that $F$ is a space) of $\lambda$-measurable functions with values on $F \subset E$. For $p \in [1,\infty]$ we write 
	\[
	\Lrm^p_\lambda(X;F) \coloneqq \setBB{f \in \Lrm(X;F)}{\int_X |f|_E^p \dd \lambda < \infty},
	\]
	the space of $\lambda$-measurable functions with values on $F \subset E$ that are $p$-integrable. We will also use the short-hand notations $\Lrm^p(X) \coloneqq \Lrm^p(X;\R)$ and 
	\[
	\Lrm_{\lambda,\loc}^p(X;F) = \setBB{f \in \Lrm(X;F)}{\text{$\int_K |f|_E^p \dd \lambda < \infty$ for all $K \Subset X$}}.
	\] 
	If $F$ is an euclidean space, Riesz' representation theorem tells us that every vector-valued measure $\mu \in \M_\loc(X;F)$ can be written as 
	\[
	\mu \; = \; f|\mu| \quad \text{for some $f \in \Lrm^\infty_{|\mu|,\loc}(\Omega;\Sbb_F)$}; %= \; \mu^{ac} \Leb^d \; + \; \mu^s \; = \; \mu^{ac} \Leb^d \; + \; \mu_s |\mu^s|
	\]
	this decomposition is often referred as the \emph{polar decomposition} of $\mu$. The set of points $x \in X$ where 
	\[
	 \lim_{r \todown 0} \int_{B_r(x_0)} |f(x) - f(x_0)| \dd |\mu|(x) = 0,
	\]
	is called the set of $|\mu|$-Lebesgue points; this set has full $|\mu|$-measure. Another resourceful representation of a measure is given by its \emph{Radon--Nykod\'ym--Lebesgue decomposition}
	\[
	\mu \; = \; \frac{\dd \mu}{\dd \Leb^d} \, \Leb^d \; + \; \frac{\dd \mu}{\dd |\mu^s|} |\mu^s| \; = \; \mu^{ac} \Leb^d \; + \; \mu_s |\mu^s|
	\] 
	where as usual $\mu^{ac} \in \Lrm^1_\loc(X;F)$, $|\mu^s| \perp \Leb^d$, and $\mu_s \in \Lrm_{|\mu^s|,\loc}(X;\Sbb_F)$. 
	
	Let $E,F$ be open or closed subsets of an euclidean finite dimensional space and let $\mu$ be a non-negative Radon measure on $E$. 
	A map $\nu : E \mapsto \M(F) : x \mapsto \nu_x$ is said to be \emph{weak-$*$ $\mu$-measurable} if the map $x \mapsto \nu_x(B)$ is $\mu$-measurable for all Borel sets $B \in \Bfrak(F)$. 
	A simple method to check the $\mu$-measurability of such a measure valued map $x \mapsto \nu_x$ is to test the $\mu$-measurability of the map $x \mapsto \int_F g(x,z) \dd \nu_x(z)$ for every $\Bfrak(E) \times \Bfrak(F)$-measurable function $g : E\times F \to \R$. %\mnote{A: is this enough, or one has to ask boundedness?; memo to check in APF}
	The set of all weak-$*$ $\mu$-measurable maps $x \mapsto \nu_x$ endowed with the norm $|\nu|_\infty(E) \coloneqq \esssup_{(E,\mu)} |\nu_x|(F)$ conforms a Banach space which will shall denote by $\Lrm^\infty_{\mu,\star}(E;\M(F))$. In the particular case that $\mu = \mathscr L^d$ we shall simply write $\Lrm^\infty_{\star}(E;\M(F))$.
	
	Given a weak-$*$ $\mu$-measurable map $\nu \in \Lrm^\infty_{\mu,\star}(E;\M(F))$ we can define the {generalized product} of $\mu$ and $\nu_x$ which is the measure taking the values 
	\[
	(\mu \otimes \nu_x )(U) \coloneqq \int_E \int_F \chi_U(x,z) \dd \nu_x(z) \dd \mu(x), \qquad \text{for all $U \in \Bfrak(E \times F)$}.
	\]
	The main step in the construction of Young measures (and multi-scale Young measures) relies on a well-known disintegration result for measures $\lambda \in \M^+(E \times F)$.
	Merely, it establishes a condition under which $\lambda = \mu \otimes \nu_x$, where $\mu$ is the push-forward of $\lambda$ under the projection onto $E$ and $\nu_x$ is a weak-$*$ $\mu$-measurable map of probability measures.
	This is recorded in the next theorem (for a proof see Theorem 2.28 in \cite{ambrosio2000functions-of-bo}).

	\begin{theorem}[disintegration]\label{thm:desintegration} 
		Let $\lambda \in \M^+(E \times F)$ and let $\pi : E \times F \to E$ be the projection on the first factor. 
		Assume that the push-forward measure $\mu \coloneqq \pi_\#\lambda \in \M(E)$ is a finite Radon measure. 
		Then there exists a weak-$*$ $\mu$-measurable map $\nu \in \Lrm^\infty_{\mu,\star}(E;\M(F))$, uniquely defined up to equivalence classes, such that $\lambda = \mu \otimes \nu_x$. 
		Moreover,  %$\nu \in \Lrm^\infty_{\star}(E,\mu;\mathrm{Prob}(F))$, that is, 
		\[
		\nu_x\in \mathrm{Prob}(F) \quad \text{for $\mu$-almost every~$x$ in $E$}.
		\]
	\end{theorem}
	
	We close this section by introducing the notion of \emph{probability tangent measure} as introduced in \cite[Sec. 2.7]{ambrosio2000functions-of-bo}. Let $\{r_j\}_{j \in \Nbb}$ be an infinitesimal sequence of positive real numbers ($r_j \searrow 0$). A \emph{local blow-up sequence} of a measure $\mu \in \M(\Omega;\R^N)$ at a point $x_0 \in \Omega$ is a sequence of (normalized) measures of the form 
	\[
	\tau_j = \frac{1}{|\mu|(\cl{Q_{r_j}(x_0)})^{-1}}\,{T^{x_0,r_j}_\# \mu} \in \M(\cl Q), 
	\]
	which are well-defined provided that $|\mu|(Q_{r_j}(x_0)) > 0$. A weak-$*$ limit $\tau \in \M(\cl{Q})$ of a local blow-up sequence is said to be a {local tangent measure}. We write
	\[
	\tau \in \Tan_1(\mu,x_0) \subset \mathrm{Prob}(\cl{Q}),
	\]
	to denote the set of all probability tangent measures (\emph{tangent measures} for short). \footnote{The term \enquote{tangent measure} refers to a more general object than the one referred to here. For the purposes of this paper, in particular the development of the localization principles in Section~\ref{sec:5}, it is technically more convenient to work with local tangent measures. However, our conclusions are compatible with the more general notion of \emph{tangent measure} introduced by Preiss \cite{preiss1987geometry-of-mea}.} At $|\mu|$-almost every $x_0 \in \Omega$, all tangent measures of $\mu$ at $x_0$ are constant multiples of a positive measure. More precisely, if $\mu =  f|\mu|$ is the polar decomposition of $\mu$, then at every $|\mu|$-Lebesgue point $x_0 \in \Omega$ it holds 
	\[
	\Tan_1(\mu,x_0) = f(x_0) \Tan_1(|\mu|,x_0).
	\]
	%we shall refer to such points as $\mu$-\textbf{Lebesgue} points.
	In particular, every tangent measure $\tau \in \Tan_1(\mu,x_0)$ can be written as
	\[
	\tau = \frac{\dd \mu}{\dd |\mu|}(x_0)|\tau|.
	\]

	\subsection{Integrands}\label{integrands}
	
	Let $n$ be a non-negative integer. Throughout this section and the rest of the paper we will consider continuous integrands $f : \Omega \times Z^{n-1} \times E \to \R$ (recall that $\Omega \subset \R^d$ is an open and bounded set, $Z = \Tbb^d$ is the $d$-dimensional torus, and $E$ is a euclidean space of finite dimension endowed with a norm $|\frarg|=|\frarg|_E$). Elements of $Z^{n-1}$ shall be denoted by $\bm{\xi} = (\xi_1,\dots,\xi_{n-1})$.
	
	Consider the transformation
	\[
	(Tf)(x,\bm{\xi},\hat z) \coloneqq (1 - |\hat z|)\cdot f\left(x,\bm{\xi},\frac{\hat z}{1 - |\hat z|}\right), \quad \text{$(x,\bm{\xi},\hat z) \in \Omega \times Z^{n-1} \times \Bbb_E$},
	\]
	defined for $f \in \Crm(\Omega \times Z \times \R^N)$ where $\Bbb_E$ is the open unit ball in $E$. The vector space
	\begin{align*}
	\E^n(\Omega;E) \coloneqq \setB{&f \in \Crm(\Omega \times Z^{n-1} \times E)}{\text{$(Tf)$ extends} \\& \qquad  \text{to a continuous function in $\Crm\big(\cl{\Omega \times Z^{n-1} \times \Bbb_E}\big)$}}
	\end{align*}
	endowed with the norm $\| f \|_{\E^n(\Omega;E)} \coloneqq \| Tf \|_{\infty}$, $\E^n(\Omega ;E)$ is a Banach space. Moreover, $T$ is a \emph{compactification} in the sense that
	\[
	T : \E^n(\Omega ;E) \to \Crm\big(\cl{\Omega \times Z^{n-1} \times \Bbb^E}\big)
	\]
	is an isometry of spaces with inverse 
	\[
	(Sg)(x,\bm{\xi},z) \coloneqq (1 + |z|)\cdot g\left(x,\bm{\xi},\frac{z}{1 + |z|}\right), \quad \text{$(x,\bm{\xi}, z) \in \Omega \times Z^{n-1} \times E$}. 
	\] 
	The map $T : \E^n(\Omega;E) \to \Crm\big(\cl{\Omega \times Z^{n-1} \times \Bbb^E}\big)$ induces, through duality, an isomorphism $T^*$ (with inverse $S^*$) of the dual spaces; $T^*$ and $S^*$ are also isometries.
	%\Jnote{ In the TeX Code there is a commented diagram too emphasize this. I have no idea if you want this in or out :)}
	%\begin{remark}
	%Since $S^*$ is a bijective isometry, it follows that 
	Hence, a subset $X \subset (\E^n(\Omega;E))^*$ is sequentially weak-$\ast$ closed if and only if its image under $S^*$,
	$S^*X \subset \M\big(\cl{\Omega \times Z^{n-1} \times \Bbb^E}\big)$, is sequentially weak-$\ast$ closed in the sense of measures.

	Every $f \in \E^n(\Omega;E)$ has linear-growth at infinity, meaning there exists a positive constant $M$ (in this case given by $M = \|Tf\|_\infty$) for which $|f(x,\bm{\xi},z)| \le M(1 + |z|)$ for all $(x,\bm{\xi}, z) \in \Omega \times Z^{n-1} \times E$. This allows one to define a regularization at infinity: for $f\in \E^n(\Omega;E)$, we define the \emph{strong recession function} of $f$ as the limit
	\[
	f^\infty(x,\bm{\xi},z) \coloneqq \lim_{\substack{x' \to x\\ \bm{\xi}' \to \bm{\xi}\\t \to \infty}}\frac{f(x',\bm{\xi}',tz)}{t},
	\qquad \text{for $(x,\bm{\xi}, z ) \in \cl\Omega \times Z^{n-1}\times E$.}
	\]
	%This limit is well-defined as a result of 
	The continuity of $Tf$ ensures the limit is well-defined, %up to $\cl{\Omega \times Z \times \Bbb_E}$ 
	and, in fact, $Tf(x,\bm{\xi},\hat z) = f^\infty(x,\bm{\xi},\hat z)$ for all $(x,\bm{\xi},\hat z) \in \cl \Omega \times Z^{n-1} \times \partial\Bbb_E$. Observe that $f^\infty$ is positively one-homogeneous in the $z$-variable and hence it can be recovered from the extended values of $Tf$. 
	%Moreover, $T$ is a \emph{compactification} in the sense that
	%\[
	%T : \E^n(\Omega ;\R^N) \to \Crm\big(\cl{\Omega \times Z^{n-1} \times \Bbb_E}\big)
	%\]
	%is an isometry of spaces with inverse 
	%\[
	%(Sg)(x,\bm{\xi},z) \coloneqq (1 + |z|)\,g\left(x,\bm{\xi},\frac{z}{1 + |z|}\right), 
	%\quad \text{$ (x,\bm{\xi}, z) \in \Omega \times Z^{n-1} \times \Rbb^N$}. 
	%\] 
	
	% %%%%%%%%%%%%%%%%%%%%%COMMUTATIVE%%%%%%%%%%%%%%%%%%%%%%%
	% \[
	% \begin{tikzcd}
	% %%\E^n(\Omega;\R^N)   \arrow[r, shift left, swap, "T"'] \arrow[d,"\ast"]
	% %%& \Crm\big(\cl{\Omega \times Z^{n-1} \times \Bbb_E}\big) \arrow[l, shift left, swap, "S"'] \arrow[d,"\ast"] \\
	% \E^n(\Omega;\R^N)^* \arrow[r, shift right, swap, "S^*"]
	% & \M\big(\cl{\Omega \times Z^{n-1} \times \Bbb_E}\big) \arrow[l, shift right, swap, "T^*"] 
	% \end{tikzcd}
	% \]

	%%%%%%%%%%%%%%%%%%%%%%COMMUTATIVE%%%%%%%%%%%%%%%%%%%%%%%

	%We may similarly define the space $\E^n(\Omega;\R^N)$, where ${\xi} \in (Z)^{n-1}$ takes the role of $\xi \in Z$. 
	%\begin{remark}
	%If $n = 1$, then only the \textit{macroscopic} scale is quantified. 
	%The space $\E^1(\Omega;\R^N)$ coincides with the dual of generalized Young measures, $\E(\Omega, \R^N)$ (cf. \cite{KR10a}).\Jnote{Actually here we should also not use the word 'introduced' right?\\
	%A: corrected}
	%\end{remark}
	%Indeed, let $f \in  \E(\Omega \times Z;\R^N)$ and $L \in \big(\E(\Omega \times Z;\R^N)\big)^*$, then
	%\[
	%\dprb{f,L}_{\E \times \E*} = \dprb{Sg,L}_{\E \times \E*} = \dprb{g,S^*L}_{\Crm \times \M}
	%\]
	%Since $\|f\|_{\E} = \|g\|_\infty$ it follows that $\|L\|_{\E^*} = \|S^*L\|_{\M}$. 
	
	To complement our notation, we also define  
	%For $h\in C^0(\R^N)$ we define, with a slight abuse of notation, $Th(\hat z) \coloneqq (1 - |\hat z|)h\left(\frac{\hat z}{1 - |\hat z|}\right)$, and
	\[
	\E^0(E) \coloneqq  \setB{h \in \Crm(E)}{\text{$h^\infty$ exists}}
	\]
	{where $h^\infty$ is the strong recession function above without $x$ or $\xi$ dependance.}
	%which is a subset of the functions on $\R^N$ with linear growth. 
	Observe that if $h \in \E^0(E)$, then the integrand $(\phi \otimes g \otimes h)(x,\bm{\xi},z) \coloneqq \phi(x)g(\bm{\xi})h(z)$
	%$\phi \otimes h \in \E^1(\Omega;\R^N)$ and 
	belongs to the space $\E^n(\Omega;E)$ whenever $\phi \in \Crm(\cl \Omega)$ and $g \in \Crm(Z^{n-1})$.
	
	A function $g : Z^{n-1} \to \R$ is called \emph{upper (lower) semicontinuous} if its $Q$-periodic extension to $(\R^d)^{n-1}$ is  upper (lower) semicontinuous. 
	Under this convention, the function $\xi \mapsto \xi$ with $\xi \in (0,1]$ is \textit{not} lower semicontinuous on $Z$ while the function $\xi \mapsto \xi$ for $\xi \in [0,1)$ \textit{is}.
	\begin{figure}[h]
		\captionsetup{justification=raggedright}  
		\centering
		\begin{tikzpicture}
		\draw[->] (0,-0.2)--(0,1.2);
		\draw(-1,0)--(2,0); 
		%  \node at (0.7,-0.2) [anchor=north]{\tiny{The function $\xi \mapsto \xi$ with $\xi \in (0,1]$}};
		\draw(-1,0)--(0,1);
		\draw(0,0) -- (1,1);
		\draw(1,0) -- (2,1);
		\draw (0,0) circle (1.5pt);
		\draw (-1,0) circle (1.5pt);
		\draw (1,0) circle (1.5pt);
		\filldraw(0,1) circle (1.5pt);
		\filldraw(1,1) circle (1.5pt);
		\filldraw(2,1) circle (1.5pt);
		\end{tikzpicture}
		\hspace{2cm}
		\begin{tikzpicture}
		\draw[->] (0,-0.2)--(0,1.2);
		\draw(-1,0)--(2,0); 
		%  \node at (0.7,-0.2) [anchor=north]{\tiny{The function $\xi \mapsto \xi$ with $\xi \in [0,1)$}};
		\draw(-1,0)--(0,1);
		\draw(0,0) -- (1,1);
		\draw(1,0) -- (2,1);
		\draw (0,1) circle (1.5pt);
		\draw (1,1) circle (1.5pt);
		\draw (2,1) circle (1.5pt);
		\filldraw(-1,0) circle (1.5pt);
		\filldraw(0,0) circle (1.5pt);
		\filldraw(1,0) circle (1.5pt);
		\end{tikzpicture}
		
		\caption{Lower semicontinuity of periodic functions, the functions $\xi \mapsto \xi$ for $\xi \in (0,1]$ and $\xi \in [0,1)$. } \label{fig:Lscofperiodic}
	\end{figure}
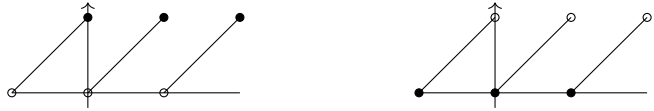
	In the following we say that a {function} $f:\Omega \times Z^{n-1} \times E \to \R$ is upper (lower) semicontinuous if $f(\frarg,\xi,\frarg)$ is upper (lower) semicontinuous on $\Omega \times E$ for all $\bm{\xi} \in Z^{n-1}$ and $f(x,\frarg,z)$ is upper (lower) semicontinuous on $Z^{n-1}$.
	
	In general, the strong recession function of a Borel (or even continuous) integrand  might fail to exist. Instead, one can always define the \emph{upper} and \emph{lower recession} functions by setting
	\[
	f^\#(x,\bm{\xi},z) \coloneqq \limsup_{\substack{x' \to x\\ \bm{\xi}' \to \bm{\xi}\\t \to \infty}}\frac{f(x',\bm{\xi}',tz)}{t},
	\qquad \text{for $(x,\bm{\xi}, z ) \in \cl\Omega \times Z^{n-1}\times E$}
	\]
	and
	\[
	f_\#(x,\bm{\xi},z) \coloneqq \liminf_{\substack{x' \to x\\ \bm{\xi}' \to \bm{\xi}\\t \to \infty}}\frac{f(x',\bm{\xi}',tz)}{t},
	\qquad \text{for $(x,\bm{\xi}, z ) \in \cl\Omega \times Z^{n-1}\times E$}.
	\]
	In the same way as $f^\infty$, $f^\#$ and $f_\#$ are positively $1$-homogeneous on the $z$-variable. Observe that $f^\#$ (respectively $f_\#$) is nothing else than the upper (lower) semicontinuous regularization of $Tf$ on $\cl{\Omega \times Z^{n-1} \times \Bbb_E}$. This observation is the key argument behind the following approximation result 
	%The following tells us that lower semicontinuous integrands and their upper (lower) recession functions can be (pointwise) approximated  from below by elements of $\E^2(\Omega;E)$; %Its proof relies on the fact that for such integrands $f$, the lower semicontinuous regularization of $Tf$ coincides with $f_\infty$ on $\partial \Sbb^{N-1}$; 
	(its proof follows the same arguments given in the proof \cite[Lemma~2.3]{alibert1997non-uniform-int} in the context of classical generalized Young measures).

	%Upper semicontinuous functions can be approximated from below by continuous functions. 
	%The next lemma tells us that a upper semicontinuous integrand $f : \Omega\times Z^{n-1} \times \R^N \to \R$ can be approximated from below by elements of $\E^n(\Omega;\R^N)$ (\textcolor{blue}{see also [REF],[REF],[REF]}). 
	\begin{proposition}\label{prop:aprox}
		%\mnote{A:do we have this proposition twice (also stated as a Lemma ahead)?} 
		Let $f : \Omega \times Z^{n-1} \times E \to \R$ be an upper semicontinuous integrand and assume there exists $M>0$ such that $f(x,\xi,z) \geq -M (1+|z|)$ for all $(x,\xi,z) \in \Omega\times Z\times E$.
		Then, there exists a non-increasing sequence of functions $(f_m)_{m \in \Nbb} \subset \E^n(\Omega;E)$ such that
		\begin{align*}
		\inf_{m > 0} f_m = \lim_{m \to \infty} f_m & = f \qquad  \textrm{pointwise on $\;\Omega \times Z^{n-1} \times E$, and} \\
		\inf_{m > 0} f_m^\infty = \lim_{m \to \infty} f_m^\infty & = f^\#  \qquad \textrm{pointwise on $\;\cl \Omega \times Z^{n-1} \times E$}.
		\end{align*}
	\end{proposition}
	%\begin{proof}
	%The proof follows by the same ideas as Lemma XXX of [] \mnote{Johannes should check this, but he didn`t yet}
	%\end{proof}
	\begin{remark} The analogous statement holds for lower semicontinuous integrands $f$ and their lower recession function $f_\#$ {for a non-decreasing sequence of approximating functions.}
	\end{remark}
	
	\section{Generalized multi-scale Young measures}\label{sec:3}
	
	Our construction extends the exposition in \cite{kristensen2010characterizatio} which itself goes back to the seminal works of DiPerna \& Majda \cite{diperna1987oscillations-an} and Alibert \& Bouchitt\'e \cite{alibert1997non-uniform-int}.
	
	Heuristically, $n$ will represent the number of hierarchical scales of the vector
	\[
	\underbrace{(x,\xi_1,\dots,\xi_{n-1})}_{n}
	\]
	where $x \in \Omega$ is the \emph{macroscopic} variable and the $\xi_i$'s conform a hierarchical family of periodic \emph{microscopic} scales.
	The indexing corresponding to the microscopic scales reflects a disassociation between the $i$-th scale and the finer $(i + 1)$-th scale. Mathematically, this is reflected by the homogenization with respect to scales
	\[
	\bigg(x,\underbrace{\frac{x}{\eps_1}}_{\xi_1},\dots,\underbrace{\frac{x}{\eps_{n-1}}}_{\xi_{n-1}}\bigg) 
	\]
	where we assume that $\eps_{i +1} \ll \eps_i$ for each $i = \{1,\dots,n-2\}$.
	We are now ready to introduce the notion of multi-scale Young measure.
	
	%In the following $E$ will be a finite dimensional euclidean space endowed with 
	%the norm $|\frarg|_E$. For example, we may consider $E = \R^N$ the canonical $N$-dimensional real vector space, or $E = \Mbb_{d \times d}^\sym(\R)$ the space of $d \times d$ symmetric matrices 
	%with real coefficients. 
	
	We are now ready to give the precise definition of multi-scale* Young measure and state their main properties. 
	
	\begin{definition}[$n$-scale* Young measure]
		A four-tuple $ \bm{\nu} =(\nu, \lambda, \rho, \nu^\infty)$ %parameterized by $x \in \cl \Omega$ and  ${\xi} \in Z^{n-1}$, 
		is called a \textit{generalized $n$-scale* Young measure} on $\cl \Omega$ with values on $E$ provided that  
		\begin{enumerate}
			\item[\textnormal{(i)}] $\lambda$ is a positive measure on $\cl \Omega$,
			\item[\textnormal{(ii)}] $\rho$ is a weak-$*$ $\lambda$-measurable map from $\cl \Omega$ into the set $\mathrm{Prob}(Z^{n-1})$ of probability measures over the product of $(n-1)$ copies of the $d$-dimensional torus, and
			\item[\textnormal{(iii)}] the map $(x,{\xi}) \mapsto \dpr{|\frarg|,\nu_{x,{\xi}}}$ belongs to  $\Lrm^1(\Omega \times Z^{n-1})$.
		\end{enumerate}
		Additionally, $\nu$ and $\nu^\infty$ are weak- and weak-$*$ measurable maps on $\cl \Omega \times Z$ respectively:
		\begin{enumerate}
			\item[\textnormal{(iv)}] $\nu$ is a weak measurable map from $\cl \Omega \times Z^{n-1}$ into the set $\mathrm{Prob}(E)$ of probability measures over $E$, and
			\item[\textnormal{(v)}] $\nu^\infty$ is a weak-$*$ $(\lambda \otimes \rho_x)$-measurable map from $\cl \Omega \times Z^{n-1}$ into the set $\mathrm{Prob}(\partial \Bbb_E)$ of probability measures supported on the unit sphere in $E$. 
		\end{enumerate}
		The set of all $n$-scale Young measures is denoted by $\Y^n(\Omega;E)$. 
		%  Here, $\{\nu_{x,{\xi}}\}_{x,{\xi}}$ is called the \emph{homogenized oscillating part}, $\{\nu_{x,{\xi}}^\infty\}_{x,{\xi}}$ is called the \emph{homogenized angle of concentration}, $\{\rho^\nu_x\}_x$ is the \emph{singular homogenization measure}, and $\lambda^\nu$ is the \emph{mass concentration density}.
		%We also define the set of locally bounded , we $n$-scale (generalized) Young measures as follows. 
		%	An $n$-scale* Young measure $\bm{\nu}$ is called a \emph{locally bounded $n$-scale* Young measure} if
		%	\[
		%	\bm{\nu}|_{U \times Z^{n-1}} \in \Y^n(U;E) \quad \text{for all open $U \Subset \Omega$}\,,
		%	\]
		%	in this case we write $\bm{\nu} \in \Y_\loc^n(\Omega;E)$. \textcolor{blue}{ we discussed this on the phone, somethings different needs to be done here :p}
	\end{definition}

	%Observe that for arbitrary $\phi \in \Crm(\cl\Omega)$ and $g\in \Crm(\Tbb^d)$, $h\in \E^0(\R^N)$ we define $(\phi \otimes g \otimes h)(x,y,z) = \phi (x) g(y) h(z)$ and it directly follows $\phi \otimes g \otimes h \in \E^2(\Omega;E)$.
	\subsection{Construction}

	We shall restrict our analysis to two-scale* Young measures  $\Y^2(\Omega;E)$. The results and ideas behind the proofs extend analogously to $n$-scale* Young measures.
	%We say that a quadruple $\nu = \big(\nu_{x,\theta},\lambda^\nu, \rho_x ^\nu,\nu_{x,\theta}^\infty\big)_{x \in \cl \Omega, \theta \in Z}$ is \emph{generalized two-scale Young measure} if 
	%\begin{align*}
	%&(\nu_{x,\theta})_{x,y} \in \Lrm^\infty_{w^*}\big(\,\Omega \times Z\,;\,\mathrm{Prob}(\R^N)\big), &g_\nu: (x,y) \mapsto \dprb{|\frarg|,\nu_{x,\theta}} \in \Lrm^1(\Omega \times Z),\\
	%&\lambda_\nu \in \M^+(\cl \Omega),  &(\rho_x) \in \Lrm^\infty_{\lambda_\nu,w^*}\big(\,\cl\Omega\,;\,\mathrm{Prob}(Z)\big),\\
	%& (\nu^\infty_{x,\theta}) \in \Lrm^\infty_{\lambda_\nu \otimes \rho_x,w^*}\big(\,\cl\Omega \times Z\,;\,\mathrm{Prob}(\Sbb^{N-1})\big). &
	%\end{align*}
	%The set of all (generalized) two-scale Young measures shall be denoted by $\Y^{2}(\Omega ; \R^N)$. 
	
	Two-scale* Young measures conform a set of dual objects to the space of integrands $\E^2(\Omega;E)$ in the following way:
	For $f \in \E^2(\Omega;E)$ and $\bm{\nu} \in \Y^2(\Omega,E)$, we define a bilinear product by setting
	\begin{equation}\label{eq:Yproduct}
	\begin{split}
	\ddprb{f, \bm{\nu}} & \coloneqq \int_{\Omega} \bigg(\int_{Z} \dprb{f(x,\xi,\frarg),\nu_{x,\xi}} \dd \xi \bigg) \dd x \\
	& \; \qquad 
	+ \int_{\cl\Omega}\bigg(\int_{Z} \dprb{f^\infty(x,\xi,\frarg),\nu_{x,\xi}^\infty}\dd \rho^\nu_x(\xi)\bigg) \dd \lambda^\nu (x).
	\end{split}
	\end{equation}
	By the definition of two-scale* Young measure and the  the linear-growth of the elements of $\E^2(\Omega;E)$ we can estimate its norm by\begin{align*}
	|\ddprb{f,\bm{\nu}}| & \le \|Tf\|_\infty\bigg(\int_{\Omega \times Z} \dprb{1 + |\frarg|,\nu_{x,\xi}} \dd \xi \dd x + \int_{\cl \Omega} \dd \lambda^\nu(x)\bigg) \\
	& =  \|Tf\|_\infty\big(\Leb^d(\Omega) + \|g^\nu\|_{\Lrm^1(\Omega\times Z)} + \lambda^\nu(\cl \Omega)\big), 
	%& \coloneqq \big(\Leb^d(\Omega) + \|g_\nu\|_{\Lrm^1(\Omega)} + (\lambda_\nu)(\cl \Omega)\big)\|Tf\|_\infty,
	\end{align*}
	where we used the short-hand notation $g^\nu(x,\xi) \coloneqq \dprb{|\frarg|,\nu_{x,\xi}}$.
	
	It follows that $\ddprb{\frarg,\bm{\nu}} \in \E^2(\Omega;E)^*$, and, in this sense, we shall identify $\Y^2(\Omega;E)$ with a subset of $\E^2(\Omega;E)^*$. For $\bm{\nu}_j,\bm{\nu} \in \Y^{{2}}(\Omega;E)$ we write $\bm{\nu}_j \toweakstar \bm{\nu}$ in $\E^2(\Omega;E)$ if
	\[
	\ddprb{f,\bm{\nu}_j} \to \ddprb{f,\bm{\nu}} \quad \text{for all $f \in \E^2(\Omega;E)$}.
	\] 
	In this case we say that $\bm{\nu}_j$ weak-$\ast$ converges to $\bm{\nu}$ as two-scale* Young measures.
	
	The following weak-$*$ semicontinuity results hold. 
	%\textcolor{blue}{Some comments: 1) Again, at this point of the paper, maybe before the lemma, i would write a short comment that this now applies on a larger space of functions f. 2) It shoutld be '(jointly) semicontinuous' in either both statements or none, and I'm not quite sure what it means ;) 3) It should be a $f^\#$ in the second statement, not a $f^\infty$, right? 4) It is correct that I put the $\nu$s in bold letters, right?}
	\begin{lemma}[semicontinuity] Let $f:\Omega \times Z \times E \to \R$ be a lower semicontinuous integrand with linear-growth at infinity. Then, the functional
		\begin{align}
		\bm{\nu} \mapsto \ddprb{f,\bm{\nu}}_\# \coloneqq \int_{\Omega} \bigg(& \int_Z \dprb{f(x,{\mathbf \xi},\frarg),\nu_{x,{\mathbf \xi}}} \dd \xi \bigg) \dd x \\
		&+\int_{\cl \Omega} \bigg(\int_Z \dprb{f_\#(x,\xi,\frarg),\nu^\infty_{x,\xi}} \dd \rho_x(\xi)\bigg) \dd \lambda(x)
		\end{align}
		is weak-$*$ lower semicontinuous in $\Y^2(\Omega;E)$. 
		
		Similarly, if $f:\Omega \times Z \times E \to \R$ is a Borel integrand that is upper semicontinuous and has linear growth at infinity, then the functional
		\begin{align}
		\bm{\nu} \mapsto \ddprb{f,\bm{\nu}}^\# \coloneqq \int_{\Omega} \bigg(& \int_Z \dprb{f(x,{\mathbf \xi},\frarg),\nu_{x,{\mathbf \xi}}} \dd \xi \bigg) \dd x  \\
		& + \int_{\cl \Omega} \bigg(\int_Z \dprb{f^\#(x,\xi,\frarg),\nu^\infty_{x,\xi}} \dd \rho_x(\xi)\bigg) \dd \lambda(x)
		\end{align}
		is weak-$*$ upper semicontinuous in $\Y^2(\Omega;E)$. 
	\end{lemma}
	\begin{proof} Since the two statements are equivalent modulo taking $(-f)$ in place of $f$, we shall only argue the case of upper-semicontinuity. Let $f : \Omega \times Z \times E \to \R$ be an upper-semicontinuous integrand and let $\bm{\nu}_j \toweakstar \bm{\nu}$ in $\Y^2(\Omega;E)$. Further let $(f_m)_{m\in\Nbb} \subset \E^2(\Omega;E)$ be the monotone approximating sequence provided by Proposition~\ref{prop:aprox}. Fix $m \in \Nbb$, then by continuity of the pairing and the properties of the approximating sequence we get
		\begin{align*}
		\limsup_{j \to \infty} \, \ddprb{f,\bm{\nu}_j}^\# & \le \lim_{j \to \infty} \,  \ddprb{f_m, \bm{\nu}_j}  \\
		& =  \ddprb{f_m,\bm{\nu}}.
		\end{align*}
		The conclusion then follows by letting $m \to \infty$ in the inequality above and arguing with the monotone convergence theorem.
	\end{proof}
	An immediate consequence of this result is that the weak-$*$ continuity of the map $\bm{\nu} \mapsto \ddprb{f,\bm{\nu}}$ can be extended from $\E^2(\Omega;E)$ to the  larger class 
	\[
	\Rbf^2(\Omega;E) \coloneqq \setB{f \in \Crm(\Omega \times Z \times E)}{f^\infty \equiv f^\# \equiv f_\#}\,
	\]
	of continuous integrands possessing a \emph{strong recession function}. In particular $z$-convex integrands with linear growth at infinity belong to this class:
	
	\begin{proposition} Let $f : \Omega \times Z \times E \to \R$ be a continuous integrand with linear growth at infinity. If $f(x,\xi,\frarg)$ is convex for every $(x,\xi) \in \Omega\times Z$, then $f \in \Rbf^2(\Omega;E)$.
	\end{proposition}
	\begin{proof}
		The proof follows from Remark~2.4 in \cite{alibert1997non-uniform-int}.
	\end{proof}
	
	%{\color{blue}It will be often resourceful to interpret the pairing $\ddpr{\frarg,\frarg}$ as a measure over $\cl \Omega$. Let us first introduce some additional notation: for $f \in \E^2(\Omega;E)$ and $\nu \in \Y^2(\Omega;E)$ consider the measure (defined through its values in $\Bfrak(\cl \Omega)$)
	%\begin{align*}
	% [  f, \bm{\nu} ] (U) & \coloneqq \int_{U}\bigg(\int_{Z} \dprb{f(x,\xi,\frarg),\nu_{x,\xi}} \dd \xi \bigg) \dd x \\
	%& \qquad + \int_{U}\bigg(\int_{Z} \dprb{f^\infty(x,\xi,\frarg),\nu_{x,\xi}^\infty}\dd \rho^\nu_x(\xi)\bigg) \dd \lambda^\nu (x), \quad U \in \Bfrak(\cl \Omega).
	%\end{align*}}
	If $\bm{\nu} \in \Y^2(\Omega;E)$, duality yields %to a paring (still denoted by the same symbol) $\ddprb{\frarg,\frarg} : \Crm\big(\cl{\Omega \times Z \times \Bbb_E}\big) \times S^*(\Y^2(\Omega;E))$ defined by
	\[
	\dprb{g,S^*\bm{\nu}}
	= \ddprb{Sg,\bm{\nu}}\quad \text{for all $g \in \Crm\big(\cl{\Omega \times Z \times \Bbb_E}\big)$}.
	\]
	In other words, the following  diagram 
	\begin{equation}\label{diagram}
	\begin{tikzcd}
	\Y^2(\Omega;E)   \arrow[r, hook] \arrow[d,"S^*"]
	& \E^2(\Omega;E)^* %\arrow[l, shift left, swap, "S"']
	\arrow[d,"S^*"] \\
	S^*(\Y^2(\Omega;E)) \arrow[r,hook]
	& \M(\cl{\Omega\times Z \times \Bbb_E}) 
	%\subset \M\big(\cl{\Omega \times (Z)^{n-1} \times \Bbb_E}\big)
	\end{tikzcd}
	\end{equation}
	is commutative. 
	Likewise, every $\mu = S^*\bm{\nu} \in S^*(\Y^2(\Omega;E))$ verifies the identity
	\[
	\dprb{Tf,\mu} = \ddprb{f,\bm{\nu}} \quad \text{for all $f \in \E^2(\Omega;E)$}.
	\]
	Using that $S^*$ is an isometry of Banach spaces, it can be deduced the isomorphism lowers to a weak-$*$ isomorphism  
	$S^*:(\E^2(\Omega;E)^*,\text{w-$*$}) \to (\M(\cl{\Omega \times Z \times \Bbb_E}),\text{w-$*$})$. 
	It is then straightforward  to check that 
	\[
	S^*\bm{\nu}_j \toweakstar \mu \quad \text{in $\M(\cl{\Omega \times Z \times \Bbb_E}) \quad \Leftrightarrow \quad \bm{\nu}_j \toweakstar T^*\mu$ \quad in $\E^2(\Omega;E)$}.
	\] 
	\begin{remark}[topological isomorphism]\label{rem:iso} Topological properties are preserved under isomorphisms, in particular
		%Recall that bounded subsets of the dual of a separable space are metrizable with respect to the weak-$\ast$ topology. Thus, by the Banach--Alaoglu Theorem,
		\begin{gather*}
		S^*[\Y^2(\Omega;E)] \; \text{is sequentially weak-$\ast$ closed in $\M(\cl{\Omega \times Z \times \Bbb_E})$} \\
		\Leftrightarrow\\
		\Y^2(\Omega;E) \; \text{is sequentially weak-$\ast$ closed in $\E^2(\Omega;E)^*$}\,.\\
		%\Leftrightarrow \\
		%\text{each bounded subset of $\Y^{2}(\Omega;\R^N)$ is weak-$\ast$ relatively compact}.
		\end{gather*}
	\end{remark}
	
	%The next lemma tells us the first topological property stated above holds.
	
	The following characterization will play a fundamental role in proving the weak-$*$ compactness of uniformly bounded sets of two-scale* Young measures. We follow the presentation given in Lemma 2 in \cite{kristensen2010characterizatio} for $\Y^1(\Omega;E)$.    
	\begin{lemma}\label{lem:representationformula}
		The set $S^*[\Y^2(\Omega;E)]$ consists precisely of all positive measures $\mu \in \mathcal{M}^+(\cl{\Omega \times Z \times \Bbb_E})$ satisfying following property: for all $\phi \in \Crm(\Omega)$ and all $g \in \Crm(Z)$ it holds that
		%\[
		%(\phi \otimes g)\cdot(1 - |\hat z|) \mu = \phi\Leb^d \otimes g \Leb_Z^d 
		%\]
		%holds.
		\begin{align}\label{eq:representation}
		\int_{\cl{\Omega \times Z\times \Bbb_E}} \, (\phi \otimes g)(x,\xi)\,(1-|\hat z|) \dd\mu (x,\xi,\hat z) = \int_\Omega \phi(x) \dd x \int_{Z}g(\xi) \dd \xi\,.
		\end{align}
		
	\end{lemma}
	\begin{proof}
		\proofstep{Necessity.} Fix $\mu = S^*\bm{\nu}$. 
		Applying $\mu$ on the function $Tf = \phi \otimes g \otimes (1 -|\frarg|)$ and using that $f = \phi \otimes g \otimes \chi_{E}$ (and hence $f^\infty \equiv 0$) we get
		\[
		\dprb{ Tf, \mu } = \dprb{ Tf, S^*\bm{\nu} } = \ddprb{ f, \bm{\nu} } =\int_\Omega \int_{Z} \phi(x) g(\xi) \dd \xi \dd x,
		\]
		which is (\ref{eq:representation}).
		The positivity of $\mu$ follows from  the positivity of $\nu$ and $\nu^\infty$ and the definition of $S$.
		
		\proofstep{Sufficiency.} For the reverse statement, fix a measure $\mu \in \mathcal{M}^+(\cl{\Omega \times Z \times \Bbb_E})$ satisfying~(\ref{eq:representation}). We want to find $\bm{\nu} \in \Y^2(\Omega;E)$ satisfying  $\mu = S^*\bm{\nu}$.  %, that is, 
		%\begin{align*}
		%\dpr{g,\mu} & = \ddprb{Sg,\bm{\nu}} \\ 
		%& = \int_\Omega \int_Z \dpr{Sg(x,\xi,\frarg),\nu_{x,\xi}} \dd \xi \dd x\, + \int_{\cl \Omega} \int_Z \dpr{g(x,\xi,\frarg),\nu_{x,\xi}^\infty} \dd \rho_x(\xi) \dd \lambda(x)\,.
		%\end{align*}
		By disintegration (see Theorem~\ref{thm:desintegration}), we may decompose $\mu$ as a double semi-product $\mu = \hat \lambda \otimes \hat \rho_x \otimes \hat \nu_{x,\xi}$ 
		where $\hat\lambda \in \M^+(\cl\Omega)$, $\hat\rho \in \Lrm^\infty_{\hat\lambda,\star}(\cl\Omega;\mathrm{Prob}(Z))$ and $\hat \nu \in \Lrm^\infty_{\hat\lambda \otimes \hat\rho_x,\star}(\cl\Omega\times Z;\mathrm{Prob}( \cl{\Bbb_E}))$. %such that $\mu = \lambda \otimes \rho_x \otimes\nu_{x,\xi}$. 
		Let us further set 
		\[
		u(x,\xi) \coloneqq \dpr{1 - |\frarg|,\hat\nu_{x,\xi}} \in \Lrm^{1}_{\hat\lambda \otimes \hat\rho_x }(\cl \Omega \times Z)\,.
		\]
		Thus, we may re-write~\eqref{eq:representation} as the equivalence $u \, (\hat \lambda \otimes \hat \rho_x) \equiv \Leb^d_\Omega \otimes \Leb^d_Z$. %and the arbitrariness of $\phi\in \Crm(\cl \Omega)$ tha
		%\,,
		% \]  
		%%\[
		%u \, (\hat\lambda \otimes \hat\rho_x) \equiv \Leb^d_\Omega \otimes \Leb^d_Z \quad \text{in $\M^+(\cl {\Omega} \times Z)$}.
		%\]
		On the one hand, this gives $u \, [(\hat \lambda^s \otimes \hat\rho_x) + (\Leb^d_\Omega \otimes \hat \rho_x^s)] \equiv 0$, which, in turn, is equivalent to 
		\begin{equation}\label{eq:tec1}
		\supp\,(\hat \nu_{x,\xi}) \subset \partial \Bbb_E \quad \text{$[(\hat \lambda^s \otimes \hat\rho_x) + (\Leb^d_\Omega \otimes \hat \rho_x^s)]$-almost everywhere}.
		\end{equation}
		On the other hand, the same equivalence yields 
		\begin{equation}\label{eq:tec2}
		\lambda^{ac}(x) \cdot \rho_x^{ac}(\xi) \cdot  u(x,\xi) =  1 \quad \text{$\Leb^d_\Omega \otimes \Leb^d_Z$-almost everywhere}.
		\end{equation}
		In particular $u$ is $\Leb^d_\Omega \otimes \Leb^d_Z$-measurable.
		
		The goal now is to exhibit a $\bm{\nu} = (\nu,\lambda,\rho,\nu^\infty) \in \Y^2(\Omega;E)$ for which $S^*\bm{\nu}= \mu$. 

		\textbf{Construction of $\bm{\nu}$.} Let us being by defining  the weak-$*$ $(\Leb^d_\Omega \otimes \Leb^d_Z)$-measurable map 
		\[
		(x,\xi) \mapsto \nu_{x,\xi}\in \mathrm{Prob}(E)\,,
		\]
		where each $\nu_{x,\xi}$ is the probability measure satisfying
		\[
		\dpr{\nu_{x,\xi},h} \coloneqq \frac{1}{u(x,\xi)}  \int_{\Bbb_E}Th(\hat z) \dd\hat \nu_{x,\xi}(\hat z) \qquad \text{for all $h \in \Crm(\beta E)$}\,.
		\]
		That $\nu$ is a weak-$*$ measurable map follows from the measurability of $u$ and the properties of $\hat \nu \restrict \partial \Bbb_E$ in terms of the measure $(\hat \lambda \otimes \hat \rho_x)$. To check that each $\nu_{x, \xi}$ is indeed a probability measure ($\Leb^d_\Omega \otimes \Leb^d_Z$-almost everywhere) follows by testing with $\chi_E$ and using the definition of $u$ (recall that $T\chi_E \equiv 1 - |\frarg |$ as functions on $\Bbb_E$). Moreover, $\dpr{1 + |\frarg|,\nu_{x,\xi}} = u(x,\xi)^{-1} \hat\nu_{x, \xi}(\Bbb_E)$ and hence by~\eqref{eq:tec2} we infer the map $(x,\xi) \mapsto \dpr{1 + |\frarg|,\nu_{x,\xi}}$ is integrable on $\Omega \times Z$.

		We define the remaining  $\lambda,\rho$ and $\nu^\infty$ as follows. First, we set  
		\[
		\lambda \coloneqq \bigg(\int_Z \hat\nu_{x,\xi}(\partial \Bbb_E) \dd \rho_x(\xi)\bigg) \, \hat \lambda \in \M^+(\cl \Omega)\,.
		\]
		Once this positive measure has been defined, we define a map $\rho$ from $\cl \Omega$ into the set $\mathrm{Prob}(Z)$ of probability measures over the $d$-dimensional torus by setting  
		\[
		x \mapsto \rho_x \coloneqq \bigg(\int_Z \hat\nu_{x,\xi}(\partial \Bbb_E)\dd\hat\rho_x(\xi)\bigg)^{-1} m_x\,\hat\rho_x\,.
		\]
		where we have used the short-hand notation $m_x(\xi) \coloneqq \hat\nu_{x,\xi}(\partial \Bbb_E)$. Since by definition $m_x$ is a $\hat \rho_x$-measurable map, we infer that $\rho$ is
		weak-$*$ $\lambda$-measurable.    %and define 
		%\[
		%x \mapsto \rho_x \coloneqq \bigg(\int_Z \hat\nu_{x,\xi}(\partial \Bbb_E)\dd\hat\rho_x(\xi)\bigg)^{-1} m_x\,\hat\rho_x\,.
		%\]
		%By definition $\rho$ is a weak-$*$ $\lambda$-measurable map from $\cl \Omega$ into the set $\mathrm{Prob}(Z)$ of probability measures over the $d$-dimensional torus. 
		Lastly, we define a map $\nu^\infty$ from $\cl \Omega \times Z$ into $\mathrm{Prob}(\partial \Bbb_E)$ by setting
		\[
		(x,\xi) \mapsto \nu_{x,\xi}^\infty \in \mathrm{Prob}(\partial \Bbb_E)\,,
		\]
		where each $\nu_{x,\xi}^\infty$ is given (in terms of duality) by
		\[
		\dpr{h,\nu_{x,\xi}} \coloneqq \dashint_{\partial \Bbb_E}h(\hat z) \dd \hat \nu_{x,\xi}(\hat z) \qquad \text{for all $h \in \Crm(\partial \Bbb_E)$}\,.
		\]
		That $\nu^\infty$ is a weak-$*$ $(\lambda \otimes \rho_x)$-measurable map is then a consequence of the weak-$*$ measurability of $\hat \nu$, and the way $\lambda$ and $\rho$ are defined. 
		Altogether these properties imply  $\bm{\nu} \coloneqq (\nu,\lambda,\rho,\nu^\infty) \in \Y^2(\Omega;E)$. 
		
		\textbf{Pre-image property $(S^* \bm{\nu} = \mu$).} Let $f \in \E^2(\Omega;E)$ be a fixed but arbitrary integrand, later we shall exploit this choice through duality. By construction of $\bm{\nu}$ we get that
		\begin{align*}
		\ddprb{f,\bm{\nu}} & = \int_{\Omega}\int_Z \dpr{f(x,\xi,\frarg),\nu_{x,\xi}} \dd \xi \dd x + \int_{\cl \Omega}\int_Z\dpr{f^\infty(x,\xi,\frarg),\nu^\infty_{x,\xi}} \dd\rho_x(\xi) \dd \lambda(x)\\
		& = \int_{\Omega}\bigg(\int_Z\bigg(\int_{\Bbb_E}Tf(x,\xi,z)\dd\hat\nu_{x,\xi}(z)\bigg)\,\hat\rho_x^{ac}(\xi)\dd\xi\bigg)\,\hat\lambda^{ac}(x) \dd x \\
		& \quad + \int_{\Omega}\bigg(\int_Z\bigg(\dashint_{\partial \Bbb_E} Tf(x,\xi,z)\dd\hat\nu_{x,\xi}(z)\bigg)\,m_x(\xi)\,\hat\rho^{ac}_x(\xi) \dd\xi\bigg)\,\hat\lambda^{ac}(x) \dd x \\
		& \qquad + \int_{\Omega}\bigg(\int_Z\bigg(\int_{\partial \Bbb_E} Tf(x,\xi,z)\dd\hat\nu_{x,\xi}(z)\bigg)\dd\hat\rho_x^s(\xi)\bigg)\,\hat\lambda^{ac}(x) \dd x \\
		& \quad\qquad + \int_{\cl\Omega}\bigg(\int_Z\bigg(\int_{\partial \Bbb_E} Tf(x,\xi,z)\dd\hat\nu_{x,\xi}(z)\bigg)\,\dd\hat\rho_x(\xi)\bigg)\,\hat\lambda^{s}(x)\,,
		\end{align*}
		where in passing to the last equality we used that 
		$f^\infty(x,\xi,\frarg) \equiv Tf(x,\xi,\frarg)$ as functions over $\partial \Bbb_E$. Furthermore, since $m_x(\xi) = \hat\nu_{x,\xi}(\partial \Bbb_E)$, the first two lines of the last equality above add up to
		\[
		\int_{\cl{\Omega \times Z \times \Bbb_E}} Tf \dd(\hat\lambda^{ac}\Leb^d_\Omega\otimes\hat\rho_x^{ac}\Leb^d_Z \otimes\hat\nu_{x,\xi})\,.
		\]
		On the other hand, using~\eqref{eq:tec1} we may re-write the last two lines in the expression of $\ddprb{f,\bm{\nu}}$ as
		\[
		\int_{\cl{\Omega \times Z \times \Bbb_E}} Tf  \dd(\hat\lambda^{ac}\Leb^d_\Omega\otimes\hat\rho_x^s\otimes\hat\nu_{x,\xi})
		+ \int_{\cl{\Omega \times Z \times \Bbb_E}} Tf  \dd(\hat\lambda^s\otimes\hat\rho_x\otimes\hat\nu_{x,\xi})\,.
		\]
		Regrouping these three summands together we deduce 
		\[
		\ddprb{f,\bm{\nu}}= \int_{\cl{\Omega \times Z \times \Bbb_E}} Tf \dd\mu = \dpr{Tf,\mu} \qquad \text{for all $f \in \E^2(\Omega;E)$}\,.
		\]
		Equivalently,  by a duality argument,
		\[
		\dprb{\Phi,\mu} = \ddprb{S\Phi,\bm{\nu}} = \dprb{\Phi,S^*\bm{\nu}} \quad \text{for all $\Phi \in \Crm(\cl{\Omega \times Z \times \Bbb_E})$}\,.
		\]
		Thence $\mu \equiv S^*\bm{\nu}$ as measures in $\cl{\Omega \times Z \times \Bbb_E}$. This proves the sufficiency.
	\end{proof}

	A direct consequence of this characterization and Remark~\ref{rem:iso} is the following fundamental property of Young measures. Here and in what follows, we shall write 
	\[
	{[\frarg]_{E}} \coloneqq \chi_{\cl \Omega} \otimes \chi_Z \otimes (1 + |\frarg|_E).
	\]
	
	%\begin{proposition} $\Y^2(\Omega;E)$ is closed under the weak-$*$ convergence of Young measures.
	%\end{proposition}
	%\begin{proof}
	%By the characterization of the previous lemma, we infer the set $S^*(\Y^2(\Omega;E))$ is sequentially weak-$\ast$ closed in $\M(\cl{\Omega \times Z \times \Bbb_E})$. The assertion then follows from Remark~\ref{rem:iso}.
	%\end{proof}

	\begin{proposition}[compactness of two-scale Young measures]\label{cor:Compactnessofyoungmeasures} The set of two-scale* Young measures $\Y^2(\Omega;E)$ is sequentially weak-$\ast$ closed in $\E^2(\Omega;E)^*$. 
		Moreover, each subset $\mathcal Y \subset \Y^2(\Omega;E)$ satisfying
		% \mnote{J: Notice: i) and ii) hold iff $\ddprb{\chi_\Omega \otimes \chi_Z\otimes |\id_{R^N}| ,\bm{nu}} \leq C$. This seems not only a compacter way to write it, but also easyier to check in proofs, right?}
		%\begin{enumerate}
		%\item[\textnormal{(i)}] $\sup_{\bm{\nu} \in X} \lambda^\nu(\overline{\Omega}) < \infty$, and
		%\item[\rm{(ii)}] the family of functions 
		%\[
		%\Big\{(x,\xi) \mapsto \dprb{|\frarg|,\nu_{x,\xi}}\Big\}_{\bm{\nu} \in X}
		%\] 
		%is uniformly bounded in $\Lrm^1(\Omega \times Z)$,
		%\end{enumerate}
		\[
		\sup \; \setBB{\int_{\cl \Omega}\int_Z \dpr{1 + |\frarg|,\nu_{x, \xi}} \dd \xi \dd x + \lambda(\cl \Omega)}{\bm{\nu} = (\nu,\lambda,\rho,\nu^\infty) \in \Ycal} < \infty
		\]
		%\begin{enumerate}[(i)]
		%\item $\sup_{u \in \Ycal}ï¿œ\|g^\nu\|_{\Lrm^1(\Omega \times Z)} < \infty$,
		%\item $\sup \lambda^\nu(\cl\Omega) < \infty$,
		%\end{enumerate}
		is pre-compact with respect to the relative weak-$\ast$ topology on $\Y^2(\Omega;E)\subset \E^2(\Omega;E)^*$.
	\end{proposition}
	\begin{proof}
		To verify that $\Y^2(\Omega;E)$ is sequentially weak-$*$ closed it suffices (by Remark~\ref{rem:iso} and Lemma~\ref{lem:representationformula}) to observe that~\eqref{eq:representation} is a closed property with respect to the sequential weak-$*$ convergence of measures in $\M^+(\cl \Omega \times Z \times \Bbb_E)$. Let $\Ycal$ as in the assumptions and observe that $T(1 + |\frarg|) =  \chi_{\Bbb_E}$ so that
		\[
		\sup_{\mu \in S^*[\Ycal]} \, |\mu|(\cl \Omega\times Z \times \Bbb_E) = \sup_{\bm{\nu} \in \Ycal} \, \ddprb{{[\frarg]_{E}},\bm{\nu}} < \infty. 
		\] 
		The Banach--Alaoglu theorem tells us that $S^*[\Ycal]$ is pre-compact with respect to the weak-$*$ topology of measures. Hence, again by Remark~\ref{rem:iso}, $\Ycal$ is 
		weak-$*$ pre-compact with respect to the relative weak-$*$ topology of $\E^2(\Omega;E)^*$. 
	\end{proof}
	
	We close this section with an important separability property. The proof of this result follows from a straightforward adaptation of the arguments given in the proof of Lemma~3 in \cite{kristensen2010characterizatio}.
	\begin{lemma} \label{lem:density}
		There exists a countable family of non-negative integrands $\{\phi_m \otimes g_m\otimes h_m\}_{m\in \N} \subset \E^2(\Omega;E)$
		that separates $\Y^2(\Omega;E)$. That is, if $\bm{\nu}, \bm{\sigma}\in \Y^2(\Omega;E)$, then
		\[
		\ddprb{ f_m, \bm{\nu}} = \ddprb{ f_m, \bm{\sigma}} \;\; \text{for all $m \in \N$ \quad $\Rightarrow$ \quad $\bm{\nu} \equiv \bm{\sigma}$ \, in \, $\Y^2(\Omega;E)$}.
		\]
		Moreover the family can be chosen so that each $h_m$ is uniformly Lipschitz on $E$ and each $g_m$ is uniformly continuously differentiable.
	\end{lemma}

	\subsection{Generating sequences}
	
	Vector-valued Radon measures can be naturally identified with a Young measure via the (compact) embedding 
	\[
	\M(\Omega;E) \embed \Y^2(\Omega;E) : \mu \mapsto \bm{\nu}_\mu \coloneqq (\delta_{\mu^{ac}},|\mu^s|,\Leb^d_Z,\delta_{\mu_s})\,.
	\]
	%\[
	%\delta[\mu]_{x} = \delta_{\frac{\dd \mu}{\dd \Leb^d}(x)}, \quad \lambda_{\delta[\mu]} = |\mu^s|, \quad \delta[\mu]_{x}^\infty = \delta_{\frac{\dd \mu}{\dd |\mu^s|}(x)}.
	%\] 
	%\[
	%(\Delta\mu)_{x,\xi} = \delta_{\mu^{ac}(x)}, \quad \lambda^{\Delta\mu} = |\mu^s|, \quad \rho^{\Delta\mu}_x = \Leb^d_{Z}, \quad (\Delta\mu)_{x,\xi}^\infty = \delta_{{\mu_s}(x)}.
	%\] 
	%This identification can be extended to an embedding $\delta_2 : \M(\Omega;\R^N) \cembed \Y^2(\Omega;E)$ by composition with the canonical embedding
	%$\Y^1(\Omega;\R^N) \embed \Y^2(\Omega;E) : \vartheta \mapsto \nu$ given by
	%\[
	%\nu_{x,\xi} = \vartheta_x, \quad \lambda^\nu = \lambda^\vartheta, \quad \rho^\nu_x = \Leb^d_{Z}, \quad \nu_{x,\xi}^\infty = \vartheta_x^\infty.
	%\]
	%We get
	%\[
	%\Delta_2[\mu]_{x,\xi} = \Delta_{\frac{\dd \mu}{\dd \Leb^d}(x)}, \quad \lambda^{\delta_2[\mu]} = |\mu^s|, \quad \rho^{\Delta_2[\mu]}_x = \Leb^d_{Z}, \quad \Delta_2[\mu]_{x,\xi}^\infty = \Delta_{\frac{\dd \mu}{\dd |\mu^s|}(x)}.
	%\] 
	Given an integrand $f \in \Rbf^2(\Omega;E)$ and a positive real $\eps$, we may define a functional 
	\[
	I^\eps_f(\mu) \coloneqq \int_{\Omega} f(x,x/\eps,\mu^{ac}(x)) \dd x  
	+  \int_{\Omega } f^\infty(x,x/\eps,\mu_s(x)) \dd |\mu^s|(x)\,,
	\] 
	defined on vector-valued Radon measures $\mu \in \M(\Omega;\R^N)$. 
	\begin{definition}[generating sequence]
		Let $\eps \searrow 0$ be an infinitesimal sequence of positive real numbers. We say that a sequence $(\mu_\eps)_\eps \subset \M(\Omega;\R^N)$ \emph{generates} the two-scale* Young measure $\bm{\nu} \in \Y^2(\Omega;E)$ if and only if 
		\begin{align*}
		I^\eps_f(\mu_\eps)  \to  \ddprb{f,\nu} \quad \text{for all $f \in \E^2(\Omega;E)$\,.}
		\end{align*}
		In the case the domain of convergence is understood we simply write $\mu_\eps \toYY \bm{\nu}$.
	\end{definition}

	\subsection{Proof of Theorem~\ref{thm:representation}}
	We are now ready to give the proof of Theorem~\ref{thm:representation} which asserts that every uniformly bounded sequence of measures generates (up to a subsequence) a two-scale* Young measure. It is worthwhile to mention the proof is not an immediate consequence of the compactness of Young measures (Corollary~\ref{cor:Compactnessofyoungmeasures}) since the microscopic variable \enquote{$x/\eps$} does not appear in the bi-linear pairing $\ddprb{\frarg,\frarg}$. Instead, the argument relies on the careful inspection of the limiting two-scale* Young measures, Proposition~\ref{lem:representationformula}, and the topological equivalence of Remark~\ref{rem:iso}.
	%\begin{theorem}[compactness]\label{thm:classicalcompactness}
	%Let $\eps \searrow 0$ be an infintesimal sequence of positve real numbers and let $(\mu_\eps)_\eps$ be a uniformly bounded sequence in $\M(\Omega;E)$. 
	%Then, there exist a subsequence $(\mu_{\eps_k})_{k \in \Nbb}$ and a two-scale* Young measure $\bm{\nu} \in \Y^2(\Omega;E)$  %= (\nu_{x,\xi} ,\nu^\infty_{x,y}, \rho_x^\nu, \lambda_\nu)_{x\in \cl \Omega, y\in Z}$ 
	% such that 
	% \[
	% \mu_{\eps_k} \toYY \bm{\nu}\qquad \text{as \; $k \to \infty$}.
	% \]
	% \end{theorem}
	\begin{proof} In the context of the notation introduced above, we may re-formulate the statement of Theorem~\ref{thm:representation}  as follows: let $(\mu_\eps)\subset \M(\Omega;E)$ a sequence of measures with uniformly bounded variation. Then, there exists  a two-scale* Young-measure $\bm{\nu} \in \Y^2(\Omega;E)$ satsifying (up to a subsequence)
		\begin{equation}
		\label{eq:goal1}
		I_f^\eps(\mu_\eps) \to \ddprb{f,\bm{\nu}} \quad \text{for all $f \in \E^2(\Omega;E)$}.
		\end{equation}
		
		% We want to show that up to taking a subsequence there exists a two-scale Young measure $\bm{\nu} \in \Y^2(\Omega;E)$ satsifying
		%\begin{equation}\label{eq:goal1}
		%I_f^\eps(\mu_\eps) \to \ddprb{f,\bm{\nu}} \quad \text{for all $f \in \E^2(\Omega;E)$}.
		%\end{equation}
		
		\proofstep{Step~1.}  Let $\delta > 0$. For an integrand $\Phi \in \Crm(\cl{\Omega \times Z \times \Bbb_E})$ define a continuous linear functional by setting
		\[
		L_\delta(\Phi)  \coloneqq I^\delta_{(S\Phi)}(\mu)\,. %\int_\Omega (S\Phi)(\,x\,,\,x/\eps\,,\,\mu_\eps^{ac}(x)\,) \dd x \;  + \; \int_{\cl \Omega} \Phi (\,x\,,\,x/\eps\,,\, (\mu_\eps)_s(x)\,) \dd |\mu_\eps^s|(x)\,,
		\]
		By the definition of the map $f \mapsto I^\delta_f$ with $\eps = \delta$ we obtain the estimate $| \langle L_\eps, \Phi \rangle | \leq \|\Phi\|_{\infty}( \Leb^d(\Omega)+ \|\mu_\eps\|)$ for every $\eps > 0$ of the infinitesimal sequence. The Banach--Alaoglu theorem and the Riesz representation theorem then yield the existence of a subsequence $(\varepsilon_k)_{k \in \Nbb}$ 
		and a measure $\mu_L\in \M\big(\cl{\Omega \times Z \times \Bbb_E}\big)$ satisfying 
		\[
		L_{\varepsilon_k} \toweakstar \dpr{\mu_L,\frarg} \quad \text{as functionals in $\Crm(\cl{\Omega \times Z \times \Bbb_E})^*$}.
		\] 
		%That is, 
		% \begin{align}\label{eq:cpctnessofmeasures}
		%\dpr{L_\eps,\phi} \rightarrow  \int_{\cl{\Omega \times Z \times \Bbb_E}}\phi(x,\xi,z) \dd L(x,\xi,z).
		% \end{align}
		%We follow-up with a disintegration argument:
		Testing this convergence with an integrand of the form $\Phi =  \phi \otimes g \otimes (1 - |\frarg|)$ %for arbitrary $\Phi \in \Crm(\cl\Omega)$ and
		and using that $S\Phi = \phi \otimes g \otimes \chi_{E}$ we further deduce
		\[
		L_\eps(\Phi) = \int_{\Omega} \phi(x) g(x/\eps_k) \dd x  \to  \dpr{L,\Phi}\,.
		\]
		This gives
		\[
		\int_{\cl{\Omega \times Z \times \Bbb_E}} \phi(x)g(\xi)(1 - |\hat z|) \dd L(x,\xi,\hat z) = \int_{\Omega}\phi(x)\dd x \int_Z  g(\xi) \dd \xi
		\]
		for all $\phi \in \Crm(\cl \Omega)$ and $g \in \Crm(Z)$. 
		We apply Lemma~\ref{lem:representationformula} to the measure $L$, to find a two-scale* Young measure  $\bm{\nu} \in \Y^2(\Omega;E)$ with $S^*\bm{\nu} = L$. 
		
		\proofstep{Step~2.} Using the commutative diagram~\eqref{diagram} and the identity {$S \circ T = \id_{\E^2(\Omega;E)}$} we conclude that 
		(recall that  $Tf(x,\xi,\frarg) \equiv f^\infty(x,\xi,\frarg)$ as functions over $\partial \Bbb_E$)
		\[
		I^\eps_f(\mu_\eps) = L_\eps(Tf) \to \dpr{L,Tf} = \dpr{S^*\bm{\nu},Tf} = \ddprb{\bm{\nu},f} \qquad \text{for all $f \in \E^2(\Omega;E)$}\,.
		\]
		This proves~\eqref{eq:goal1}. \end{proof}

	\subsection{Barycenter measures and two-scale convergence}
	
	We now turn to the concept of two-scale convergence. Following~\cite{nguetseng1989a-general-conve} we extend this notion to the two-scale convergence of measures as follows.
	
	\begin{definition}[two-scale convergence]\label{def:tsc} Let $\eps \searrow 0$ be a sequence of infinitesimal real numbers. Let also $\kappa \in \M^+(\cl \Omega)$ be a positive measure, and $\theta$ be a {weak-$*$} $\kappa$-measurable map from $\cl \Omega$ into $\M(Z;E)$.  We say that the sequence of measures $(\mu_\eps)_\eps \subset \M(\Omega;E)$ two-scale converges (as $\eps \todown 0$) to the generalized product measure $\mu = \kappa \otimes \theta_x$ if and only if %, for all $\Psi \in \Crm(\cl \Omega \times Z)$, we have
		%\[
		%g\bigg( \frac{\frarg}{\eps}\bigg) \mu_\eps \otimes \delta_{\floor{x}}  \toweakstar \mu
		%\]
		\[
		\int_\Omega \Psi\bigg(x,\frac{x}{\eps}\bigg)\dd \mu_\eps(x)  \to \int_{\cl \Omega}\bigg(\int_{Z} \Psi(x,\xi) \dd \theta_x(\xi) \bigg) \dd \kappa(x) \quad \text{for all $\Psi \in \Crm(\cl \Omega \times Z)$}.
		\]
		%To denote that $\mu_\eps$ two-scales converges to $\lambda \otimes \nu_x$ we write $\mu_\eps \twoheadrightarrow \lambda \otimes \nu_x$. \more[check std notation]% in $\Omega$.
		%If moreover $\mu_\eps \toweakstar \mu$ in $\M(\cl \Omega;\R^N)$, then
	\end{definition}
	%\begin{remark} If $\mu_\eps$ two-scale converges to $\lambda \otimes \nu_x$ on $\Omega$, then automatically $\mu_\eps \toweakstar \lambda$ as measures in $\M(\cl\Omega;\R^N)$.
	%\end{remark}
	%\begin{example}
	%The two-scale limit can be easily identified in following situations:
	%\begin{itemize}
	%\item If $u_\eps \to u$ in $\Lrm^1(\Omega;\R^N)$, then $u_\eps$ two-scale converges to $u \Leb^d \otimes \Leb^d_{Z}$. 
	%\item More generally, if $\mu_\eps \to \mu$ in measure, that is, $|\mu_\eps - \mu|(\Omega) \to 0$, then, up to taking a subsequence, $\mu_\eps$ two-scale converges to 
	%$\mu \otimes \sigma_x$ where $\sigma_x \in \Tan(\mu,x)$ for $|\mu|$-a.e. $x \in \Omega$.
	%\end{itemize}
	%\end{example}
	This {limit} concept is linked to the notion of \emph{barycenter} and \emph{second-scale barycenter} of a two-scale* Young measure (defined below) which will be significant for our techniques.
	%The concept of \emph{barycenter} and \emph{second-scale barycenter} of a two-scale Young measure will be significant for our techniques. Its properties are linked to the notions of weak-$*$ convergence and \emph{two-scale convergence};  discussed next.
	
	\begin{definition}[barycenter]\label{def:bary}
		Let $\bm{\nu}  = (\nu,\lambda,\rho,\nu^\infty) \in \Y^2(\Omega;E)$. We define the \emph{barycenter} of $\bm{\nu}$ to be the $E$-valued  measure in $\M (\cl\Omega; E)$ defined as %$[\bm{\nu}]\in \M (\cl\Omega; \R^N)$ which is uniquely determined through its values on Borel subsets $U \subset \cl \Omega$ as
		\begin{align*}
		[\bm{\nu}] \coloneqq \bigg(\int_{Z} \dprb{ \id_{E},\nu_{x,\xi} } \dd \xi \bigg) \Leb^d_\Omega  +  \bigg( \int_{Z} \dprb{ \id_{\partial \Bbb_E},\nu^\infty_{x,\xi} } \dd \rho_x(\xi)\bigg) \lambda\,.
		\end{align*}
		%Notice that if $\mu_\eps \toYY \nu$ and $\mu_\eps \toweakstar \mu$, then $\mu_\eps \toweakstar [\nu]$ as measures on $\cl \Omega$.
		%We may alternatively write $[\bm{\nu}](U) = \ddprb{\chi_U \otimes \chi_Z \otimes \id_{\R^N}, \bm{\nu} }$.
	\end{definition}
	Let $\bm{\nu} = (\nu,\lambda,\rho,\nu^\infty) \in \Y^2(\Omega;E)$ be fixed. By the Radon--Nikodym decomposition theorem, applied to the measure $\lambda$, there exist a partition of $\cl \Omega = \reg_{\bm{\nu}}(\Omega) \cup \sing_{\bm{\nu}}(\Omega)$ by subsets satisfying %$\reg_{\bm{\nu}}(\Omega)$ and  $\sing_{\bm{\nu}}(\Omega)$ of  $\cl \Omega$ satsifying %$\Leb^d(\sing_{\bm{\nu}}) = \lambda^s(\reg_{\bm{\nu}}(\Omega)) = 0$ and $\cl \Omega = \reg_{\bm{\nu}}(\Omega) \cup \sing_{\bm{\nu}}(\Omega)$.
	%
	%We split $\Omega$ into the sets 
	%\[
	%S_{\bm{\nu}} \coloneqq \setBB{x \in \Omega}{\frac{\dd \Leb^d}{\dd \lambda^\nu}(x) = 0} \cup \partial \Omega,
	%\]
	%and
	%\[
	%R_{\bm\nu} \coloneqq \setBB{x \in \Omega}{\frac{\dd \lambda^\nu}{\dd \Leb^d}(x) = 0}.
	%\]
	%Note that 
	\[
	\Leb^d(\sing_{\bm{\nu}}(\Omega)) = \lambda^s(\reg_{\bm{\nu}}(\Omega)) = 0\,.
	\]
	Moreover, 
	\[
	\sing_{\bm{\nu}}(\Omega) = \setBB{x \in \cl \Omega}{\frac{\dd \Leb^d}{\dd |\lambda^s|}(x) = 0}.
	\]
	We are now in position to give the notion of second-scale barycenter.
	
	\begin{definition}[second-scale barycenter]\label{def:bary2}
		Let $\bm{\nu} = (\nu,\lambda,\rho,\nu^\infty) \in \Y^2(\Omega;E)$. The \emph{second-scale barycenter} of $\bm{\nu}$ is the weak-$\ast$ $(\Leb^d_\Omega + \lambda^s)$-measurable map $\bbpr{\bm{\nu}} : x \mapsto \bbpr{\bm{\nu}}_x \in \M(Z;E)$ where $\bbpr{\bm{\nu}}_x$ is the measure defined by the values on Borel subsets $V \subset Z$ as 
		\[
		\bbpr{\bm{\nu}}_{x}(V) = 
		\begin{cases} \displaystyle \int_V \dprb{\id_{E},\nu_{x,\xi}} \dd \xi \; + \;  \lambda^{ac}(x) \int_V \dprb{\id_{\partial \Bbb_E},\nu_{x,\xi}^\infty} \dd \rho_{x}(\xi) & \text{if
			$x \in \reg_{\bm{\nu}}(\Omega)$}\\
		\\
		\displaystyle \int_V \dprb{\id_{\partial \Bbb_E},\nu_{x,\xi}^\infty} \dd \rho_{x}(\xi) & \text{if $x \in \sing_{\bm{\nu}}(\Omega)$}
		\end{cases}\,.
		\]
	\end{definition}
	The barycenter $[\bm{\nu}]$ can be recovered by integration from the second-scale barycenter. Indeed,  %$(\Leb^d + (\lambda^\nu)^s)$ and the second-scale barycenter, that is, 
	$[\bm{\nu}] = \bbpr{\bm{\nu}}(Z) \, (\Leb^d_\Omega + \lambda^s)$ and
	%\[
	% [\bm{\nu}](U) = \big((\Leb^d + (\lambda^\nu)^s) \otimes \bbpr{\bm{\nu}}_x\big) (U \times Z) \quad \text{for all $U \in \Bfrak(\cl \Omega)$}. %\dd( \Leb^n + (\lambda^\nu)^s) \qquad  \text{ for all $E\in \Bfrak (\cl \Omega)$.} 
	%\]
	in particular
	\begin{equation}\label{eq:density2}
	\bbpr{\nu}_x(Z) = \begin{cases}\displaystyle
	\frac{\dd [\nu]}{\dd \Leb^d}(x) & \text{$\Leb^d$-almost everywhere in $\Omega$}\\
	\\
	\displaystyle \frac{\dd [\nu]}{\lambda^s}(x) & \text{$\lambda^s$-almost everywhere in $\cl \Omega$}\\
	\end{cases}\,.
	\end{equation}
	Moreover, if $\mu_\eps \toY^2 \bm{\nu}$, then 
	\begin{equation}\label{eq:1-2}
	\mu_\eps \toweakstar [\bm{\nu}] \quad \text{as measures on $\cl \Omega$}
	\end{equation}
	and
	\begin{equation}\label{eq:1-3}
	\text{$\mu_\eps$ two-scale converges to $(\Leb^d_\Omega + \lambda^s) \otimes \bbpr{\bm{\nu}}_x$}.
	\end{equation}

	%\begin{corollary} Let $\{\mu_\eps\} \subset \M(\Omega;E)$ be a uniformly bounded sequence of measures. Then, up to passing to a subsequence, the sequence $\mu_\eps$ has a two scale limit.
	%\end{corollary} 

	\begin{corollary}[compactness of two-scale convergence]\label{cor:tsc} Let $\eps \searrow 0$ be an infinitesimal sequence of real numbers and let $(\mu_\eps)_\eps \subset \M(\Omega;E)$ be a sequence of measures with uniformly bounded total variation.
		%\[
		%\sup_\eps |\mu_\eps|(\Omega) < \infty.
		%\] 
		Then there exists a subsequence $(\eps_k)_{k \in \Nbb}$, a positive measure $\kappa \in \M(\cl \Omega)$, and a weak-$*$ $\kappa$-measurable map $\theta$ from $\cl \Omega$ into $\M(Z;E)$ such that
		\[
		\mu_{\eps_k} \; \text{two-scale converges to $\kappa \otimes \theta_x$}\,.
		\]
		Moreover $(\kappa \otimes |\theta_x|)(\cl \Omega \times Z) \le \liminf_{\eps \todown 0} |\mu_\eps|(\Omega)$.
	\end{corollary}
	\begin{proof}Apply Theorem~\ref{thm:representation} and~\eqref{eq:1-3} to the sequence $(\mu_\eps)_\eps$. The first conclusion follows by setting  $\kappa = \Leb^d_\Omega + \lambda^s$ and $\theta = \bbpr{\bm{\nu}}$. The lower semicontinuity of the norms follows from the fact that 
		\[
		\mu_\eps \toweakstar \bpr{\bm{\nu}} = \theta(Z)\, \kappa\quad \text{as measures in $\cl \Omega$}\,.
		\] 
	\end{proof}
	
	\begin{remark}[classical two-scale convergence] Notice that if a sequence of functions $(u_\eps)_\eps$ is equi-integrable (or if its is uniformly bounded in $\Lrm^p$ for some $1 < p \le \infty$) and the sequence $u_\eps \Leb^d_\Omega$ generates a two-scale* Young measure $\bm{\nu} = (\nu,\lambda,\rho,\nu^\infty) \in \Y^2(\Omega;E)$, then $\lambda \equiv 0$ and hence 
		\[
		\mu_\eps \; \text{two-scale converges to $u\, (\Leb^d_\Omega \otimes \Leb^d_Z)$}
		\]
		where
		\[
		u(x,\xi) \coloneqq \dpr{\id_E, \nu_{x,\xi}}\,.
		\]
		Thus the space of all such two-scale limits can be identified with $\Lrm^1(\Omega \times Z;\R^N)$. In particular, our definition of two-scale convergence extends Nguetseng's original definition of two-scale convergence in $\Lrm^p$-spaces \cite{nguetseng1989a-general-conve} (see also \cite{allaire1992homogenization-}).
		
	\end{remark}

	%A simple consequence of this is, that if $\bm{\nu}$ is an $\Acal$-free two-scale Young measure, then $\Acal([\nu]) =0$ in the sense of distributions.
	%\mnote{J: A-freeness is not introduced yet}
	
	\subsubsection{Weighted barycenter measures} 
	It will be often resourceful to interpret the pairing $\ddprb{\frarg,\frarg}$ as a measure over $\cl \Omega$. Let us introduce some additional notation by extending the definitions of barycenter and two-scale barycenter of a two-scale* Young measure $\bm{\nu} \in \Y^2(\Omega;E)$. 
	
	\begin{definition}Let $f \in \Rbf^2(\Omega;E)$. The $f$-\emph{barycenter} of $\bm{\nu} = (\nu,\lambda,\rho,\nu^\infty)$ is the vector-valued measure %$[\bm{\nu}]\in \M (\cl\Omega; E)$ uniquely determined by its values 
		\begin{align*}
		[f,\bm{\nu}] \coloneqq  \bigg(\int_{Z} \dprb{f(x,\xi,\frarg),\nu_{x,\xi} } \dd \xi \bigg) \Leb^d_\Omega 
		+ \bigg( \int_{Z} \dprb{f^\infty(x,\xi,\frarg),\nu^\infty_{x,\xi} } \dd \rho_x(\xi)\bigg) \lambda\,.
		\end{align*}
	\end{definition}
	Using this notation we get $\bpr{f,\bm{\nu}}(\cl \Omega) = \ddprb{f, \bm{\nu}}$, and $[\bm{\nu}]=\big(\bpr{(\id_{\R^N})_j,\bm{\nu}}\big)_{j=1,\dots,N}$.
	We also define a weighted \emph{second-scale $f$-barycenter}. 
	
	\begin{definition}[second-scale $f$-barycenter] Let $f \in \Rbf^2(\Omega;E)$. We define the $f$-\emph{barycenter} of $\bm{\nu} = (\nu,\lambda,\rho,\nu^\infty)$ as the weak-$*$ $(\Leb^d + \lambda^s)$-measurable map $\bbpr{f,\bm{\nu}} : x \mapsto \bbpr{ f, \bm{\nu}}_{x} \in \M(Z;{E})$ where is the measure defined by its values on Borel subsets $V \subset Z$ as  %$E \in \Bfrak (Z)$ and $x \in \cl \Omega$ by %\mnote{J:second scale weighted barycenter is not the density of first scale barycenter, because it should operate on the torus, as defined below. I think we will not need this guy.}
		%\mnote{J: I changed this, and hope that this is the way you wanted it? ;p}
		\begin{equation}\label{def:weighted}
		\bbpr{ f, \bm{\nu}}_{x}(V) = 
		\begin{cases}
		\displaystyle \int_V \dprb{f(x, \xi,\frarg),\nu_{x,\xi}}\, \dd \xi \\
		\displaystyle\qquad + \; \lambda^{ac}(x) \int_Z \dprb{f^\infty(x,\xi,\frarg),\nu_{x,\xi}^\infty}\dd \rho_{x}(\xi) \quad & \text{if $x \in \reg_{\bm{\nu}}(\Omega)$}\\
		\\
		\displaystyle \int_V \dprb{f^\infty(x, \xi,\frarg),\nu_{x,\xi}^\infty} \dd \rho_{x}(\xi) & \text{if $x \in \sing_{\bm{\nu}}(\Omega)$}
		\end{cases}.
		\end{equation}
	\end{definition}
	As we have seen before for the barycenter measures, a similar integral property holds for the weighted barycenters, namely $\bpr{f,\bm{\nu}} = \bbpr{f,\bm{\nu}}(Z) \, (\Leb^d_\Omega + \lambda^s)$ which in particular entails the identities 
	\begin{equation}\label{eq:density3}
	\bbpr{f,\bm{\nu}}_x(Z) = \begin{cases}\displaystyle
	\frac{\dd [f,\bm{\nu}]}{\dd \Leb^d}(x) & \text{$\Leb^d$-almost everywhere in $\Omega$}\\
	\\
	\displaystyle \frac{\dd [f,\bm{\nu}]}{\lambda^s}(x) & \text{$\lambda^s$-almost everywhere in $\cl \Omega$}\\
	\end{cases}\,.
	\end{equation}

	\subsection{Heuristics and some generic examples}

	Let us begin by recalling the criterion for $\Lrm^1$-weak compactness due to Dunford \& Pettis: a sequence $(w_\eps)_\eps \subset \Lrm^1(\Omega)$ is $\Lrm^1$-weak relatively compact if and only if it is equi-integrable, that is, whenever
	\[
	\lim_{R \to \infty} \bigg(\limsup_{\eps \todown 0} \int_{\Omega \cap \{|w_\eps| \ge R\}} |w_\eps| \bigg)= 0.
	\]
	At those regions where a sequence fails to be equi-integrable, but its weak-$*$ limit remains absolutely continuous ---which we call \enquote{continuous concentration}, we speak of a \emph{biting limit} of the sequence:
	\begin{lemma}[Chacon]\label{chacon} Let $(w_\eps) \subset \Lrm^1(\Omega)$ be a uniformly bounded sequence. Then there exist $w \in \Lrm^1(\Omega)$, a subsequence $(w_{\eps_j})_{j \in \Nbb}$, and a non-increasing sequence of measurable subsets $(K_m)_{m \in \Nbb}$ of $\Omega$ such that
		\begin{enumerate}
			\item $\Leb^d(K_m) \to 0$ as $m \to \infty$,
			\item $w_{\eps_j} \toweak w$ in $\Lrm^1(\Omega \setminus K_m)$  for all $m \in \Nbb$.
		\end{enumerate}
		Provided that the these hold, we say that $w$ is the biting limit of $(w_{\eps_j})_{j \in \Nbb}$, and the set 
		\[
		B_u \coloneqq \setBB{x \in \Omega}{x \notin \bigcap_{m \in \Nbb} K_m}
		\]
		is called the set of biting points of $(w_{\eps_j})_{j \in \Nbb}$.
	\end{lemma}
	
	Given a sequence $(\mu_\eps)_\eps$ that generates a two-scale* Young measure $\bm{\nu} = (\nu,\lambda,\rho,\nu^\infty)$ it is possible to understand, qualitatively speaking, how each element of the four-tuple can be understood in terms of the generating sequence. In the following arguments we will assume with a slight abuse of notation that $(\mu_\eps = u_\eps)_\eps \subset \Lrm^1(\Omega;E)$.
	A simple consequence of the representation of two-scale* Young measures is that $\lambda$ and $\rho_x$ are supported (with respect to the $x$-variable) at non-biting points of the sequence that generates $\bm\nu$. Moreover, this two measures carry the mass of the sequence $(|\mu|_\eps)_\eps$ for different length-scales:
	
	\begin{enumerate}[(i)]
		\item \emph{The measure $\lambda$ quantifies the limit mass carried by $(|\Leb^d_\Omega u_\eps|)_\eps$ in the set where it fails to be equi-integrable.} %If $(|\mu|_\eps)_\eps \subset \Lrm^1(\Omega;\R^N)$, then this set coincides with the set where the sequence fails to be weak-$\Lrm^1$ compact.} 

		By Theorem 2.9 in~\cite{alibert1997non-uniform-int}, the biting limit of the sequence is the measure $\dprb{|\frarg|,\tilde\nu{\dd x}} \Leb^d_\Omega$ where $\upsilon = (\tilde\nu,\tilde \lambda,\tilde \nu^\infty) \in \Y^1(\Omega;E)$ is the Young measure generated by $(u_\eps)_\eps$. Hence, we conclude from the representation of two-scale*
		Young measures that
		\[
		|u_\eps \Leb^d_\Omega| - \dprb{|\frarg|,\tilde\nu{\dd x}} \Leb^d_\Omega \toweakstar \tilde \lambda \equiv \lambda \quad \text{as measures on $\cl\Omega$}.
		\]
		%This backs-up our claim.

		\item \emph{The probability measure $\rho_{x_0}$ quantifies, at a given non-biting point $x_0$ of $(|u_\eps \Leb^d_\Omega|)_\eps$, the homogenized  mass carried by the sequence about $x_0$ in the following sense: If $A \subset Z$, then %$\rho_{x}(E)$ coincides is the limit mass of the sequence $\chi_E(x/\eps)|\mu_\eps|$.
			\[
			\int_U \chi_A\bigg(\frac{x}{\eps}\bigg)|u_\eps| \dd x \to \int_U \rho_x(A) \dd \lambda(x) \quad \text{for all Borel subsets $U \subset \cl \Omega$}.
			\]}
		%Moreover, a point $\xi_0$ in the $d$-dimensional torus is a singular point of $\rho_{x_0}$ if and only if 
		%\[
		%\liminf_{s \todown 0} \Big(\liminf_{\eps\todown 0} |\mu_\eps|\big(\{\xi_0+ \eps\Zbb^d\} \cap B_s(x_0)\big)\Big) > 0
		%\]
		%\note{Drawing of the concentric squares here}
		The argument is a direct consequence of the representation of two-scale* Young measures and point (i). 
		%
		%\item \emph{Given a biting point $x_0$ of $\{\mu_\eps\}$ and $\xi_0$ a point the $d$-dimensional torus, 
		%}
	\end{enumerate}
	
	To give a better understanding of how the different components of the limiting Young measure are connected to the features of a limiting sequence, we present the following examples, that also emphasize the possibility of concentration for $\rho$ and $\lambda$. 
	\begin{figure}[h]
		\captionsetup{justification=raggedright}  
		\centering
		\begin{tikzpicture}
		\draw (1.5,1)--(1.5,0.95);
		\draw (-1.5,1)--(-1.5,0.95);
		\draw (-1.5,1)--(1.5,1);
		\node at (1.5,0.95)[anchor = north] {\scriptsize{$1$}};
		\node at (-1.5,0.95)[anchor = north] {\scriptsize{$-1$}};
		\draw (0,1)--(0,0.95);
		\node at (0,0.95) [anchor=north]{\scriptsize{$0$}};
		
		\draw[->] (0,1)--(0,3);
		
		\draw (0.5,1)--(0.5,0.95);
		\node at (0.5,0.95)[anchor=north]{\scriptsize{$\eps^\alpha$}};
		
		\draw (0,2.5)--(0.5,2.5)--(0.5,1);
		\node at (0,2.5)[anchor=east]{\scriptsize{$\eps^{-\alpha}$}};
		\fill[ gray, opacity =0.3] (0,1) rectangle (0.5,2.5);

		\node at (0,0.4){~};
		\end{tikzpicture}
		\hspace{0.8cm}
		\begin{tikzpicture}
		
		\draw (0,1)--(3,1);
		\draw (3,1)--(3,0.95);
		\node at (3,1)[anchor = north] {\scriptsize{$1$}};
		\draw (1.5,1)--(1.5,0.95);
		\node at (1.5,1)[anchor = north] {\scriptsize{$\frac 12$}};
		%\node at (-1.5,1)[anchor = north] {\scriptsize{$-1$}};
		\draw (0,1)--(0,0.95);
		\node at (0,0.95) [anchor=north]{\scriptsize{$0$}};
		
		\draw[->] (0,1)--(0,3);
		
		\node at (0,2)[anchor=east]{\scriptsize{$\eps^{-1}$}};
		
		\node at (0.7,2)[anchor=south]{\scriptsize{$\eps^{2}$}};
		
		\draw (0,1)rectangle(0.1,2);
		\fill[ gray, opacity =0.3] (0,1)rectangle(0.1,2);
		\draw (0.3,1)rectangle(0.4,2);
		\fill[ gray, opacity =0.3] (0.3,1)rectangle(0.4,2);
		\draw (0.6,1)rectangle(0.7,2);
		\fill[ gray, opacity =0.3] (0.6,1)rectangle(0.7,2);
		\draw (0.9,1)rectangle(1,2);
		\fill[ gray, opacity =0.3] (0.9,1)rectangle(1,2);
		\draw (1.2,1)rectangle(1.3,2);
		\fill[ gray, opacity =0.3] (1.2,1)rectangle(1.3,2);

		\node at (0.75,1)[anchor=north]{\scriptsize{$\eps^{-1}$ many}};
		
		\node at (0,0.4){~};
		\end{tikzpicture}
		\hspace{0.8cm}
		\begin{tikzpicture}
		\draw (1.5,1)--(1.5,0.95);
		\draw (-1.5,1)--(-1.5,0.95);
		
		\draw (-1.5,1)--(1.5,1);
		\node at (1.5,1)[anchor = north] {\scriptsize{$1$}};
		\node at (-1.5,1)[anchor = north] {\scriptsize{$-1$}};
		\draw (0,1)--(0,0.95);
		\node at (0,0.95) [anchor=north]{\scriptsize{$0$}};
		
		\draw[->] (0,1)--(0,3);
		\draw (0.5,1)--(0.5,0.95);
		\node at (0.5,0.95)[anchor=north]{\scriptsize{$\frac 12 \eps$}};

		\draw (0.5,1)--(0.5,2.75)--(0.65,2.75)--(0.65,1);
		\node at (0,2.75)[anchor=east]{\scriptsize{$\eps^{-2}$}};
		\node at (0.595,2.75)[anchor=south]{\scriptsize{$\eps^{2}$}};
		\fill[ gray, opacity =0.3] (0.5,1) rectangle (0.65,2.75);
		
		\node at (0,0.4){~};

		\end{tikzpicture}
		
		\caption{Sketches for the functions in Examples~\ref{expl:Singularpart}(a),~\ref{expl:Singularpart}(b) and~\ref{expl:xi0isthere} respectively.} \label{fig:Expl}
	\end{figure}
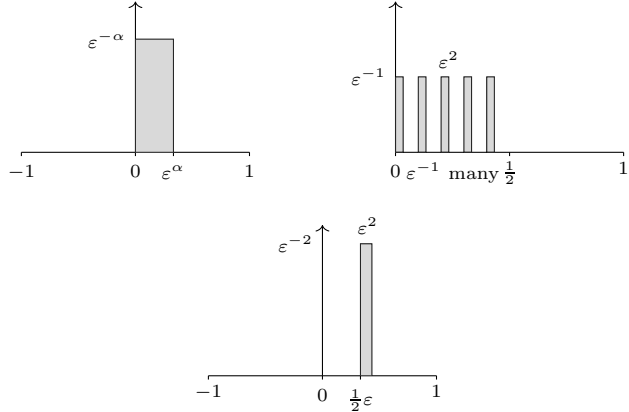
	
	For the sake of simplicity we will discuss examples in dimensions $d = 1$ and $E = \R$; their correspondent versions to higher dimensions are easily constructed by adding invariant directional measures or by mimicking similar constructions along transversal directions. We cover generic examples for each of the qualitative \emph{first-scale/second-scale} scenarios: %\emph{singular/regular}, \emph{singular/singular},  \emph{non-biting/regular}, \emph{non-biting/singular}, and the classical \emph{biting/regular}.
	\begin{table}[h!]
		\centering
		\begin{tabular}{| c | c | c |} 
			\hline
			macro-scale ($x$) & micro-scale ($\xi$) & correlation   \\ [0.5ex] 
			\hline
			abs. cont. & abs. cont. & $\mu_\eps$ equi-integrable   \\ 
			\cline{2-3}
			& concentration/dissipation & scale of phenomena $\gg \eps$  \\
			\hline
			concentration/dissipation & abs. cont. & scale of phenomena $\le \eps$   \\
			\cline{2-3}
			& concentration/dissipation & scale of phenomena $\gg \eps$   \\   [1ex] 
			\hline
		\end{tabular}
		\caption{Qualitative properties of the elements of a two-scale* Young measure.}
		\label{table:1}
	\end{table}

	\begin{example}[singularity on the macro-scale]\label{expl:Singularpart} 
		Fix $\alpha > 0$ and let $\Omega = (-1,1)$. We consider the family of functions
		
		\[
		u_{\varepsilon}(x) =  \varepsilon^{-\alpha} \chi_{(0,\varepsilon^\alpha)}(x), \quad \eps > 0.
		\]
		%\begin{figure}[h]
		%	\captionsetup{justification=raggedright}  
		%	\centering
		%	\begin{tikzpicture}
		%	\draw (1.5,1)--(1.5,0.95);
		%	\draw (-1.5,1)--(-1.5,0.95);
		%	\draw (-1.5,1)--(1.5,1);
		%	\node at (1.5,0.95)[anchor = north] {\scriptsize{$1$}};
		%	\node at (-1.5,0.95)[anchor = north] {\scriptsize{$-1$}};
		%	\draw (0,1)--(0,0.95);
		%	\node at (0,0.95) [anchor=north]{\scriptsize{$0$}};
		%	
		%	\draw[->] (0,1)--(0,3);
		%	
		%	\draw (0.5,1)--(0.5,0.95);
		%	\node at (0.5,0.95)[anchor=north]{\scriptsize{$\eps^\alpha$}};
		%	
		%	\draw (0,2.5)--(0.5,2.5)--(0.5,1);
		%	\node at (0,2.5)[anchor=east]{\scriptsize{$\eps^{-\alpha}$}};
		%	\fill[ gray, opacity =0.3] (0,1) rectangle (0.5,2.5);
		%	
		%	
		%	
		%	\node at (0,0.4){~};
		%	\end{tikzpicture}
		%\end{figure}
		Let us assume, up to taking a sequence of $\eps$'s that $u_\eps \toYY \bm{\nu} = (\nu,\lambda,\rho,\nu^\infty)$ (in this case every subsequence generates the same Young measure). The following observations are easy to check:
		\begin{enumerate}
			\item \emph{Pure concentration in the $x$-variable.} On the one hand $|u_\varepsilon| \Leb^1 \toweakstar \delta_0$. On the other hand, using $f(x,y,z)= |z|$ as a test function the Young measure representation yields 
			\[
			|u_{\varepsilon_k}|  \Leb_\Omega^d \toweakstar \bigg(\int_{Z}\langle| \frarg | ,\nu_{x,\xi} \rangle \dd\xi\bigg) \Leb_\Omega^d + \lambda\,,
			\]
			whence we deduce $\nu_{x,\xi} = \delta_0$ for $\Leb^1 \otimes \Leb^1$-almost every $(x,\xi)$ in $[-1,1] \times Z$ and $\lambda \equiv \delta_0$.  Thus, it suffices to characterize $\bm\nu$ at $x = 0$.
			\item Testing with an integrand of the form $f(x,\xi,z) = \phi(x)g(\xi)|z|$ we see through a change of variables that
			\begin{align}
			\int_{-1}^{1} \phi(x) g(x/\varepsilon)h(u_\varepsilon(x))  \dd x
			& = \int_{0}^{\eps^\alpha} \varepsilon^{-\alpha}\phi(x)g(x/\varepsilon)h(1)\dd x	 \nonumber\\
			& = \int_0^{1} \phi(\varepsilon^\alpha y )g(\varepsilon^{\alpha-1} y)h(1) \dd y  \label{eq:explsingulartorus}  \rightarrow c(\alpha)\,.
			%	 \begin{cases}
			%		     \phi(0) g(0)h(1)			&\alpha >1\\
			%		     \phi(0)h(1) \int_0^1g(s) \dd s		&\alpha \le1 
			%	 \end{cases}. \nonumber
			\end{align}
			\begin{enumerate}
				\item \emph{Pure singularity on the micro-scale:} if $\alpha > 1$,  the Young measure representation and the limit above give 
				\[
				\int_Z g(\xi)\dpr{h,\nu_{0,\xi}^\infty} \dd \rho_0(\xi)  = c(\alpha) =  g(0)h(1) \quad \Rightarrow.
				\]
				In conclusion $\rho_0 = \delta_0$, $\nu_{0,0}^\infty = \delta_1$, and
				\[
				\bm{\nu} = (\delta_0,\delta_0,\delta_0,\delta_1)\,.
				\]
				\item \emph{Absolute continuity in the second-scale:} if, in turn $\alpha \le 1$, we get 
				\[
				\int_Z g(\xi)\dpr{h,\nu_{0,\xi}^\infty} \dd \rho_0(\xi) = c(\alpha) = \phi(0)h(1)\int_Z g(\xi) \dd \xi.
				\]
				Therefore $\rho_0 = \Leb^1_Z$ and $\nu_{0,0}^\infty = \delta_1$, which altogether yields
				\[
				\bm{\nu} =  (\delta_0,\delta_0,\Leb^1_Z,\delta_1)\,.
				\]
				
				% (\ref{eq:explsingulartorus}) yields $f(x,x/\varepsilon,u_\varepsilon(x))  \cdot \Leb^1 \toweakstar \int_0^1 \psi(s)\dd s \cdot \delta_0 $.
				%    By Theorem \ref{thm:classicalcompactness}, together with $\nu_{x,y} = \delta_0$ and $\lambda^\nu = \delta_0$ we compute, using the fact that $f^\infty(x,y,z) =\psi(y)|z|$,
				%    \begin{align*}
				%      f(x,x/\varepsilon,u_\varepsilon(x))  \cdot \Leb^1  \toweakstar 
				%      \int_0^1 \psi(y) \int_{\{-1,1\}} |z| \dd\nu^\infty_{x,y}(z) \dd\rho^\nu_x(y) \cdot \delta_0.
				%    \end{align*}
				%      We have that $\int_{\{-1,1\}} |z| \dd\nu^\infty_{y,x}(z)= 1$ and hence $\rho^\nu_0 = \Leb^1$. \\
			\end{enumerate}
		\end{enumerate}
	\end{example}
	\begin{remark}
		Observe that when oscillations of the sequence $(u_\varepsilon)_\eps$ happens at a coarser length-scale than $\{\eps\}$, then the the two-scale Young measure does not provide more information than the classical Young measure. On the other hand, if oscillations of $(u_\varepsilon)_\eps$  takes part at a finer scale than $\{\varepsilon\}$, then the two-scale* Young measure cannot be recovered from the classical Young measure.
	\end{remark}

	%% Using $f(x,y,z)= |z|$ as a test function we notice that  $|u_{\varepsilon_k}| \cdot \Leb_\Omega^d \toweakstar \int_{Z}\langle| \frarg | ,\nu_{x,y} \rangle \dd \Leb^d(y)\cdot \Leb_\Omega^d + \lambda^\nu $.
	%\begin{itemize}
	% \item[i)]Fix $\alpha>0$,  consider  $u_{\varepsilon}(x) =  \varepsilon^{-\alpha} \chi_{(0,\varepsilon^\alpha)}(x)$ and notice that $|u_\varepsilon| \cdot \Leb^1 \toweakstar \delta_0$.
	% Hence $\nu_{x,y} = \delta_0$ $\Leb^2$ almost everywhere and $\lambda^\nu = \delta_0$.\\
	% Let $\psi$ be 1-periodic in $\R$ and let $f(x,y,z)= \psi(y)|z|$.
	%	  Then for any   $\phi\in \Crm((-1,1))$  we have by change of coordinates
	%	  \begin{align}
	%	  \int_{-1}^{1} \phi(x) f(x,x/\varepsilon,u_\varepsilon(x))  \dd x
	%	 & = \int_{0}^{\eps^\alpha} \varepsilon^{-\alpha}\phi(x)\psi(x/\varepsilon)\dd x	 \nonumber\\
	%	 & = \int_0^{1} \phi(\varepsilon^\alpha y )\psi(\varepsilon^{\alpha-1} y) \dd y  \label{eq:explsingulartorus} \rightarrow
	%	 \begin{cases}
	%		     \phi(0) \psi(0)				&\alpha >1\\
	%		     \phi(0) \int_0^1\psi(s) \dd s		&\alpha <1. 
	%	 \end{cases} \nonumber
	%	  \end{align}

	\begin{example}[non-biting limit] Fakir's construction (also known as \emph{Fakir's carpet}) provides a good way to produce continuous concentrations. The idea is to create \emph{many} small concentrations which \emph{diffuse} before each of them can actually gain mass. 
		
		Let $\Omega = (0,1)$. In order to showcase the sensitivity of Young measures we shall consider a \enquote{bi-directional half-carpet} which is generated by the functions
		\[
		u_{\frac 1k}(x) \coloneqq k \sum_{i=0}^{k/2} (-1)^i\chi_{[\frac ik, \frac ik + \frac 1{k^2}]}(x), \quad  k = 2,4,6,\dots
		\]
		Similarly to the example above, let us already assume that $u_{\frac 1k} \toYY \bm\nu$. The following observations can be deduced directly from the representation of two-scale* Young measures:
		
		\begin{enumerate}
			\item \emph{First-scale analysis:} $|u_\varepsilon|\, \Leb^1_\Omega\toweakstar \Leb^1\restrict [0,1/2]$. There exist no singular points in the first variable which is encoded by the equality of measures $\lambda^s \equiv 0$. However, every point in the interval $[0,1/2]$ is a non-biting point of the sequence $(u_{\eps_k})_k$ (clearly the sequence fails to  be equi-integrable at any subset of this interval) and hence
			\[
			\lambda \equiv \Leb^1\restrict [0,1/2] \quad \text{as measures on $[0,1]$}\,.
			\]
			Moreover,   
			\[
			x \mapsto \int_Z \dpr{\nu_{x,\xi},|\frarg|} \dd \xi \equiv 0 \quad \Rightarrow \quad \nu_{x,\xi} = \delta_0 \quad \text{$\Leb^1_\Omega \times \Leb^1_Z$-almost everywhere}.
			\]

			%Then $|u_\varepsilon|\cdot \Leb^1\toweakstar \Leb^1|_{(0,1/2)}$ and in particular $(\lambda^\nu)^s=0$. 
			
			\item \emph{Second-scale analysis:} let $f = \phi \otimes g \otimes h \in \Rbf^2(\Omega;\R)$. The limit of the energies $I^{\eps_k}_f(u_{\eps_k})$ as $k \to \infty$ can be computed by Riemann-integral partial sums as %Let $\psi$ be 1-periodic in $\R$ and let $f(x,y,z)= \psi(y)|z|$.
			%Fix $\phi\in \Crm_0((-1,1))$ and $\gamma>0$. 
			%For $\varepsilon$ small enough we have $|\phi(x)-\phi(y)|\leq \gamma$ for $|x-y|\leq \varepsilon^2$ and $|\psi(x)-\psi(y)|\leq \gamma$ for $|x-y| \leq \varepsilon$.
			%We compute
			\begin{align*}
			\int_{0}^{1} \phi(x) g(kx)h(u_{\frac 1k}(x))  \dd x
			& = k \sum_{i=0}^{k/2} \int_{\frac i k}^{\frac i k + \frac 1{k^2}}\phi(x) g(kx)h((-1)^i) \dd x\\
			& =  \sum_{i=0}^{k/2} \int_{i}^{i + \frac 1{k}}\phi(y/k) g(y)h((-1)^i) \dd y \\ 
			& \sim {g(k^{-1})} \sum_{i=0}^{k/2}\phi(i/k) h((-1)^i) \dd y \\ % + \psi(0)\gamma + \|\phi\|_{\infty} \gamma\\
			& \rightarrow  g(0)\bigg(\frac{h(1)}{2}\int_0^\frac{1}{2} \phi(x) \dd x + \frac{h(-1)}{2}\int_0^\frac{1}{2} \phi(x) \dd x\bigg)\,.
			\end{align*}
			The representation of two-scale* Young measures and a density argument then give
			\[
			\bm{\nu} = (\delta_0,\Leb^1 \restrict [0,1/2],\delta_0,\frac 12 \delta_{-1} + \frac 12 \delta_1)\,.
			\]	 
			
			\begin{remark}
				As it can be seen from the representation of Young measures, at biting-points $x \in [0,1]$ of the sequence $\{u_\eps\}$, the correspondent probability measures $\rho_x$ must be the uniform measure $\Leb^1\restrict(0,1]$.
			\end{remark}
			%	 
			%	 
			%That is: $ f(x,x/\varepsilon,u_\varepsilon(x)) \cdot \Leb^1 \toweakstar \psi(0) \cdot \Leb^1|_{(0,1/2)}$.
			% By Theorem \ref{thm:classicalcompactness} we again compute, using the fact that $f^\infty(x,y,z) =\psi(y)|z|$,
			%    \begin{align*}
			%      f(x,x/\varepsilon,u_\varepsilon(x))  \cdot \Leb^1  \toweakstar &\int_0^1 \psi(y) \int_{\R} |z| \dd \nu_{x,y}(z) \dd\Leb(y) \cdot \Leb^1\\
			%      &+\int_0^1 \psi(y) \int_{\{-1,1\}} |z| \dd \nu^\infty_{x,y}(z) \dd\rho_x(y) \cdot \lambda^\nu.
			%    \end{align*}
			%Comparing the two weak limits we conclude that ${\rho_x}= \delta_0$ for $\Leb^d$ almost every $x\in (0,1/2)$, $\nu_{x,y}= \delta_0$ and $\lambda^\nu=0$ on $(-1,1)\setminus (0,1/2)$. \\
			% That is: It is indeed possible to have a singular measure on the torus, although $\Lambda$ is absolutely continuous.
			% Notice that this combination only occurs if the oscillation in $u_\varepsilon$ is of the sames speed as the oscillation in $x/\varepsilon$.
		\end{enumerate}
		
		%\end{itemize}

	\end{example}

	The precise representation of a Young measure does not only depend on  the generating sequence $u_\eps$ but is also strongly influenced by the speed of oscillation $\eps$. 
	This is an interesting and important feature of the compactness, that will also play a role in the localization principles below, and is emphasized in the following example.
	\begin{example}[non-uniqueness] \label{expl:xi0isthere}
		Let again $d=N=1$ and $\Omega = (0,1)$. For a fixed $\eps > 0$ and arbitrary $c,d \in (0,1]$, we consider the function 
		\[
		u_\eps (x) \coloneqq \frac 1{\eps^{2}} \chi_{(a + b\eps, a + b \eps +\eps^2)}(x).
		\]
		\emph{First-scale:} If $\eps \searrow 0$, then for every subsequence $\eps_1\searrow 0$ and
		\begin{equation}\label{eq:non}
		u_{\eps_1} \toY \upsilon = (\delta_0,\delta_a,\delta_1). 
		\end{equation}
		
		\noindent\emph{Second-scale:} We may assume that $\area{\frac a{\eps_1}} \to \xi_1 \in Z$. Here, we recall that $\area{x} \in Z$ stands for the equivalence class of $x \in \R$ in the one-dimensional torus $Z$. Testing with an integrand $f = \phi \otimes g \otimes h$ gives
		
		%and fix a sequence $\eps_k \searrow 0$.
		%As in the previous example we see that $|u_{\eps_k}|=u_{\eps_k} \toweakstar \delta_{1/2}$, hence, after taking a subsequence (not relabeled!), $u_{\eps_k} \toY \bm{\nu}$ for some $\bm{\nu} \in \Y$ with $\nu_{x,\xi}= \delta_{1/2}$ and $(\lambda^\nu)^{ac}=0$.
		%We define $f = \chi_\Omega \otimes \psi \otimes |\frarg |$ for some arbitrary test function $\psi \in C^0(Z)$  and conclude $\ddprb{f,\bm{\nu}}= \int_Z \psi(\xi) \dd \rho^\nu_{1/2}(\xi)$.
		%On the other hand we have
		\begin{align*}
		\int_\Omega f(x,x/{\eps_1},u_{\eps_1}) \dd x
		& =  \int_{a + b {\eps_1}}^{a + b  \eps_1 +\eps_1^2} \phi(x)g(x/\eps_1)h(\eps_1^{-2}) \dd x \\
		& = \eps h(\eps_1^{-2})\int_{\frac{a}{\eps_1} + b}^{\frac{a}{\eps_1} +b + \eps_1} \phi(\eps_1 y)g(y) \dd y\\
		& \sim  \phi(a)g\big(\area{{a}/{\eps_1}} + b + \eps_1\big) \cdot \eps_1^2 h(\eps_1^{-2})\\
		& \rightarrow \phi(a)g(\xi_1 + b)h^\infty(1).
		\end{align*}
		From this we infer that $u_{\eps_1} \toYY \bm{\nu}_1$, where
		\[
		\bm{\nu}_1 =  (\delta_0,\delta_a,\rho^1 = \delta_{\xi_1 + b},\delta_1).
		\]	 
		%
		%where $\xi_0$ is such that $\frac 12 \eps_k^{-1} \rightarrow \xi_0 \in Z$.
		%We directely follow that $\rho^\nu_{1/2} = \delta_{1/3+\xi_0}$.
		Notice that $\xi_1$ does not depend on $\{u_{\eps_1}\}$, but solely on subsequence $\{\eps_1\}$. 
		%In particular: If starting with an arbitrary sequence, taking the subsequence that yields the existance of the generated Young measure will in particular lead to the convergence of $\frac 1{2\eps_k}$ and determine $\xi_0$. 
		%If we a priori chose $\eps_k$ such that $\frac 1{2\eps_k} \rightarrow \bar{\xi_0}$, then taking a subsequence will not change this feature	, so the generated Young measure will satisfy $\rho = \delta_{1/3+\bar{\xi_0}}$
		Hence, the choice of a different subsequences generates a range of two-scale* Young measures  (compare this with the uniqueness of \eqref{eq:non}). Indeed, if $\{\eps_2\}$ is another subsequence satisfying
		\[
		\xi_2 \coloneqq \lim_{\eps_2 \todown 0} \area{\frac a {\eps_2}} \neq \xi_1,
		\]
		then $u_{\eps_2} \toYY \bm{\nu}_2 = (\delta_0,\delta_a,\rho^2 = \delta_{\xi_2 + b},\delta_1)$.
		However, passing to a subsequence does not entirely \emph{forget}  in the sense that it is possible to relate $\bm\nu_1$ and $\bm \nu_2$ by a translation in the torus:
		\[
		\Gamma^{\xi_1}_\# \rho^1 \equiv \Gamma_\#^{\xi_2} \rho^2 \equiv \delta_b\,,
		\]
		where $\Gamma^\eta : \xi \mapsto \xi - \eta$ is a translation map in  $\R^d$ (which also determines a translation in the $d$-dimensional torus).
		In fact this translation in the second-scale also occurs at biting points of the sequence. However, this is not reflected in the two-scale* Young measure since $\Leb^d_Z$ is an invariant measure under translation, that is, 
		\[
		\Gamma^\xi_\# \Leb^d_Z \equiv \Leb^d_Z\quad \text{for all $\xi \in Z$}.
		\]

	\end{example}%

	\section{Localization principles}\label{sec:5}
	
	In this section we treat the (measure theoretic) \emph{differentiation} of Young measures which confirms the observation that the convergence
	\[
	\mu_\eps \toYY \bm\nu \in \Y^2(\Omega;E)
	\]
	is in fact local (with respect to the macroscopic variable $x$).
	We show that at a point $x_0 \in \Omega$, the information carried by $\bm\nu$ can be recovered by simply looking at the homogeneous Young measures $\bm\sigma \in \Y^2(Q;E)$ generated by blow-ups of the generating sequence $(\mu_\eps)_\eps$ at $x_0$. With the mathematical thrust of introducing Young measures as a serving tool, we establish localization principles at both continuity and singular points.
	Next, we recall some facts about the push-forward of blow-up maps on measures. %The proof of localization principles for classical Young measures  
	
	%The general idea is to dilate, with $x_0$ as the observer point, each element $\mu_\eps$ by a factor $r_\eps^{-1}$ where $(r_\eps)_\eps$ is an infinitesimal sequence. Of course, the rate at which $r_\eps$ converges to $0$ cannot be expected to be arbitrary since it must observe the first- and second-scale behavior of $(\mu_\eps)_\eps$. To formalize this reasoning we shall require some facts about the push-forward of blow-up maps on measures.
	
	Throughout this section we will indistinctly use the zero-extension map $\M(\Omega;E) \embed \M_\loc(\R^d;E)$ to identify measures defined on $\Omega \subset \R^d$ with measures defined on the whole space $\R^d$. 
	Fix $\mu \in \Mcal(\Omega;\R^N)$ and consider the map $T^{(x_0,r)}(x) := (x - x_0)/r$, 
	which blows-up $B_r(x_0)$, the open ball around $x_0 \in \Omega$ with radius $r > 0$, into the open unit ball $B_1 \subset \R^d$. Following the definition of push-forward, it is easy to check that
	\[
	T^{(x_0,r)}_\# \mu(B) := \mu(x_0+rB ) \quad \text{whenever $B \subset \frac 1r(\Omega-x_0)$ is a Borel set.}
	\]
	A simple calculation shows that the Radon--Nykodym decomposition of $T^{(x_0,r)}_\#\mu$ can be re-written in terms of $\mu$ as 
	\begin{equation}\label{eq:blow_r}
	T^{(x_0,r)}_\#\mu = a \, \Leb^d + s \, T_\#^{(x_0,r)} |\mu^s|,  
	\end{equation}
	where the densities $a \in \Lrm^1_\loc(\R^d;\R^N)$ and $s \in \Lrm^\infty_{T_\#^{(x_0,r)} |\mu^s|}(\R^d;\Sbb^{N-1})$ are defined by the rules 
	\begin{align*}
	a(y) & \coloneqq r^d\frac{\dd \mu}{\dd \Leb^d}(x_0 + ry), \\
	s(y)& \coloneqq \frac {\dd\mu}{\dd |\mu^s|} (x_0+ry).
	\end{align*}
	
	%we always have the decomposition
	%\begin{align} 
	%T_\#^{(x_0,r)} \mu &= r^d \frac {\dd\mu}{\dd \Leb^d}(x_0+r\cdot \frarg)  \Leb^d + \frac {\dd\mu}{\dd |\mu^s|} (x_0+r\cdot \frarg) T_\#^{(x_0,r)} |\mu^s| \nonumber \\
	%\text{ and }\quad
	%  T_\#^{(x_0,r)} |\mu^s|& = |(T_\#^{(x_0,r)} \mu)^s|\label{eq:coefficentsofblowupmeasure}
	%\end{align}

	%\mnote{J: This will be probably be moved in the introduction, then maybe with a label such that we can shortly reference on it}
	%For $x\in \Omega$, $r>0$ we define $T^{(x,r)}:\R^d\rightarrow \R^d$ by $T^{(x,r)}(y) = \frac {y-x}r$. In particular: $T^{(x,r)}$ maps $Q_r(x)$ bijective to $Q_1(0)$.
	%This defines for $\mu \in \Mcal_{\loc} (\R^d)$ the blow up measure $T_\#^{(x,r)} \mu \in \Mcal_{\loc} (\R^d)$ by $T_\#^{(x,r)} \mu(A) = \mu (T^{(x,r)}(A))$. 
	%If $\mu \in \Mcal(\Omega)$ we expand it by zero. Notice that we always have the decomposition
	%\begin{align} 
	%T_\#^{(x_0,r)} \mu &= r^d \frac {\dd\mu}{\dd \Leb^d}(x_0+r\cdot \frarg)  \Leb^d + \frac {\dd\mu}{\dd |\mu^s|} (x_0+r\cdot \frarg) T_\#^{(x_0,r)} |\mu^s| \nonumber \\
	%\text{ and }\quad
	%  T_\#^{(x_0,r)} |\mu^s|& = |(T_\#^{(x_0,r)} \mu)^s|\label{eq:coefficentsofblowupmeasure}
	%\end{align}
	
	Retaking the notation of the last section, we write $\Gamma^{\xi_0}:\xi \mapsto \xi - \xi_0$ to denote the translation in $\R^d$ by a vector $\xi_0 \in \R^d$. To avoid a more intricate 
	notation, we shall also write $\Gamma^{\xi_0}$ (with $\xi_0 \in [0,1)^d$) to denote the same translation when restricted to $Z$, that is, $\Gamma^{\xi_0} : Z \to Z : {\xi} \to \area{\xi - \xi_0}$.  
	In this way, the push-forward action $\Gamma_\#^{\xi_0} : \M(Z) \to \M(Z)$ defines an automorphism of spaces. 
	In particular, if $g\in L^1_\rho(Z)$, then 
	\begin{align}\label{eq:translationinvariance}
	\int_{Z} g(\xi - \xi_0)   \dd \rho(\xi)   =  \int_{Z}  g(\xi)   \dd (\Gamma^{\xi_0}_{\#}\rho)(\xi).
	\end{align}
	Notice that if $\rho$ is a uniform measure in $Z$, then $\rho$ is translation invariant:
	\[
	\Gamma^{\xi_0}_{\#}\rho \equiv \rho\quad \text{for all $\xi_0 \in Z$)}.  
	\]
	
	To avoid any possible confusion we shall write $y \in Q$ (or $\R^d$) to denote the blow-up variable; we keep the notation $x \in \Omega$ for the macroscopic scale.
	%\begin{definition}
	%We say that $\sigma \in \Y_\loc(\Omega;\R^N)$ is a tangent Young measure of $\nu$ at $x_0$ if there exist ...
	%\end{definition}
	\subsection{Localization at regular points}
	
	\begin{lemma}\label{lem:locatregular}
		Let $\bm{\nu} = (\nu,\lambda,\rho,\nu^\infty) \in \Y^2(\Omega;E)$ be a two-scale* Young measure generated by a sequence  $(\mu_\eps)_\eps \subset \M(\Omega;E)$. %generated by sequences $\eps \searrow 0$ and $(\mu_\varepsilon) \subset \M(\Omega;\R^N)$. 
		Then, at $\Leb^{d}$-almost every $x_0\in \Omega$ there exists  (up to a translation in $Z$) a regular tangent two-scale* Young measure ${\bm{D\nu}}$ $= (D\nu,D\lambda,D\rho,D\nu^\infty) \in \Y^2(Q;E)$ of $\bm\nu$ at $x_0$. That is, there exists vector $\xi_0 = \xi(x_0)\in Z$ such that
		\begin{align}
		D\lambda & = \lambda^{ac}(x_{0}) \, \Leb^d_Q \in \Tan_1(\lambda,x_0),\\
		D\rho_y & = \Gamma^{\xi_0}_{\#}\rho_{x_0} \quad \text{for $D \lambda$-almost every $y\in \R^d$}.
		\end{align}
		Moreover,  $\{D\nu_{y,\xi}\},\{D\nu^\infty_{y,\xi}\}$ are homogeneous Young measures in the sense that
		\begin{align}
		D\nu_{y,\xi} & =  \nu_{x_0,\xi + \xi_0}\quad \text{for $(\Leb^d \otimes \Leb^d_Z)$-almost every $(y,\xi) \in \R^d \times Z$},\\
		D\nu^\infty_{y,\xi} & = \nu^\infty_{x_0,\xi + \xi_0}\quad \text{ for $(D\lambda \otimes D\rho_y)$-almost every $(y,\xi) \in \R^d \times Z$}.
		\end{align}
	\end{lemma}
	% \mnote{J: This proof is now ready for you, reading it and finding mistakes. I put some effort in the order of arguments to make sure, that taking diagonal sequences etc.~doesn't make problems.}
	
	% \begin{remark}
	%For any $H\in \Crm(Z\times \R^N)$ such that $\chi_\Omega \otimes H \in \E^2(\Omega;E)$ we have:
	%\[
	%   \frac {\di [ \chi_\Omega \otimes T^{\xi_0} H, \bm {\nu} ]} {\di \Leb^d} (x_0) \cdot \Leb^d_Q = [\chi_{Q} \otimes H ,\bm{\sigma} ]\qquad \text{ and in particular } \qquad [\bm{\sigma}] = \frac {\di [\bm{\nu}]}{\di \Leb^d}(x_0) \cdot \Leb^d_Q .
	% \]
	%
	% \end{remark}
	\begin{proof} 
		%Let $\eps_i \searrow 0$ and let $\{\mu_{\varepsilon_i}\}\subset \M(\Omega;\R^N)$  be an approximately $\Acal$-free sequence that generates $\bm \sigma$ on $\cl Q$.
		Let $\{f_m \coloneqq \phi_m \otimes g_m \otimes h_m \,|\,m\in \N\} \subset \E^2(Q;E)$ be the restriction to $\cl Q \times Z \times E$ of the dense subset provided by Lemma~\ref{lem:density}; without loss of generality assume that $g_1 \equiv \chi_Z \in \Crm(Z)$. Let also $\{\xi_k\}_{k \in \Nbb}$ be a countable dense subset of $Z$.

		\emph{Step~1. Selection of regular points.}
		%
		%
		%Then there is, after taking a subsequence, $\xi_0 \in Z$ such that $\Big \langle \frac {x_0}{\varepsilon_k} \Big \rangle \rightarrow \xi_0$. 
		%We recall, that for any $\psi \in \Crm(Q_1)$ and $F\in L^1(\Omega)$, $\Leb^d$-almost every $x_0\in \Omega$ is a \emph{$\psi$-weighted Lebesque point}, that is
		%\[
		% \dashint_{Q_1} \psi(y) \dd y F(x_0)= \lim_{r\searrow 0} \dashint_{Q_r(x_0)} \psi\Big( \frac {x-x_0}r\Big) F(x) \dd x.
		%\]
		%By Lemma \ref{lem:density} we find 
		First, let us define a set of full $\Leb^d$-measure where we aim to show the assertions of the lemma. To do this we first list three Lebesgue-type properties which are satisfied for $\Leb^d$-almost every $\tilde x \in \Omega$: \\
		\textit{Lebesgue property 1.} Lebesgue points of the measure $\lambda$, that is,
		\begin{equation}\label{eq:x_01}
		\lim_{r \todown 0} \frac{\lambda(Q_r(\tilde x))}{r^d} = \lambda^{ac}(x_0), \quad \lim_{r \todown 0} \frac{\lambda^s(Q_r(\tilde x))}{r^d} = 0.
		\end{equation}
		\textit{Lebesgue property 2.} Lebesgue points of the map  
		\begin{align}\label{eq:x_03}
		x \mapsto \int_Z  \langle |\frarg|_E, \nu_{x,\xi} \rangle \dd \xi.
		\end{align}
		\textit{Lebesgue property 3.} Lebesgue points of the family of weighted barycenter measures  $\{\bpr{f_{k,m}, \bm\nu}\}_{k,m \in \Nbb}$ where $f_{k,m} \coloneqq \phi_m \otimes g_m \circ \Gamma^{\xi_k}\otimes h_m$. Recall from~\eqref{eq:density3} that {being a Lebesque point for all elements of the family} is equivalent to 
		\begin{align}\label{eq:x_02}
		\frac{\dd [f_{k,m},\bm\nu]}{\dd \Leb^d}(\tilde x) = \bbpr{f_{k,m},\bm\nu}_{\tilde x}(Z) %\;\; \text{and} \;\; \frac{\dd [f_{k,m},\nu]}{\dd (\lambda^\nu)^s}(\tilde x) = 0
		\quad \text{for all $k,m \in \Nbb$}.
		%\bbpr{h_m} \coloneqq \int_Z g_m \circ \Gamma^{\xi_0})(\xi) \dprb{h_m,\nu_{x,\xi}} \dd \xi + \frac{\dd \lambda^\nu}{\dd \Leb^d}(x) \int_Z g\circ \Gamma^{\xi_0})(\xi) \dprb{h_m,\nu^\infty_{x,\xi}} \dd \rho^\nu_{x}(\xi).
		\end{align}
		%and in particular
		%\[
		%\int_Z g_k(\xi) \dpr{h_m,\nu_{\tilde x + ry,\xi}}  \dd \xi\to \int_Z g_k(\xi) \dpr{h_m,\nu_{\tilde x,\xi}} \dd \xi \quad \text{in $\Lrm^1_\loc(\R^d)$, \quad for all $k \in \Nbb$}.
		%\]
		%By density of the family $\{g_k\}_{k \in \Nbb}$ in $\Crm(Z)$, this entails the convergence
		%\[
		%\int_Z g(\xi) \dpr{h_m,\nu_{\tilde x + ry,\xi}}  \dd \xi\to \int_Z g(\xi) \dpr{h_m,\nu_{\tilde x,\xi}} \dd \xi \quad \text{in $\Lrm^1_\loc(\R^d)$, \quad for all $g \in \Crm(Z)$}.
		%\]
		We shall show the conclusions of the lemma hold for all 
		\[
		x_0 \in R \coloneqq \set{\tilde x \in \Omega}{\text{$\tilde x$ satisfies \eqref{eq:x_01}-\eqref{eq:x_02}}},
		\]
		which is a set of full $\Leb^d$-measure in $\Omega$. 

		\emph{Step~2. Blow-up sequence.} As before, we write $\langle z \rangle$ to denote the equivalence class of a vector $z \in \R^d$ in the $d$-dimensional torus $Z$. Since $Z$ is a compact manifold we may assume (up to passing to a  subsequence $(\mu_{\eps_i})_{i \in \Nbb}$) that 
		\begin{equation}\label{eq:limitpoint}
		\Big\langle \frac{x_0}{\eps_i} \Big \rangle \to  \xi_0 \in Z \quad \text{as $i \to \infty$}.
		\end{equation}
		Let %$r_0 < \dist (\partial \Omega, \{x_0\})$ and fix 
		$r_j \todown 0$ (with $r_1 = 1$) be an infinitesimal sequence of radii and consider, for fixed $j \in \Nbb$, % $r_j \leq r_0$.
		the blow-up sequence
		\[
		\gamma_{\delta_{i,j}} \coloneqq \frac 1{r_{j}^{d}}T_\#^{(x_0,r_{j})} \mu_{\eps_{i}} \in \M(Q;E), \quad i \in \Nbb,
		\]
		where  
		\[
		\delta_{i,j} \coloneqq \frac{\eps_{i}}{r_{j}},\quad i \in \Nbb,
		\]
		is the readjusted blow-up length-scale sequence (this conforms again an infinitesimal sequence).
		%where we have set $\delta_{i,j} \coloneqq \eps_i/r_j$.
		Since $(\mu_{\eps_i})_{i \in \Nbb}$ is uniformly bounded in $\M(\Omega;\R^N)$, we also have
		\[
		\sup_{i \in \Nbb}\gamma_{\delta_{i,j}}(Q) % = \sup_{i \in \Nbb}  \mu_{\eps_= \sup_{i \in \Nbb} \frac{1}{r_j}\mu_{\eps_i}\big(Q_{r_jR}(x_0)\big) 
		< \infty
		\quad \text{for each $j \in \Nbb$.} %and all $\ell \in \Nbb$} .  
		\]
		%\Jnote{In your annotations you wanted to include that $Q_R\subset \Omega$. Right now we are working with this Preiss-noatation where the blow up is just continued with zero outside of $\Omega$ but the limiting measure is a local measure in $\R^d$. This is why I do  not include it. If we decide that we do not need this, we may rewrite this and the singular part accordingly}
		For $j = 1$, we use the compactness result in Theorem~\ref{thm:representation} to find a subsequence $\{1(i)\} \subset \{i\}$ 
		%(of $\{\delta_i\}_{i \in \Nbb}$) 
		%$(\delta_{1(i)),1})_{p \in \Nbb}$ 
		and $\bm{\sigma}^{(1)} \in \Y^2(Q;E)$ 
		such that 
		\[
		\gamma_{\delta_{1(i),1}} \toYY \bm{\sigma^{(1)}} \quad \text{on $\cl Q$ \quad as \quad $i \rightarrow \infty$.}
		\]
		Recursively, for each $2 \le j \in \Nbb$, we may find a sequence $\{j(i)\}\subset \{(j-1)(i)\}$ and a two-scale* Young measure $\bm{\sigma^{({j})}}\in \Y^2(Q;E)$ such that
		\[
		\gamma_{\delta_{j(i),j}} \toYY \bm{\sigma^{({j})}} \qquad \text{ as \quad $i \to \infty$.}
		\]
		
		%$\delta_{k,j} = \varepsilon_k /r_j$ and the functions
		% $v_{\delta_{k,j}}^{r_j}:Q^d\rightarrow \R$ by $v_{\delta_{k,j}}^{r_j}(y) = u_{r_j \cdot \delta_{k,j}}(x_0 +r_j y)$, who are uniformly bounded in $L^1(Q^d)$.
		%Then by Theorem \ref{thm:classicalcompactness}, applied for each $r_j$, there is $\bm{\sigma_{r_j}} \in \Y^2(Q;\R^N)$, a subsequence $\delta_m$ (successively choosen for all $r_j$) such that $v^{r_j}_{\delta_m} \toYY \bm{\sigma_{r_j}}$.
		
		\emph{Step 3. Characterization of $\bm{\sigma^{({j})}}$.}
		In this step we fix $j \in \Nbb$. %We then define %Consider $\bm{\sigma}^{(r_{j})}$ as an element in $\Y^2(Q;\R^N)$ and 
		Let $f = \phi \otimes g \otimes h \in \E^2(Q;\R^N)$ with $g \in \Crm^1(Z)$ and $h$ uniformly Lipschitz. %For the sake of simplicity, let us for now write $\Tcal = [T^{(x_0,r_j)}]^{-1}$. 
		A change of variables and the decomposition in~\eqref{eq:blow_r} yields 
		\begin{equation}\label{eq:loctest1}
		\begin{split}
		\ddprb{f, \bm{\sigma^{(j)}}} & = \lim_{i \to \infty}\bigg(
		\int_Q \phi(y)\,g(\,y/\delta_{j(i),j}\,)\,h(\,\mu_{\eps_{j(i)}}^{ac}(x_0 + r_jy)\,) \dd y \\
		&+ \int_{\cl Q}  \phi(y)g\bigg(\frac{y}{\delta_{j(i),j}}\bigg)h^\infty\bigg(\frac{\dd \mu_{\eps_{j(i)}}}{\dd |\mu^s_{\eps_{j(i)}}|}(x_0 + r_jy)\,\bigg) \dd \big( r_j^{-d}
		T_\#^{(x_0,r_j)}|\mu^s_{\eps_{j(i)}}|)(y)\bigg) \\
		& =  \frac 1{r_j^{d}}\lim_{i \to \infty} \bigg(\int_{Q_{r_j}(x_0)} \big(\varphi\circ T^{(x_0,r_j)}\big)(x)  \, g\bigg(\frac{x - x_0}{\eps_{j(i)}}\bigg)h(\,{\dd \mu^{ac}_{\eps_{j(i)}}}(x)\,) \dd x \\
		&+  \int_{\cl{Q_{r_j}(x_0)}}  \big(\varphi\circ T^{(x_0,r_j)}\big)(x)  \,g\bigg(\frac{x - x_0}{\eps_{j(i)}}\bigg)h^\infty\bigg(\frac{\dd \mu_{\eps_{j(i)}}}{\dd |\mu^s_{\eps_{j(i)}}|}(x)\bigg) \dd|\mu^s_{\eps_{j(i))}} |(x)\bigg).
		% =& \lim_{m \rightarrow \infty} \int_{Q^d} \phi (y) g(y/\delta_m) h(v^{r_j}_{\delta_m} (y)) \dd y  \\
		% =&\lim_{m \rightarrow \infty}{r_j}^{-d} \int_{Q_{r_j}(x_0)} \phi \Big(\frac{x-x_0} {r_j}\Big) g\Big( \frac x {\varepsilon_m}-\frac {x_0} {\varepsilon_m} \Big) h(u_{\eps_m}(x)) \dd x.
		\end{split}
		\end{equation} 
		%Hence, letting $i \to \infty$ we get
		%\begin{align*}
		%\ddprb{\phi \otimes g \otimes h , \bm{\sigma}^{(r_j)}} = \ddprb{\phi \circ T^{r_j,x_0} \otimes \Gamma_{\xi_0} g \otimes h, \bm{\nu}},
		%\end{align*}
		%where we have used that we can estimated the first summand of the last equality in~\eqref{eq:loctest1}
		Setting $C \coloneqq \limsup_{\eps \todown 0}|\mu_{\eps}|(\overline{Q}) < \infty$ (from the original sequence) we may estimate the limiting behavior of the difference 
		\begin{align*}%\label{eq:xi0appears1}
		\bigg{|} \int_{\cl{Q_{r_j}(x_0)}} \phi \bigg(\frac{x-x_0} {r_j}\bigg)\bigg[g\bigg( \frac{x- x_0}{\varepsilon_{j(i)}}\bigg) - g\bigg(\frac{x}{\varepsilon_{j(i)}} - \xi_0\bigg)\bigg] \dd |\mu_j|(x) \bigg|
		%  \leq \|\phi\|_\infty \|Dg\|_\infty |\mu_\eps|\big(\cl Q_r(x_0)\big) \Big{|}\floorB{\frac{x_0}{\varepsilon_m}}  -\xi_0 \Big{|} %\label{eq:estxi0}
		\end{align*}
		by
		\begin{align} \label{eq:xi0appears}
		C \|\phi\|_\infty \cdot \|Dg\|_\infty%\; \displaystyle \intbar_{Q_{r_j}(x)} \bigg|\frac{\dd \mu_j}{\dd \Leb^d}(x)-\frac{\dd \mu_j}{\dd \Leb^d}(x_0)\bigg| \dd x 
		\cdot \bigg|\dprBB{\frac{x_0}{\eps_{{j(i)}}}} - \xi_0\bigg| = \BigO(i) \to 0 \quad \text{as $i \to \infty$},
		\end{align}
		where to see that the last term vanishes as $i \to \infty$ we have used~\eqref{eq:limitpoint}.
		
		Using this  we may re-write~\eqref{eq:loctest1} as 
		\begin{equation}\label{eq:loc2}
		\begin{split}
		\ddprb{f , \bm{\sigma^{(j)}}} & = \frac 1{r_j^{d}} \lim_{i \to \infty} \bigg(\int_{Q_{r_j}(x_0)} (\varphi \circ T^{(x_0,r_j)})(x) \, (g \circ \Gamma^{\xi_0}))\bigg(\frac{x}{\eps_{j(i)}}\bigg) \, h(\,{\dd \mu^{ac}_{\eps_{j(i)}}}(x)\,) \dd x \\
		& + \int_{\cl{Q_{r_j}(x_0)}}  (\varphi \circ T^{(x_0,r_j)})(x) (g \circ \Gamma^{\xi_0}))\bigg(\frac{x}{\eps_j}\bigg)h^\infty\bigg(\frac{\dd \mu_{\eps_{j(i)}}}{\dd |\mu^s_{\eps_{j(i)}}|}(x)\bigg) \dd|\mu^s_{\eps_{j(i)}}| (x)\bigg) \\
		& = r^{-d}_j\ddprb{\varphi \circ T^{(x_0,r_j)} \otimes  g \circ \Gamma^{\xi_0} \otimes h, \bm{\nu}},
		\end{split}
		\end{equation}
		where in passing to the last equality we have used that $\mu_{\eps} \toYY \bm{\nu}$ in $\Y^2_\loc(\R^d;E)$.
		Applying  this to $|\frarg|_{E}$ yields, together with (\ref{eq:x_01}) and (\ref{eq:x_03}),
		\begin{align}\label{eq:tris}
		\sup_{j \in \Nbb} \ddprb{\phi \otimes g \otimes |\frarg|_{E} , \bm{\sigma^{(j)}}} 
		%=  \sup_{j \in \Nbb}\Big| r_j^{-d} \left( \int_{Q_{r_j}(x_0)} \int_Z  \langle |\frarg|, \nu_{x,\xi} \rangle \dd \xi \dd x + \lambda^\nu(Q_{r_j}(x_j)) \right) \Big|,
		< \infty \quad \text{for all $\phi \in \Crm_c(\R^d)$ and $g \in \Crm^1(Z)$}\,.
		\end{align}
		We are then in position to apply the following global version of Corollary \ref{cor:Compactnessofyoungmeasures} (whose proof relies on a localization argument): there exists  a subsequence of $(r_j)_{j \in \Nbb}$ (not relabeled) satisfying
		\[
		\text{$\bm{\sigma^{(j)}} \toweakstar \bm{D\nu(x_0)}$ \; in \; $\E^2(Q;E)^*$,  \; for some $\bm{D\nu(x_0)}\in  \Y^2(Q;E)$}\,.
		\]

		%which converges to zero as $m\rightarrow \infty$. 
		%We define $T^{x_0,{r_j}}: \Crm_0(Q^d)\rightarrow \Crm(\Omega)$ by $T^{x_0,{r_j}}\phi(x) = \phi \Big( \frac {x-x_0}{{r_j}} \Big)$, 
		%extended by zero.
		
		\emph{Step 4: Characterization of $\bm{\sigma}$.} Fix $m \in \Nbb$ and let $j \in \Nbb$ be an arbitrary positive integer. 
		Let us write $\tilde f_m \coloneqq \phi_m \otimes g_m \circ \Gamma^{\xi_0} \otimes h_m$. For an arbitrary positive real number $\eta > 0$, we may use the uniform continuity of $g_m$ (recall that $g_m \in \Crm^1(Z)$) and the density of the set of points $\{\xi_k\}_{k \in \Nbb}$ in $Z$ to find a sufficiently large $k = k(\eta)$ with the following property (here we use the positivity of $\phi_m$ and $h_m$):
		\begin{gather*}
		|\xi_0 - \xi_{k}| = \BigO(\eta),  
		\end{gather*}
		and
		\begin{align*}
		|\tilde f_m - f_{k,m}|(x,\xi,z) & \le \|g_m \circ \Gamma^{\xi_0} - g_m \circ \Gamma^{\xi_k} \|_\infty \cdot \phi_m(x)\cdot h_m(z) \\
		& \le \BigO(\eta) \, [{\varphi \otimes \chi_Z \otimes |\frarg|_E(x,\xi,z)}],
		\end{align*}
		where $\BigO(\eta) \to 0$ as $\eta$ tends to zero. In particular,
		\begin{align*}
		\limsup_{j \in \Nbb} \; & \Big|{r_j}^{-d} \bpr{\tilde f_m,\bm\nu}(\cl{Q_{r_j}(x_0)}) -  {r_j}^{-d} \bpr{f_{k(\eta),m},\bm\nu}(\cl{Q_{r_j}(x_0)})\Big| \\
		% \int_Z & |\dpr{\tilde f_m - f_{k(\gamma),m},\nu_{x_0 + r_jy}}|\dd\xi\dd y  \\
		& \le \BigO(\eta) \cdot \|\varphi_m\|_\infty \cdot \textrm{Lip}(h_m) \cdot \limsup_{j \in \Nbb} r_j^{-d}\bpr{\phi_m \otimes \chi_Z \otimes |\frarg|_E,\bm{\nu}}(\cl{Q_{r_j}(x_0)}) \\
		&  \lesssim_{(m)} \BigO(\eta) \cdot \liminf_{j \to \infty} \bigg( 1 + \frac{\lambda(\cl Q_{r_j}(x_0))}{r_j^d} \bigg).
		\end{align*}

		Testing~\eqref{eq:loc2} with $f_m$ (not to be confused with $\tilde f_m$), it follows from the estimate above and~\eqref{eq:x_01} that %after a change of variables of the form $\{y \mapsto x_0 + r_jy\}$ we get
		\begin{equation}\label{eq:x_04}
		\begin{split} 
		\ddprb{f_m, \bm{\sigma^{(j)}}} 
		& = {r_j}^{-d} \ddprb {\phi_m \circ T^{(x_0,r_j)} \otimes g_m \circ \Gamma^{\xi_0} \otimes h_m, \bm{\nu}}  \\
		& = {r_j}^{-d} \bpr{\tilde f_m,\bm\nu}(\cl Q_{r_j}(x_0)) \\
		% & = \int_{Q}  \int_Z \dpr{f_{k(\gamma),m},\nu_{x_0 + ry,\xi}} \dd \xi \dd y + \SmallO(\gamma)\\ 
		% & \qquad + \int_{\cl Q} \int_Z \dpr{f_{k(\gamma),m}^\infty,\nu_{x_0 +ry,\xi}^\infty} \dd \rho_{x_0 + ry}^\nu(\xi) \dd (r_j^{-d}T^{(x_0,r_j)}_\#\lambda^\nu)(y) \\
		% & =  \intbar_{Q_{r_j}(x_0)} \varphi \circ T^{(x_0,r_j)}(x) H_m(x) \dd x \\
		% =& {r_j}^{-d} \int_\Omega T^{x_0,{r_j}}\phi (x )  \frac{ \di[ \chi_\Omega \otimes T^{\xi_0} g \otimes h, \bm{\nu}]}{\di \Leb^d}(x)\dd x \\ &
		& = {r_j}^{-d} \bpr{f_{k(\eta),m},\bm\nu}(\cl Q_{r_j}(x_0)) + \BigO(\eta). %\qquad \supp \phi_m \subset Q_\rho.
		%& \qquad +{r_j}^{-d}  \int_{\cl{Q_{r_j}(x_0)}}  T^{x_0,{r_j}} \phi_m(x) \int_Z \big(g_m\circ \Gamma^{\xi_0})\big)(\xi) \dprb{h^\infty_m,\nu_{x,\xi}}\dd \rho_x^\nu(\xi) \dd (\lambda^\nu)^s(x) \\
		% & \rightarrow  \int_Q\phi(y) H_m(x_0) \dd y \qquad \text{ as $j\rightarrow \infty$,} 
		%  \frac{ \di[ T^{x_0,{r_j}}\phi \otimes T^{\xi_0} g \otimes h, \bm{\nu}]}{\di (\lambda^\nu)^s}(x)\dd (\lambda^\nu)^s( x ).
		\end{split}
		\end{equation}
		%Indeed, the last equality follows from the estimate 
		%
		%and~\eqref{eq:x_01}.
		%This relates to first summand. On the other hand,
		%\begin{align*}
		%\sup_{j \in \Nbb} \int_{\cl Q}  \int_Z |\dpr{\tilde f^\infty_m - & f^\infty_{k(\gamma),m},\nu^\infty_{x_0 + r_jy}}|\dd \rho_{x + ry}^\nu(\xi)\dd \lambda^\nu(y) \\
		%& \le \gamma \sup_{j \in \Nbb} \dashint_{\cl{Q_{r_j}(x_0)}} \dpr{f^\infty_{1,m},\upsilon_{x}}\dd \lambda^\nu(x) \stackrel{\eqref{eq:x_02}} \coloneqq \gamma \cdot c_2(m).
		%\end{align*}
		%A careful inspection of the Radon--Nykodym derivative and~\eqref{eq:Lebesquepointsforfm} give that $x_0$ is also a Lebesgue point of the weighted barycenter $[\tilde f_m,\nu]$, and in fact,
		%\[
		%\frac{\dd [\tilde f_m,\nu]}{\dd \Leb^d}(x_0) = \bbpr{\tilde f_m,\nu}_{x_0}(Z) \quad \text{and} \quad \frac{\dd [\tilde f_m,\nu]}{\dd (\lambda^\nu)^s}(x_0) = 0.
		%\]
		%Indeed, 
		%where we used~\eqref{eq:choicex0no1} .
		Letting $j \to \infty$ at both sides of~\eqref{eq:x_04},~\eqref{eq:x_02} and the weak-$*$ convergence $\bm{\sigma^{(j)}} \toweakstarY \bm{D \nu(x_0)}$ imply 
		%in~\eqref{eq:characterizationofsigmar} together with~\eqref{eq:Lebesquepointsforfm} 
		\begin{align*}
		\ddprb{f_m,\bm{D \nu(x_0)}}  = \lim_{j \to \infty}  \ddprb{f_m, \bm{\sigma^{(j)}}}  = \bbpr{f_{k(\gamma),m},\nu}_{x_0}(Z) + \BigO(\eta).
		%& = \int_Q \varphi(y) H_m(x_0) \dd y %\\
		%& = \int_Q \int_Z \phi_m(y) \big(g_m \circ \Gamma^{\xi_0}) \big)(\xi)\dprb{h_m,\nu_{x_0,\xi}} \dd \xi \dd y \\ 
		%& \qquad + \frac {\dd \lambda^\nu}{\dd \Leb^d}(x_0)\int_{\cl Q} \int_Z \phi_m(y) \big(g_m \circ \Gamma^{\xi_0})\big)(\xi)\dprb{h_m^\infty,\nu_{x_0,\xi}^\infty} \dd \rho_{x_0}^\nu(\xi) \dd y.
		\end{align*}
		Following analogous arguments to the ones in~\eqref{eq:x_04} and using $|\xi_0 - \xi_{k(\gamma)}| = \BigO(\gamma)$, we may let $\eta \searrow 0$ to deduce\begin{equation}\label{eq:x_05}
		\begin{split}
		\ddprb{f_m,\bm{D \nu(x_0)}} & = \bbpr{\tilde f_m,\nu}_{x_0}(Z) \\
		& = \int_Q\bigg(\int_Z \dpr{f_m,\nu_{x_0,\xi_0 + \xi}} \dd(\Gamma^{\xi_0}_\#\Leb^d_Z)(\xi) \\
		& \qquad + \frac{\dd \lambda^\nu}{\dd \Leb^d}(x_0)\int_Z\dpr{f_m^\infty,\nu_{x_0,\xi_0 + \xi}^\infty} \dd(\Gamma^{\xi_0}_\#\rho_{x_0})(\xi)\bigg)\dd y.
		\end{split}
		\end{equation}
		The sought assertion follows from the arbitrariness of $m$ on Step~4 (see~\eqref{eq:x_05}) and Lemma~\ref{lem:density} which translates into the equivalence
		\[
		\sigma \equiv \Big(\nu_{x_0,\xi_0 + \xi},\lambda^{ac}(x_0) \, \Leb^d,\Gamma^{\xi_0}_\#\rho_{x_0},\nu_{x_0,\xi_0 + \xi}^\infty\Big)_{y \in \cl Q, \xi \in Z}
		\]
		as two-scale* Young measures in the set $\Y^2(Q;\R^N)$. Notice we have used that the $d$-dimensional Lebesgue is uniformly distributed and hence $\Gamma_\#^{\xi_0} \Leb_Z^d \equiv \Leb_Z^d$. This proves the desired result.
	\end{proof}
	\subsection{Localization at singular points}
	
	\begin{proposition}\label{prop:loc_singular}Let $\bm{\nu} = (\nu,\lambda,\rho,\nu^\infty) \in \Y^2(\Omega;E)$ be two-scale* Young measure which is generated by a sequence of measures $(\mu_{\eps_i})_{i \in \Nbb} \subset \M(\Omega;\R^N)$.
		Then, there exists a set $S \subset \Omega$ with full $\lambda^s$-measure that satisfies the following property: 
		at every $x_0 \in S$ there exists (up to a translation in $Z$) a local tangent two-scale* Young measure $\bm{D\nu}= (D\nu,D \lambda,D \rho,D\nu^\infty)\in \Y^2(Q;\R^N)$ of $\bm\nu$ at $x_0$. That is, there exists $\xi_0 = \xi(x_0) \in Z$ such that
		\begin{align}
		D \lambda & \in \Tan_1(\lambda^s,x_0)\,, \\ %\quad \text{with $\lambda^\sigma(Q)=1$ and $\lambda^\sigma(\partial Q)=0$},\\
		D\rho_y & = \Gamma^{\xi_0}_{\#}\rho_{x_0} \quad \text{for $D \lambda$-almost every $y\in Q$}\,.
		%\sigma_{y,\xi} & = \delta_0 \quad \text{for $\Leb^d \otimes \Leb^d$-almost every $y \in Z$}
		\end{align}
		Moreover,  $D\nu$ is concentrated on zero and $D\nu^\infty$ is homogeneous in the sense that
		\begin{align}
		D\nu_{y,\xi} & =  \delta_0 \qquad \quad \text{for $(\Leb^d_Q \otimes \Leb^d_Z)$-almost every $(y,\xi) \in Q \times Z$},\\
		D\nu^\infty_{y,\xi} & = \nu^\infty_{x_0,\xi + \xi_0} \quad  \text{for $\big(D \lambda \otimes D\rho_y\big)$-almost every $(y,\xi) \in Q \times Z$}.
		\end{align}
		%\begin{enumerate}[(i)]
		%\item $\lambda^\sigma \in \Tan((\lambda^\nu)^s,x_0)$,\quad $\lambda^\sigma(Q)=1$, \quad $\lambda^\sigma(\partial Q)=0$,
		%\item $\rho_y^\sigma = \Gamma^{\xi_0}_\#(\rho_{x_0}^\nu)$,
		%\item $\sigma_{y,\xi} = \delta_0$ \quad $\Leb^d \otimes \Leb^d$-almost everyhwere in $Q \times Z$, \quad and 
		%\item $\sigma^\infty_{y,\xi} = \nu^\infty_{x_0,\xi + \xi_0}$ \quad $\lambda^\sigma \otimes \rho_{y}^\sigma$-almost everyhwere in $Q \times Z$.
		%\end{enumerate}
		% Moreover, there is  $z_{\delta_\ell} \in C^\infty_{per}(Q,\R^n)\cap \operatorname{ker}(\Acal^k)$  such that $z_{\delta_\ell} \toYY \bm{\sigma}$ in $\Y^2(Q;\R^N)$.
	\end{proposition}
	
	\begin{proof}
		The structure of the proof is very similar to the one of the localization principle for regular points. The main difference lies in the scaling of the blow up and, as a consequence of that, different terms  vanish.

		Let $\eps \searrow 0$ and $(\mu_{\eps}) \subset \Mcal (\overline {\Omega};E)$ be such that $\mu_j \toYY \bm{\nu}$. We select a countable dense family $\{\xi_k\}_{k \in \Nbb}$ in $Z$ and we write $\{f_m = \phi_m \otimes g_m \otimes h_m\,|\, m\in \N \} \subset \E^2(Q,E)$ to denote the restriction to $\cl Q \times Z \times E$ of the dense subset introduced in Lemma \ref{lem:density}. Additionally, this time we will assume without loss of generality that $h_1 \equiv \chi_{E}$. Recall we may assume the $h_m$'s to be Lipschitz continuous. 
		
		\emph{Step~1. Selection of the singular set $S$.} We shall consider points $\tilde x \in \Omega$ satisfying the density estimate
		%$f(x) =1 + \int_Z \langle|\frarg|,\nu_{x, \xi}\rangle \dd \xi+ \frac{\dd \lambda}{\dd \Leb^d}(x)$ we have for $\lambda^s$ a.e.~$x_0\in\Omega$
		\begin{gather} \label{eq:x01}
		\lim_{r\todown 0}\displaystyle \frac { r^d + \int_{Q_r(\tilde x)} \int_Z \dpr{|\frarg| \dd \xi,\nu_{x,\xi}} \dd \xi + \int_{Q_r(\tilde x)} \lambda^{ac}(x) \dd x}{\lambda^s(Q_r(\tilde x))} =0\,,
		%\\
		%\lim_{r\todown 0} \frac {\int_Q \int_Z \dpr{|\frarg| \dd \xi,\nu_{x,\xi}} \dd \xi}{(\lambda^\nu)^s(Q_r(x_0))} =0, \\ 
		%\lim_{r \todown 0} \frac { (\lambda^\nu)^{ac}(Q_r(x_0))}{(\lambda^\nu)^s(Q_r(x_0))}=0.
		\end{gather}
		and which are
		\begin{equation}
		\text{$\lambda^s$-Lebesgue points of $\bpr{f_{k,m},\bm\nu} \in \M(\cl\Omega)$ \quad for all $k,m \in \Nbb$}\,.
		\end{equation} Here, these barycenter measures are parametrized by the $x$-homogenous integrands 
		\[
		f_{k,m} \coloneqq \chi_\Omega \otimes g_m\circ \Gamma^{\xi_k} \otimes h_m \quad k,m \in \Nbb.
		\]
		In particular, by~\eqref{eq:density3} and the Lipschitz continuity of the $h_m$'s, at such points $\tilde x \in \Omega$ it holds
		\begin{align}\label{eq:x02}
		\frac{\dd [f_{k,m},\bm\nu]}{\dd \lambda^s}(\tilde x) = \bbpr{f_{k,m},\bm\nu}_{\tilde x}(Z) = \bbpr{f_{k,m}^\infty,\bm\nu}_{\tilde x}(Z) 
		%\;\; \text{and} \;\; \frac{\dd [f_{k,m},\nu]}{\dd (\lambda^\nu)^s}(\tilde x) = 0\quad 
		\quad \text{for all $k,m \in \Nbb$}.
		%\bbpr{h_m} \coloneqq \int_Z g_m \circ \Gamma^{\xi_0})(\xi) \dprb{h_m,\nu_{x,\xi}} \dd \xi + \frac{\dd \lambda^\nu}{\dd \Leb^d}(x) \int_Z g\circ \Gamma^{\xi_0})(\xi) \dprb{h_m,\nu^\infty_{x,\xi}} \dd \rho^\nu_{x}(\xi).
		\end{align}
		For the rest of the proof we fix a point $x_0 \in S \coloneqq \setb{\tilde x \in \Omega}{\text{$\tilde x$ satisfies \eqref{eq:x01}-\eqref{eq:x02}}}$, which is a set of full $\lambda^s$-measure in $\Omega$.
		
		\proofstep{Step~2. Blow-up sequence.} Since $Z$ is a compact manifold, we may again restrict to a subsequence of the generating sequence $(\mu_{\eps_i})_{i \in \Nbb}$  and find $\xi_0 = \xi_0(x_0) \in Z$ such that
		\[
		\Big\langle \frac{x_0}{\eps_i} \Big \rangle \to \xi_0 \in Z \quad \text{as $i \to \infty$}.
		\]
		%Let us  further assume that $x_0$ satisfies: for every $m \in \Nbb$, $x_0$ is a $(\lambda^\nu)^s$-Lebesgue point of the functions
		%\begin{align}\label{eq:Lebesquepointsforfmsing}
		%x \mapsto H_m(x)  =  \int_Z g\circ \Gamma^{\xi_0})(\xi) \dprb{h_m,\nu^\infty_{x,\xi}} \dd \rho^\nu_{x}(\xi).
		%\end{align}
		%Additionally we have by \cite{Pre87} that for $(\lambda^\nu)^s$-a.e. $x_0\in \Omega$
		%
		%\begin{align} \label{eq:Preisstyleequation}\textcolor{blue}{
		% \sup_{r>0} \frac{(\lambda^\nu)^s(x_0+rK) +\Leb^d(x_0+rK)}{(\lambda^\nu)^s(Q_r(x_0))} < \infty \quad \text{ for every }K \subset \subset \R^d.}
		%\end{align}
		%
		%\Jnote{ This is blue, due to two reasons: It is only needed if one wants to get the Preiss-formulation  with local measures on $\R^d$ instead of measures on $\bar Q$, which I am not sure if we want to do this. On the other hand this is (I think) not exactly what we will require when applying it for compactness (see the blue part in the proof below)}
		%All these properties are satisfied for $(\lambda^\nu)^s$-almost every $x_0 \in \Omega$. In the following, $x_0$ is fixed.
		
		For a positive radius $r>0$,  we set $c_r  \coloneqq\lambda^s(Q_{r}(x_0))^{-1}$. By the compactness properties of measures we may find a weak* convergent blow-up sequence in the sense there exists a sequence of infinitesimal radii $(r_j)_{j \in \Nbb}$ such that
		\begin{equation}\label{eq:Lambda}
		\Lambda_j \coloneqq c_{r_j} T^{(x_0,r_j)}_\# \lambda^s \toweakstar D\lambda \in \Tan_1(\lambda^s,x_0) \quad \text{and} \quad \gamma(\cl Q) = 1. 
		\end{equation}
		%There is a sequence $r_j \searrow 0$, $\gamma\in \M(\cl Q)$ such that $(\lambda^\nu)^s(\partial Q_{r_j})(x_0) =0$, $c_{r_j} T_\#^{(x_0,r_j)} (\lambda^\nu)^s \toweakstar \gamma$ in $\M(\cl Q)$ and $\gamma (\partial Q)=0$.
		%Since $\lambda^s \geq 0$ we immediately get that $\gamma (Q)=1$.
		For the sake of simplicity let us write $c_j \coloneqq c_{r_j}$. We consider, for fixed $j \in \Nbb$, the $i$-indexed sequence 
		\[
		\gamma_{i,j} \coloneqq c_{j_0}T_\#^{(x_0,r_{j_0})} \mu_{\eps_i}, \quad \text{where \;
			$\delta_{i,j} \coloneqq \frac{\eps_i}{r_{j}}$} \searrow 0\quad \text{as $i \to \infty$}\,.
		\]
		After an iterative procedure as the one for regular points, for each $j \in \Nbb$ we may find a sequence $j(i)$ with the following properties: $\{(j+1)(i)\}_{i \in \Nbb} \subset \{j(i)\}_{i \in \Nbb}$ for all $j \in \Nbb$, and 
		\[
		\gamma_{j(i),j} \toYY \bm{\sigma^{({j})}} \; \text{for some $\bm{\sigma^{({j})}} \in \Y^2(Q;E)$}, \qquad j \in \Nbb.
		\]
		Moreover, up to passing to a subsequence of $\{j\}_{j \in \Nbb}$, we may assume there exists $\bm{D\nu(x_0)} \in \Y^2(Q;E)$ for which
		\begin{equation}\label{eq:limiting}
		\bm{\sigma^{({j})}} \toweakstar \bm{D\nu(x_0)} \quad \text{in $\E^2(Q;\R^N)^*$}.
		\end{equation}
		
		%As above, this is uniformly bounded for fixed $j$ and due to Theorem \ref{thm:classicalcompactness} we may successively chose subsequences of the $\eps_i$ such that
		%\[
		% \gamma_{i^{(j)},j} \toYY \bm{\sigma}^{r(r_j)} \qquad \text{ as $ i^{(j)} \rightarrow \infty$} \qquad \text{ for some $\bm{\sigma}^{r(r_j)} \in \Y^2_{\loc}(\R^d;\R^N)$ .}
		%\]
		\emph{Step 3: Characterization of $\bm{\sigma^{({j})}}$.} %Since again  first $i$ goes to infinity for a fixed $r_j$ we write $i^{(j)} = i$. Consider $\bm{\sigma}^{(r_j)} \in \Y^2(Q,\R^N)$.
		For fixed $\phi \otimes g \otimes h \in \Crm(\overline {Q}) \times \Crm^1(Z) \times \E^0(\R^N)$ we deduce, by the same  change of variables used in (\ref{eq:loctest1}), that 
		\begin{equation}
		\begin{split}
		\ddprb{f, & \bm{\sigma^{(j)}} }  = \lim_{i \to \infty}\bigg(
		\int_Q \phi(y)\,g({y}/{\,\delta_{j(i)}}\,)\,h\big(\,c_jr_j^d\mu^{ac}_{\eps_{j(i)}}(x_0 + r_jy)\,\big) \dd y \\
		&+ \int_{\cl Q}  \phi(y)g\bigg(\frac{y}{\delta_{j(i)}}\bigg)h^\infty\bigg(\frac{\dd \mu_{\eps_{j(i)}}}{\dd |\mu^s_{\eps_{j(i)}}|}(x_0 + r_jy)\bigg) \dd \big(c_{j}
		T_\#^{(x_0,r_j)}|\mu^s_{\eps_{j(i)}}|)(y)\bigg) \\
		& =  \lim_{i \to \infty} \frac 1{r_j^{d}}\bigg(\int_{Q_{r_j}(x_0)} \big(\varphi\circ T^{(x_0,r_j)}\big)(x)  \, g\bigg(\frac{x - x_0}{\eps_{j(i)}}\bigg)h\big(c_jr_j^d{\mu^{ac}_{\eps_{j(i)}}}(x)\big) \dd x \\
		&+  \int_{\cl{Q_{r_j}(x_0)}}  \big(\varphi\circ T^{(x_0,r_j)}\big)(x)  \,g\bigg(\frac{x - x_0}{\eps_{j(i)}}\bigg)h^\infty\bigg(\frac{\dd \mu_{\eps_{j(i)}}}{\dd |\mu^s_{\eps_{j(i)}}|}(x)\bigg) \dd(c_jr_j^d|\mu^s_{\eps_{j(i)}} |)(x)\bigg).
		% =& \lim_{m \rightarrow \infty} \int_{Q^d} \phi (y) g(y/\delta_m) h(v^{r_j}_{\delta_m} (y)) \dd y  \\
		% =&\lim_{m \rightarrow \infty}{r_j}^{-d} \int_{Q_{r_j}(x_0)} \phi \Big(\frac{x-x_0} {r_j}\Big) g\Big( \frac x {\varepsilon_m}-\frac {x_0} {\varepsilon_m} \Big) h(u_{\eps_m}(x)) \dd x.
		\end{split}
		\end{equation} 
		Making use of~\eqref{eq:x01} and the estimate~\eqref{eq:xi0appears} as it was used in~\eqref{eq:loc2}, and the fact that $\mu_{\eps_{j(i)}} \toYY \bm{\nu}$, we can re-write the limit on the right hand side in terms of a weighted barycenter measure as (here we use that $h^\infty(c \frarg) = ch^\infty$ for all $c \ge 0$)
		\begin{equation}\label{eq:sloc2}
		\begin{split}
		\ddprb{f , & \bm{\sigma^{(j)}}}  =  \frac 1{r_j^{d}}\lim_{i \to \infty} \bigg(\int_{Q_{r_j}(x_0)} (\varphi \circ T^{(x_0,r_j)})(x) \, (g \circ \Gamma^{\xi_0}))\bigg(\frac{x}{\eps_{j(i)}}\bigg) \, h\big(c_jr_j^d{\mu^{ac}_{\eps_{j(i)}}}(x)\big) \dd x \\
		& + \int_{\cl{Q_{r_j}(x_0)}}  (\varphi \circ T^{(x_0,r_j)})(x) (g \circ \Gamma^{\xi_0}))\bigg(\frac{x}{\eps_{j(i)}}\bigg)h^\infty\bigg(c_jr_j^d\frac{\dd \mu_{\eps_{j(i)}}}{\dd |\mu^s_{\eps_{j(i)}}|}(x)\bigg) \dd|\mu^s_{\eps_{j(i)}}| (x)\bigg) \\
		& = r_j^{-d}[\varphi \circ T^{(x_0,r_j)} \otimes  g \circ \Gamma^{\xi_0} \otimes h(c_jr_j^d \frarg),\bm{\nu}](\cl{Q_{r_j}(x_0)}).%r^{-d}_j\ddprb{\varphi \circ T^{(x_0,r_j)} \otimes  g \circ \Gamma^{\xi_0} \otimes h, \bm{\nu}},
		\end{split}
		\end{equation}
		
		\proofstep{Step~4. Characterization of $\bm{D\nu(x_0)}$.} The last step consists on characterizing the limiting Young measure  in~\eqref{eq:limiting} as $j$ tends to infinity. Here, we expect the biting part of $\bm{\sigma}$ to vanish since we are performing a blow-up at a purely singular point $x_0 \in S$. Arguing as in Step~4 of the proof of the localization at regular points, we cast the $\BigO(\eta)$-approximation of the dense family $\{\xi_k\}$ behind~\eqref{eq:x_04}; this time with the integrands $\tilde f_{k,m} \coloneqq \phi_m \circ T^{(x_0,r_j)} \otimes g_m \circ \Gamma^{\xi_k} \otimes h_m(c_j r_j^d \frarg)$. For positive $\eta$ we find $k(\eta) \in \Nbb$ sufficiently large so that that (using the density of the family $\{\xi_k\}_{k \in \Nbb}$ in $Z$, and the uniform continuity of $g_m$)
		\[
		|\xi_0 - \xi_k(\eta)| = \BigO(\eta) \quad \text{and} \quad |\tilde f_m - \tilde f_{k,m}| \le \BigO(\eta) \,  [{\varphi \otimes \chi_Z \otimes (1 + c_jr_j^d|\frarg|_E) (x,\xi,z)}],
		\] 
		where $\tilde f_m  \coloneqq \phi_m \circ T^{(x_0,r_j)} \otimes g_m \circ \Gamma^{\xi_0} \otimes h_m(c_j r_j^d \frarg)$.  
		Using these estimates in~\eqref{eq:sloc2} we deduce by~\eqref{eq:x01} that
		\begin{equation}\label{eq:?}
		\begin{split}
		\ddprb{f_m, \bm{\sigma^{(j)}}}  
		& %= {r_j}^{-d} \ddprb {f_{k(\eta),m} , \bm{\nu}}  \\
		%& = {r_j}^{-d} \bpr{\tilde f_m,\bm\nu}(\cl Q) \\
		= {r_j}^{-d} \bpr{\tilde f_{k(\eta),m},\bm\nu}(\cl{Q_{r_j}(x_0)}) + \BigO(\eta).
		\end{split}
		\end{equation}

		For the next step let  $\eta$ be an arbitrary positive real. We claim  that taking then the limit as $j\to \infty$ at both sides of the equality yields 
		\begin{equation}\label{eq:yacasi}
		\begin{split}
		\lim_{j \to \infty} \ddprb{f_m, \bm{\sigma^{(j)}}} & =  h_m (0)\int_Q \phi_m(y) \bigg(\int_Z g_m \circ \Gamma^{k(\eta)}(\xi) \dd \xi \bigg) \dd y  + \\ & \qquad  \bbpr{\tilde f_{k(\eta),m},\bm{\nu}}_{x_0}(Z)\cdot \int_{\cl{Q}}\phi_m(y) \dd (D\lambda)(y)  +  \BigO(\eta)\,.
		\end{split}
		\end{equation}
		First, we deal with the absolutely continuous part of the barycenter measure in~\eqref{eq:?} as follows. For all $\Leb^d$-almost every $x \in \cl{Q_{r_j}(x_0)}$ we estimate %bound the density of the absolutely continuous part of a translation of the weighted barycenter as 
		(recall that $h_1 \equiv \chi_{\R^N}$)
		
		\[
		\left| \frac{\dd \bpr{\tilde f_{k(\eta),m} - \phi_m\circ T^{(x_0,r_j)} \otimes g_m \circ \Gamma^{\xi_{k(\eta)}} \otimes h_m(0)h_1,\bm\nu}}{\dd \Leb^d}(x)\right| \le  c_jr_j^dC(m)\cdot \bbpr{|\frarg|_E,\bm{\nu}}_{x}(Z),
		\]
		where $C(m) \coloneqq \|\phi_m\|_\infty\cdot\|g_m\|_\infty \cdot\textrm{Lip}(h_m)$. This estimate leads to the bound
		\begin{align*}
		\dashint_{{Q_{r_j}(x_0)}}& \bigg|\frac{\dd \bpr{\tilde f_{k(\eta),m} - \phi_m\circ T^{(x_0,r_j)} \otimes g_m \circ \Gamma^{\xi_{k(\eta)}} \otimes h_m(0)h_1,\bm\nu}}{\dd \Leb^d}(x)\bigg| \dd x\\& \lesssim_{(m)} 
		\frac {\int_{Q_{r_j}(x_0)} \int_Z \dpr{|\frarg|_E \dd \xi,\nu_{x,\xi}} \dd \xi + \int_{Q_{r_j}(x_0)} \lambda^{ac}(x) \dd x}{\lambda^s(Q_{r_j}(\tilde x))} = \BigO(r_j).
		\end{align*}
		Here, to reach the last equality we have used that $x_0 \in S$ satisfies~\eqref{eq:x01}. Thus, when passing to limit $j \to \infty$, we may substitute the absolutely continuous part of $r_j^{-d}\bpr{\tilde f_{k(\eta)}, \bm{\nu}}$ in~\eqref{eq:?} by the integrable function 
		\[
		x \mapsto \frac{1}{r^d_j}\bigg( \int_Z\dprb{\tilde f_{k(\eta),m}(x,\xi,\frarg),\delta_0} \dd \xi \bigg)\,.
		\] 
		We now deal with the singular  part and the passing to the limit. By~\eqref{eq:limiting} and the density identity from~\eqref{eq:x02} (applied to the point $\tilde x = x_0 \in S$), we may let $j$ tend to infinity at both sides of~\eqref{eq:yacasi} to deduce (recall that $h_m^\infty(c_jr_j^d \frarg) = c_jr_j^dh_m^\infty$)
		\begin{align*}
		\ddprb{f_m,\bm{D\nu(x_0)}} &  =  \lim_{j \to \infty} \bigg(\dashint_{Q_{r_j(x_0)}}\int_Z \dpr{\tilde f_{k(\eta),m}(x,\xi,\frarg),\delta_0} \dd \xi \dd x \\
		& \qquad  +  \frac{1}{r_j^d}\int_{\cl{Q_{r_j}(x_0)}}\frac{\dd \bpr{\tilde f_{k(\eta),m}^\infty,\bm\nu}}{\dd \lambda^s}(x) \dd \lambda^s(x) \bigg) + \BigO(\eta) \\
		&  =  h_m (0)\int_Q \phi_m(y) \bigg(\int_Z g_m \circ \Gamma^{\xi_{k(\eta)}}(\xi) \dd \xi \bigg) \dd y  + \\ & \qquad  \lim_{j \to \infty} \dashint_{\cl{Q_{r_j}(x_0)}} \phi_m(x)\frac{\bpr{\tilde f^\infty_{k(\eta),m},\bm{\nu}}}{\dd \lambda^s}(x) \dd\lambda^s(x) + \BigO(\eta)\\
		& =  h_m (0)\int_Q \phi_m(y) \bigg(\int_Z g_m \circ \Gamma^{\xi_{k(\eta)}}(\xi) \dd \xi \bigg) \dd y  + \\ & \qquad  \bbpr{\tilde f_{k(\eta),m},\bm{\nu}}_{x_0}(Z)\cdot \int_{\cl{Q}}\phi_m(y) \dd (D\lambda)(y)  + \BigO(\eta)\,,
		\end{align*}
		where we have used $\Lambda_j = c_j T^{(x_0,r_j)}_\#\lambda^s \toweakstar D\lambda$ on $\cl Q$ in passing to the last equality. This proves the claim.

		To conclude we observe that $\eta$ has so far been chosen arbitrarily. Therefore, by the dominated convergence theorem, the identity~\eqref{eq:yacasi} implies %as in the argument behind~\eqref{eq:x_05} (where it is utilized that $|\xi_0 - \xi_{k(\gamma)}| = \BigO(\eta)$) we get 
		\begin{equation}
		\begin{split}
		\ddprb{f_m,\bm{D\nu(x_0)}} & = h_m (0)\int_Q \phi_m(y) \bigg(\int_Z g_m \circ \Gamma^{\xi_0}(\xi) \dd \xi \bigg) \dd y \\ & \qquad +  \bbpr{\tilde f_{m},\bm{\nu}}_{x_0}(Z)\cdot\int_{\cl{Q}}\phi_m(y) \dd (D\lambda)(y) \\
		& = \int_Q\bigg( \int_Z\dprb{f_m(x,\frarg,\frarg),\delta_0}  \dd \xi\bigg)  \dd y \\
		& \qquad  +\int_{\cl Q}\bigg(\int_Z\dpr{f_m^\infty(x,\frarg,\frarg),\nu_{x_0,\xi_0 + \xi}^\infty} \dd(\Gamma^{\xi_0}_\#\rho_{x_0})(\xi) \bigg) \dd (D\lambda)(y)\,.
		\end{split}
		\end{equation}
		Here, we have used that
		$\Gamma^{\xi_0}_\# \Leb^d_Z = \Leb^d_Z$ to reach the second equality. Since $\{f_m\}_{m \in \Nbb}$ separates $\E^2(Q;\R^N)$, %and the identity~\eqref{eq:?} hold for arbitrary $m \in \Nbb$,
		Lemma~\ref{lem:density} gives
		\[
		\bm{D\nu(x_0)} = (\delta_0,D\lambda,D\rho,D\nu^\infty), \quad D\lambda \in \Tan_1(\lambda^s,x_0),
		\]
		where $D\rho(x) = \Gamma^{\xi_0}_\#\rho_{x_0}$ and $D\nu^\infty(y,\xi) = \nu_{x_0,\xi_0 + \xi}^\infty$ for all $y \in \cl Q$. This finishes the proof. 
	\end{proof}

	\section{PDE-constrained Young measures}\label{sec:5a}
	
	Let $k \in \N$ and let $E,F$ be finite-dimensional real vector spaces. For $\alpha \in \N^d$ we define its modulus $|\alpha| \coloneqq \alpha_1 + \dots + \alpha_d$. We shall consider general homogeneous differential operators of order $k$ on $\R^d$ from $E$ to $F$, that is, operators of the form 
	\[
	\A = \sum_{\substack{\alpha \in \N^d\\|\alpha| = k}}  A_\alpha\,\partial^\alpha, \qquad A_\alpha \in \textrm{Lin}(E;F).
	\]
	A vector-valued measure $\mu \in \M(\Omega;E)$ is called $\A$-free provided that 
	\[
	\A\mu = 0 \quad \text{in the sense of distributions on $\Omega$}.
	\]
	We say that $\bm\nu \in \Y^2(\Omega;E)$ is an $\A$-free two scale* Young measure if it is generated by a sequence of (asymptotically) $\A$-free measures: 
	
	\begin{definition}[$\A$-free Young measures] Let $1 < p < \frac{d}{d-1}$. A two-scale* Young measure $\bm\nu \in \Y(\Omega;E)$ is called $\A$-free if there exist a sequences $\eps \searrow 0$ and $(\mu_\eps)_\eps \subset \M(\Omega;E)$ such that
		\[
		\A\mu_\eps \to 0 \;\; \text{strongly in $\Wrm^{-k,p}(\Omega)$ \quad and \quad $\mu_\eps \toYY \bm\nu$}.
		\]
		%The set of such $\A$-free two-scale* Young measures is denoted by $\Y^2_\A(\Omega;E)$.
	\end{definition}

	\subsection{Rigidity properties of $\A$-free two-scale* Young measures}
	
	Clearly, the barycenter $[{\bm{\nu}}]$ of an $\A$-free two-scale* Young measure is $\A$-free. This same property is inherited to second-scale barycenters $\bbpr{\bm{\nu}}$ as it is portrayed in the next proposition. To deal with a possible abuse of notation about the domain of partial differential operators, we shall write $\A_\xi$ to denote the action of $\A$ on $\Dcal'(Z;E)$, i.e., %that is, %if $\eta \in \Dcal'(Z;E)$, then $\A_\xi \eta$ is the distribution defined as
	\[
	(\A_\xi\eta)[g] = \dprb{\eta,\A^*g} \quad \text{for all $\eta \in \Dcal'(Z;E)$ and $g \in \Crm^\infty_\per(Q;F)$}\,.
	\]  
	
	%An important property of $\A$-free two-scale Young measures is that their second-scale barycenter is an $\A$-free measure.
	\begin{proposition}\label{prop:secondbarycenterisafree} Let $\bm{\nu} = (\nu,\lambda,\rho,\nu^\infty) \in \Y^2(\Omega;E)$ be an $\A$-free two-scale* Young measure. Then at $(\Leb^d + \lambda^s)$-almost every $x_0 \in \Omega$ it holds that
		%tEquivalently, for  $(\Leb^d + (\lambda^\nu)^s)$-a.e.~$x_0 \in \cl \Omega$, 
		\[
		\A_\xi\bbpr{\bm{\nu}}_{x_0} = 0 \quad \text{in the sense of distributions on $Z$}\,.
		\] 
	\end{proposition}
	\begin{proof}
		Let $\phi \in \Crm^\infty_c(\Omega)$ and let $g \in \Crm^\infty(Z)$ be arbitrary functions. 
		Let $(\mu_\eps)_\eps$ be a sequence of $\A$-free measures that generates the Young measure $\bm{\nu}$ (on $\Omega$) and let $T^\varepsilon(x) = x/\varepsilon$ be the re-scaling by the factor $\eps$.
		As a consequence of the product rule there exist constants $c_{\alpha ,\beta}$ such that
		\[
		\A^*(\phi \cdot \eps^k (g\circ T^\eps)) = (-1)^k\sum_{\substack{|\alpha|\le k\\\beta \le \alpha}} \eps^{k - |\beta|}c_{\alpha ,\beta}A^T_\alpha  (\partial^{\alpha - \beta} \phi)(\partial^\beta g) \circ T^\eps,
		\]
		where we write $\beta \le \alpha$ if and only if $\beta_i \le \alpha_i$ for every $i = 1,\dots,d$.
		Observe that $c_{\alpha,\alpha} = 1$ for every $\alpha$, and in particular it follows that 
		\[
		\| \A^*(\phi \cdot  \eps^k (g\circ T^\eps)) - \phi \cdot (\A^* g)\circ T^\eps\|_\infty \to 0, \quad \text{as $\eps \searrow 0$}.
		\]
		Hence, since $\A \mu_\eps \rightarrow 0$ in strongly in $\Wrm^{-k,p}(\Omega) \cembed \Crm_0(\Omega)^*$, we deduce from the convergence above that
		\begin{equation}\label{eq:discarde}
		\begin{split}
		0 & = \lim_{\eps \searrow 0} \int_\Omega \A^*[\phi \cdot \eps^k (g\circ T^\eps)](x) \dd \mu_\eps(x) \\
		& =  \lim_{\eps \searrow 0} \int_\Omega \phi(x)(\A^*g)(x/\eps) \dd \mu_\eps(x) \ \\
		& = \ddprb{\phi \otimes \A_\xi^*g \otimes \id_{E},\bm{\nu}}.
		\end{split}
		\end{equation}
		Using the tensor structure of the integrand and the property~\eqref{def:weighted} of the second-scale barycenters we obtain %for $\phi\in \Crm(\cl\Omega)$, $g\in\Crm(Z)$ we have
		\begin{align*}
		0 & = \int_{\cl\Omega}\int_{Z}  \dd \bbpr{\phi \otimes \A_\xi^* g \otimes \id_{\R^N},\bm{\nu}}_x(y) \dd (\Leb^d + \lambda^s)(x) \\
		& = \int_{\Omega} \phi(x) \bigg(\int_Z (\A_\xi^*g)(\xi) \dd \bbpr{\bm\nu}_x(\xi)\bigg) \dd (\Leb^d + \lambda^s)(x).
		\end{align*}
		%Using this property of the two-scale barycenter of $\bm{\nu}$ we obtain from~\eqref{eq:discarde} that
		%\[
		%0 = \int_{\cl \Omega} \phi(x) \bigg(\int_{Z} \A^*g(y) \dd \bbpr{\bm{\nu}}_x(y)\bigg) \dd (\Leb^d + (\lambda^\nu)^s)(x).
		%\]
		Since the choice of $\phi \in \Crm^\infty_c(\Omega)$ was arbitrary, this identity must hold locally. Namely, for $(\Leb^d + \lambda^s)$-a.e.~$x_0 \in \Omega$ it holds
		\[
		\dprb{g,\A_\xi \bbpr{\bm{\nu}}_{x_0}} = 
		\int_{Z} (\A_\xi^*g)(\xi) \dd \bbpr{\bm{\nu}}_{x_0}(\xi) =0 \quad \text{for all $g \in \Crm^\infty(Z)$}.
		\]
		By the distributional definition of derivative this is equivalent to
		\[
		\A_\xi \bbpr{\bm{\nu}}_{x_0} = 0 \quad \text{in the sense of distributions on $Z$}.
		\]
		This proves the assertion.
	\end{proof}
	
	\begin{corollary}[differential rigidity of the second-scale]\label{cor:Afree} Let $\eps \searrow 0$ and let $(\mu_\eps)_\eps \subset \M(\Omega;E)$ be a sequence of asymptotically $\A$-free measures that two-scale converges to a limit $\lambda \otimes \theta_x$. That is, such that
		\[
		\text{$\A\mu_\eps \to 0$ \;\; strongly in $\Wrm^{-k,p}(\Omega)$}\,.
		\] 
		Then, at $\lambda$-almost every $x_0 \in \Omega$, it holds
		\[
		\A_\xi \theta_{x_0} = 0 \quad \text{in the sense of distributions on $Z$.}
		\]
	\end{corollary}
	\begin{proof}
		This follows directly from~\eqref{eq:1-3} and the previous proposition.
	\end{proof}
	
	\begin{corollary}[structure of $\A$-free two-scale* Young measures]\label{cor:GuidoFilip}
		Let $\bm{\nu}\in \Y^2_\A(\Omega;E)$ be an $\A$-free two-scale* Young measure. Then, at $\lambda^s$-almost every $x_0 \in \Omega$, the following differential inclusion holds:
		\[
		\frac{\dd \bbpr{\bm{\nu}}_{x_0}}{\dd |\bbpr{\bm{\nu}}_{x_0}^s|}(\xi) \in \Lambda_{\A} \quad \text{for $|\bbpr{\bm{\nu}}_{x_0}^s|$-almost every $\xi \in Z$}.
		\]  
	\end{corollary}
	\begin{proof}
		By Proposition \ref{prop:secondbarycenterisafree} we have that for $(\Leb^d + \lambda^s)$-almost every $x_0 \in \Omega$, the second-scale barycenter  $\bbpr{\bm{\nu}}_{x_0}$ is an $\A_\xi$-free measure on $Z$. Since locally, $\A_\xi$ and $\A$ coincide as operators in $\R^d$, it follows from \cite[Theorem~1.1]{de-philippis2016on-the-structur} that
		\[
		\frac{\dd \bbpr{\bm{\nu}}_{x_0}}{\dd |\bbpr{\bm{\nu}}_{x_0}^s| }(\xi)    \in \Lambda_{\A} \text{ for } |\bbpr{\bm{\nu}}_{x_0}^s| \text{-a.e. $\xi \in Z$.}
		\]
		This finishes the proof. 
		%By definition we have that $\bbpr{\bm{\nu}}_x^s = G \cdot \rho_\nu^s$ for some $ G \in L^1(Z;\rho_\nu^s)$ and the assertion follows.
		%\mnote{J: Inserted the proof, could you quickly doublecheck?}
	\end{proof}
	
	The next lemma asserts the support of the purely singular part of an $\A$-free measure cannot be arbitrary. In fact, it must be contained in the smallest vectorial space containing the wave cone $\Lambda_{\A}$.
	
	\begin{lemma}\label{lem:support} Let $\bm{\nu} = (\nu,\lambda,\rho,\nu^\infty) \in \Y^2(\Omega;E)$ be an $\A$-free two-scale* Young measure. %Assume also that 
		%	\[
		%	[\bm{\nu}] \in \M(Q;\spn\Lambda_\A)\,.
		%	\]
		Then the support of the purely singular part of $\bm\nu$ is contained in $\spn \{\Lambda_\A\}$, that is, 
		\[
		\supp \, (\nu^\infty_{x, \xi}) \subset \spn \{\Lambda_{\A}\} \cap \partial \Bbb_E \quad \text{for $(\lambda^s \otimes \rho_x)$-almost every $(x,\xi) \in {\Omega} \times Z$}\,. 
		\]  	
	\end{lemma}
	\begin{proof}
		Recall that if $(\mu_\eps)_\eps$ generates the two-scale* Young measure $\bm \nu$, then the same sequence generates the generalized Young measure $\bm\upsilon = (\upsilon,\lambda,\upsilon^\infty)$ where $\upsilon^\infty$ is the weak-$*$ $\lambda$-measurable map $x \mapsto \upsilon_x^\infty \in \mathrm{Prob}(\partial \Bbb_E)$ and each probability measure $\upsilon_x$ is defined by duality as   
		\begin{equation}\label{eq:com3}
		\dprb{h^\infty,\upsilon_x^\infty} = \int_Z \bigg(\int_{\partial \Bbb_E} h^\infty(\hat z) \dd \nu^\infty_{x, \xi}(\hat z)\bigg) \dd \rho_x(\xi) \quad \text{for all $h^\infty \in \E^0(E)$}\,.
		\end{equation}
		Moreover, if $(\mu_\eps)_\eps$ is originally an asymptotically $\A$-free sequence, then by definition $\bm{\upsilon}$ is an $\A$-free generalized Young measure (see~\cite{arroyo-rabasa2017lower-semiconti} for the corresponding definition). The localization principle~\cite[Proposition~2.25]{arroyo-rabasa2017lower-semiconti} (for generalized Young measures at singular points) yields the existence of a set $S \subset \Omega$ with full $\lambda^s$-measure and satisfying the following property: at each $x_0 \in S$ there exists an $\A$-free sequence $(u_\eps)_\eps \subset \Lrm^1(Q;E)$ ---depending on the point $x_0$--- such that 
		\begin{equation}\label{eq:com1}
		u_\eps \toY  \bm{\upsilon(x_0)} = (\delta_0,\lambda_{x_0},\upsilon_{x_0}) \in \Y^1(Q;E), \quad \tilde \lambda_{x_0}(\partial Q) = 0,  
		\end{equation}
		and $\tilde \lambda_{x_0} \in \Tan_1(\lambda,x_0)$ is a probability measure (hence a non-zero positive measure). Moreover, by Theorem~\cite[Theorem~1.1]{de-philippis2016on-the-structur}, we may further assume 
		${\dprb{\id_{\partial \Bbb_E},\upsilon_{x_0}}} = {|\dprb{\id_{\partial \Bbb_E},\upsilon_{x_0}}|} \cdot {\dd [\upsilon]}/{\dd |[\upsilon]|^s}(x_0) \in \Lambda_\A$. 
		In particular
		\begin{equation}\label{eq:com2}
		[\bm{\upsilon(x_0)}] \in \M(Q;\spn\{\Lambda_\A\}) \text{\quad for every $x_0 \in S$}\,. 
		\end{equation}
		
		By properties~\eqref{eq:com1}-\eqref{eq:com2} we may apply~\cite[Lemma~3.2]{arroyo-rabasa2017lower-semiconti} to each generalized young measure $\bm{\upsilon(x_0)}$. Using~\eqref{eq:com2} once more we deduce 
		\[
		\text{$\supp\,(\upsilon_{x_0}^\infty) \subset \spn\{\Lambda_\A\} \cap \partial \Bbb_E$ for every $x_0 \in S$}\,.
		\]  
		Hence, by~\eqref{eq:com3}, $\supp\,(\nu^\infty_{x,\xi}) \subset  \spn\{\Lambda_\A\} \cap \partial \Bbb_E$ for $(\lambda^s \otimes \rho_x)$-almost every $(x,\xi) \in Q \times Z$. This finishes the proof. 
	\end{proof}

	In the introduction we have defined the $\A$-free homogeneous envelope for integrands $\Crm(Z \times E)$ which are convex in their second argument. This definition is nothing else than a simplified representation of the (general) definition of $\A$-free homogeneous envelope defined for arbitrary integrands: 
	
	\begin{definition}
		Let $h : Z \times E \to \R$ be a continuous integrand. The $\A$-free homogeneous envelope of $h$ is the integrand 
		\begin{equation}\label{eq:hom}
		\begin{split}
		h_\Ahom(z) & \coloneqq \inf\setBB{\dashint_{Q_R} h(y,z + w(y)) \dd y }{ R \in \Nbb, \\
			& \qquad  \quad \qquad  w \in \Crm_\per^\infty(Q_R;E) \cap \ker \A, \int_{Q_R} w \dd y = 0}\,, \quad z \in E\,.
		\end{split}
		\end{equation}
	\end{definition}
	The following relation about the commutability of the recession operation and the homogenization of an integrand holds. 
	
	\begin{proposition}[recession regularization vs. homogenization] \label{prop:exinftyhom}
		%The $(\frarg)^\infty$ and $(\frarg)_{\h}$ operations
		%do not commute in $\Crm(\Omega \times Z \times \R^N)$. 
		Let $h \in \Crm(Z \times \R^N)$ be an integrand with linear-growth at infinity. Then
		\[
		(h^\#)_{\Ahom} \ge (h_{\Ahom})^\#. %\quad \text{for all $z \in \R^N$}. 
		\]
		\begin{proof}
			%Let us without loss of generality assume that $f$ does not depend on the $x$-component.
			Fix a vector $z \in E$. Let also $R \in \Nbb$ and $w \in \Crm_\per^\infty(Q_R;E)$ as in~\eqref{eq:hom}. Recall that $h(\xi,t \frarg)/t \le M(1 + |\frarg|) \in \Lrm^1_\loc(\R^d)$, hence we may use Fatou's lemma and the definition of $h^\#$ to obtain
			\begin{equation}\label{eq:inequality}
			\begin{split}
			\dashint_{Q_R} h^\#(\xi,A + w(y))\dd y  & =   \dashint_{Q_R} \limsup_{t \to \infty} \frac{h(y,tz + tw(y))}{t} \dd y \\
			& \geq \limsup_{t \to \infty} \dashint_{Q_R} \frac{h(y,tz + tw(y))}{t} \dd y \\
			& \geq  \limsup_{t \to \infty} \frac{h_{\Ahom}(tz)}{t} = (h_{\Ahom})^\#(z).
			\end{split}
			\end{equation}
			In passing to the first inequality we have used the linearity of both $\A$ and the mean value operation to ensure that $tw \in \Crm^\infty(Q_R;E) \cap \ker \A$, $\int_{Q_R} tw = 0$.
			Taking the infimum over such $w$'s first, and subsequently over all $R \in \Nbb$ in~\eqref{eq:inequality} gives 
			\[
			(h^\#)_{\Ahom}(z) \ge (h_{\Ahom})^\#(z),
			\] 
			as desired.
		\end{proof}
	\end{proposition}
	
	\section{Convex homogenization}\label{sec:6}

	\subsection{Jensen-type inequalities} 
	This section is devoted to the study of (Jensen) integral inequalities satisfied by the {$\A$}-homogeneous envelope of convex integrands with respect to arbitrary $\A$-free two-scale* Young measures. The plan is to establish Jensen type inequalities at the first- and second-scale; naturally involving the first and second barycenters. Bridging these two inequalities into an \emph{homogenized}  Jensen inequality is, in turn, the key argument towards the proof of Theorem~\ref{thm:hom}. We close this section with an open problem and a discussion about non-convex integrands.
	
	\subsubsection{First-scale Jensen's inequality}
	%\note{The righthandside is just $\bpr{h,\bm{\nu}}_x$, right?
	%}
	\begin{proposition}\label{lem:J1}
		Let $\bm{\nu}= (\nu,\lambda,\rho,{\nu}^\infty) \in \Y^2(\Omega;E)$ be an $\A$-free two-scale* Young measure and let $h \in \Crm(Z \times E) \cap \Rbf^2(\Omega;E)$ be an integrand that is convex in its second argument, that is, $h(\xi,\frarg)$ is convex for all $\xi \in Z$. Then,
		\begin{enumerate}[(i)]
			\item at every regular point $x \in \reg_{\bm{\nu}}(\Omega)$ it holds 
			\begin{equation}\label{eq:J1r}
			\begin{split}
			h_{*\A}\bigg(\frac{\dd \bpr{\bm{\nu}}}{\dd \Leb^d}(x)\bigg) \leq \int_Z & h\bigg(\xi,\frac{\dd \bbpr{\bm{\nu}}_x}{\dd \Leb^d}(\xi)\bigg) \dd \xi  \\ 
			& + \int_Z h^\infty\bigg(\xi,\frac{\dd \bbpr{\bm{\nu}}^s}{\dd \rho_x}(\xi)\bigg) \dd \rho_x (\xi)\,,
			\end{split}
			\end{equation}
			\item and, at every singular point $x \in \sing_{\bm{\nu}}(\cl\Omega)$,
			\begin{equation}\label{eq:J1s}
			(h^\infty)_{*\A}\bigg(\frac{\dd \bpr{\bm{\nu}}_x}{\dd \lambda^s}(x)\bigg) \le \int_Z h^\infty\bigg(\xi,\frac{\dd \bbpr{\bm{\nu}}_x}{\dd \rho_x}(\xi)\bigg) \dd \rho_x(\xi)\,.
			\end{equation}
		\end{enumerate}
	\end{proposition}
	\begin{proof}
		Let  $\phi$ be a non-negative mollifier on $Z$ with $\int_Z \phi = 1$. Set $\phi_\delta \coloneqq \delta^{-d}\phi(x/\frarg)$ so that $\phi_\delta \Leb^d_Z$ is a probability measure {on $Z$}. %Recall from XXX, that the weighted second-scale barycenter... CCC 
		We define, for fixed $\delta > 0$ and $x \in \Omega$, the mollified second-order barycenter %t  $H^{\delta,\infty}_x:Z \rightarrow \R^N$ the \emph{first-component convolution density} of $\rho^\delta_x$ for fixed first component' by
		\[
		v_\delta(\xi) \coloneqq (\phi_\delta \ast \bbpr{\bm{\nu}}_x)(\xi)\,. %= \int_Z \phi_\delta(\xi- \omega) \dprb{h(\omega,\frarg),\nu_{x,\omega}^\infty} \dd \rho_x^\nu (\omega).
		\]
		Recall from  \cite[Remark 2.2]{arroyo-rabasa2017lower-semiconti} that convergence of mollified measures is strengthened to area-convergence, that is, $v_\delta$ area-converges to $\bbpr{\nu}_x$ on $Z$ (as $\delta \todown 0$). Hence, by \cite[Theorem~5]{kristensen2010characterizatio} we obtain 
		\begin{equation}\label{eq:join1}
		\int_Z h(\xi,v_\delta(\xi)) \dd \xi \to \int_Z h\bigg(\xi,\frac{\dd \bbpr{\bm{\nu}}_x}{\dd \Leb^d}(\xi)\bigg) 
		+ \int_Z h^\infty\bigg(\xi,\frac{\dd \bbpr{\bm{\nu}}^s_x}{\dd |\bbpr{\bm{\nu}}_x^s|}(\xi)\bigg) \dd |\bbpr{\bm{\nu}}_x^s|(\xi)\,.
		\end{equation}
		On the other hand, by the properties of mollifiers and the differential rigidity of the second-scale barycenter proved in Corollary~\ref{cor:Afree}, it holds that every test function $\tilde v_\delta \coloneqq v_\delta - \bbpr{\bm{\nu}}_x(Z)$ is $\A_\xi$-free and has zero mean value. This property of the $\tilde v_\delta$'s enables us to use the definition of the $\A$-homogeneous envelope (of a convex integrand) which yields 
		\begin{equation}\label{eq:join2}
		\int_Z h(\xi,v_\delta(\xi)) \dd \xi = \int_Z h(\xi,\tilde v_\delta(\xi) + \bbpr{\bm{\nu}}_x(Z)) \dd \xi \ge h_{*\A}(\bbpr{\bm{\nu}}_x(Z))\,.
		\end{equation}
		From \eqref{eq:join1} and~\eqref{eq:join2} we conclude 
		\begin{equation*}
		\int_Z h\bigg(\xi,\frac{\dd \bbpr{\bm{\nu}}_x}{\dd \Leb^d}(\xi)\bigg) 
		+ \int_Z h^\infty\bigg(\xi,\frac{\dd \bbpr{\bm{\nu}}^s_x}{\dd |\bbpr{\bm{\nu}}_x^s|}(\xi)\bigg) \dd |\bbpr{\bm{\nu}}_x^s|(\xi) \ge h_{*\A}(\bbpr{\bm{\nu}}_x(Z))\,.
		\end{equation*}
		Thus, taking into account that $|\bbpr{\bm{\nu}}_x^s| = |\dpr{\id_{\partial \Bbb_E},\nu^\infty_{x,\xi}}| \rho_x^s$ and $h^\infty(\xi,\dpr{\id_{\partial \Bbb_E},\nu^\infty_{x,\xi}}) = h^\infty(\xi,0) = 0$ for $\rho_x^*$, where $\rho_x^*$ is the singular part of $\rho_x^s$ with respect to $|\bbpr{\bm{\nu}}_x^s|$, we get (using the 1-homogeneity of $h^\infty$) the refined estimate
		\begin{equation}\label{eq:join3}
		\int_Z h\bigg(\xi,\frac{\dd \bbpr{\bm{\nu}}_x}{\dd \Leb^d}(\xi)\bigg) 
		+ \int_Z h^\infty\bigg(\xi,\frac{\dd \bbpr{\bm{\nu}}^s_x}{\dd \rho_\nu^s}(\xi)\bigg) \dd \rho_x^s(\xi) \ge h_{*\A}(\bbpr{\bm{\nu}}_x(Z))\,.
		\end{equation}
		In particular, if $x \in \reg_\nu(\cl\Omega)$, we apply the inequality~\eqref{eq:join3} to deduce
		\begin{align*}
		\int_Z h\bigg(\xi,\frac{\dd \bbpr{\bm{\nu}}_x}{\dd \Leb^d}(\xi)\bigg) 
		+ \int_Z h^\infty\bigg(\xi,\frac{\dd \bbpr{\bm{\nu}}^s_x}{\dd \rho_x^s}(\xi)\bigg) \dd \rho_x^s(\xi) & \stackrel{\eqref{eq:density2}}{\ge} h_{*\A}\bigg(\frac{\dd \bpr{\bm{\nu}}}{\dd \Leb^d}(x)\bigg).
		\end{align*}
		This proves~\eqref{eq:J1r}. If on the other hand $x \in \sing_\nu(\cl \Omega)$, we use~\eqref{eq:join3} (with $h^\infty$ in place of $h$) to get 
		\begin{align*}
		\int_Z h^\infty\bigg(\xi,&\frac{\dd \rho_x}{\dd \Leb^d}(\xi)\dpr{\id_{E},\nu_{x,\xi}^\infty} \bigg) \dd \xi + 
		\int_Z h^\infty\bigg(\xi,\dpr{\id_{\partial \Bbb_E},\nu_{x,\xi}^\infty} \bigg)\dd \rho_x^s (\xi) \\[5pt]
		& \qquad = \int_Z h^\infty\bigg(\xi,\frac{\dd \bbpr{\bm{\nu}}_x}{\dd \rho_x}(\xi)\bigg) \dd \rho_x (\xi) \stackrel{\eqref{eq:density2}}{\ge} (h^\infty)_{*\A}\bigg(x,\frac{\dd \bpr{\bm{\nu}}}{\dd \lambda^s}(x)\bigg).
		\end{align*}
		%\textcolor{blue}{Maybe its a good idea to give one additional sentence to the last argument. $h^\infty$ in place  of $h$ means: Using the forgoing arguments for $h^\infty$ since its convex if $h$ is convex and $(h^\infty)^\infty =h^\infty$?}
		This proves~\eqref{eq:J1s} and the proof is completed.
	\end{proof}
	
	\subsubsection{Second-scale Jensen inequalities}
	
	\begin{proposition}[at regular points]\label{lem:J2r} Let $\bm{\nu}= (\nu,\lambda,\rho,{\nu}^\infty) \in \Y^2(\Omega;E)$ be an $\A$-free two-scale* Young measure and let $h \in \Crm(Z \times E) \cap \Rbf^2(\Omega;\E)$ be an integrand that is convex in its second argument, that is, $h(\xi,\frarg)$ is convex for all $\xi \in Z$. Then, for all regular points $x \in \reg_{\bm{\nu}}(\cl\Omega)$ the following inequalities hold:
		%\item At regular points. Let $h \in \Rbf^2(\R^N)$ be an integrand that is convex in its second argument, then for all $x \in \reg_\nu(\Omega)$ it holds
		
		\begin{enumerate}[(i)]
			\item at $\Leb^d_Z$-almost every $\xi$ in $Z$,
			\begin{equation}\label{eq:J2r}
			h\bigg(\xi,\frac{\dd \bbpr{\bm{\nu}}_x}{\dd \Leb^d}(\xi)\bigg) \le \frac{\dd \bbpr{h,\bm{\nu}}_x}{\dd \Leb^d}(\xi),
			\end{equation} 
			\item and, at $\rho_x^s$-almost every $\xi$ in $Z$,
			\begin{equation}\label{eq:J2r'}
			h^\infty\bigg(\xi,\frac{\dd \bbpr{\bm{\nu}}_x}{\dd \rho_x^s}(\xi)\bigg) \le \frac{\dd \bbpr{h^\infty,\bm{\nu}}_x}{\dd \rho_x^s}(\xi).
			\end{equation} 
		\end{enumerate}
		%\dpr{h(\xi,\frarg),\nu_{x,\xi}} + \bigg(\frac{\dd \lambda^\nu}{\dd \Leb^d}(x) \cdot \frac{\dd \rho_x^\nu}{\dd \Leb^d}(x)\bigg) \dpr{h^\infty(\xi,\frarg),\nu_{x,\xi}^\infty} 
	\end{proposition}
	\begin{proof}Let $x \in \reg_\nu(\Omega)$ and $\xi \in Z$ so that ${\dd \bbpr{\bm{\nu}}_{x}}/{\dd \Leb^d}(\xi)$ exists (note that this property is satisfied at $\Leb^d_Z$-almost every $\xi$ in $Z$). Observe  that since $h$ is $\xi$-uniformly Lipschitz in its second argument, ${\dd \bbpr{h,\bm{\nu}}_{x}}/{\dd \Leb^d}(\xi)$ also exists at such $\xi$'s. We have
		\[
		\frac{\dd \bbpr{\bm{\nu}}_{x}}{\dd \Leb^d}(\xi) =\dpr{\id_{E},\nu_{x,\xi}} + \frac{\dd \lambda}{\dd \Leb^d}(x)\cdot \frac{\dd\rho_x}{\Leb^d}(\xi)\cdot \dpr{\id_{\partial \Bbb_E},\nu_{x,\xi}^\infty}.
		\]
		%which is nothing else than the convolution $\dprb{h^\infty(\frarg,}),\nu_{x,\frarg}^\infty} \ast \rho_x^\nu$.
		%For $\lambda$-a.e.~$x\in\cl \Omega$ we have 
		By the classical Jensen's inequality and the positive $1$-homogeneous character of $h^\infty$ we further obtain
		\begin{align*}
		\frac{\dd \bbpr{h,\bm{\nu}}_x}{\dd \Leb^d}(\xi) & = \dpr{h(\xi,\frarg),\nu_{x,\xi}}  + \frac{\dd \lambda}{\dd \Leb^d}(x)\cdot \frac{\dd\rho_x}{\Leb^d}(\xi)\cdot \dpr{h^\infty(\xi,\frarg),\nu_{x,\xi}^\infty}  \nonumber \\
		& \ge h(\xi,\dpr{\id_{E},\nu_{x,\xi}}) + h^\infty\bigg(\xi,\frac{\dd \lambda}{\dd \Leb^d}(x)\cdot \frac{\dd\rho_x}{\Leb^d}(\xi)\cdot \dpr{\id_{\partial\Bbb_E},\nu_{x,\xi}^\infty}\bigg). 
		\end{align*}
		Now, let us recall the following sub-additive property satisfied by convex functions and their recession functions: every convex function $g : E \to \R$ 
		$g(z_1 + z_2) \le g(z_1) + g^\infty(z_2)$ for all $z_1,z_2 \in E$. It follows directly from this observation and the inequality above that 
		\[
		\frac{\dd \bbpr{h,\bm{\nu}}_x}{\dd \Leb^d}(\xi) \ge  h\bigg(x,\xi,\frac{\dd \bbpr{\bm{\nu}}_x}{\dd \Leb^d}(\xi)\bigg)\,. %\quad \text{for $\Leb^d$-almost everywhere $\xi \in Z$}\,. %\ge h_x\bigg(\xi,).
		\]
		%Since $\frac{\dd \bbpr{\nu}_{x}}{\dd \Leb^d}(\xi)$ exists for $\Leb^d$-a.e.~$\xi \in Z$, 
		This proves the first inequality in~\eqref{eq:J2r}.
		
		For the proof of the second inequality we let $\xi \in Z$ be such that ${\dd \bbpr{\bm{\nu}}_x}/{\dd \rho_x^s}(\xi)$ exists, hence also ${\dd \bbpr{h,\bm{\nu}}_x}/{\dd \rho_x^s}(\xi)$ exists ($h^\infty$ is $\xi$-uniformly Lipschitz continuous in its second argument). In fact, these two measure-theoretic derivatives are given by the parings $\dpr{\id_{\partial\Bbb_E}, \nu^\infty_{x,\xi}}$ and $\dpr{h(\xi,\frarg), \nu^\infty_{x,\xi}}$ respectively. Once more, a simple application of Jensen's classical inequality gives
		\begin{align*}
		\frac{\dd \bbpr{h,\bm{\nu}}_x}{\dd \rho_x^s}(\xi) & = \dpr{h^\infty(\xi,\frarg),\nu^\infty_{x,\xi}}  \\
		& \geq h^\infty(\xi,\dpr{\id_{\partial\Bbb_E},\nu_{x,\xi}^\infty})
		\end{align*}
		Since ${\dd \bbpr{\bm{\nu}}_x}/{\dd \rho_x^s}(\xi)$ and $\dpr{h(\xi,\frarg), \nu^\infty_{x,\xi}}$ exist $\rho_x^s$-almost everywhere in $Z$, this proves~\eqref{eq:J2r'}.\end{proof}
	%\mnote{Why aren't we done here? }
	%& \ge \int_Z \phi_\delta(\omega - \xi)\cdot h\bigg(\dpr{\id,\nu_{x,\xi}}) + \frac{\dd \lambda^\nu}{\dd \Leb^d}(x)\dpr{\id,\nu_{x,\xi}}\bigg)\dd \xi
	%
	%Another application of Jensen's inequality to each summand yields
	%\begin{align}\label{eq:HR}
	%H^\delta_x(\xi) & \ge h_x\bigg(\xi,\int_Z\phi_\delta(\omega - \xi) \dpr{\id,\nu_{x,\omega}} \dd \omega \bigg) \nonumber \\ 
	%& \quad + (h_x)^\infty\bigg(\frac{\dd \lambda^\nu}{\dd \Leb^d}(x)\int_Z \phi_\delta(\omega - \xi)\dpr{\id,\nu_{x,\xi}}\dd \rho_x^\nu(\omega)\bigg). 
	%\end{align}
	%
	%Now, fix $x \in \Sbf_\nu$ and let $\xi \in Z$ be such that $\frac{\dd \bbpr{\nu}_x}{\dd \rho_x^\nu}(\xi)$ exists. Using the convexity in the form of Jensen's inequality (first with the probability measures $\{\nu_{x,\omega}\}_{\omega \in Z}$, and subsequently with the family of probability measures $\{\rho_x^\nu\}_{x \in \Omega}$) and the one-homogeneity of $h^\infty$ in the $z$-variable, we obtain
	%\begin{align*}
	%\frac{\dd \bbpr{h,\nu}}{\dd \rho_x^\nu}(\xi) & = \dpr{h^\infty(x,\xi,\frarg),\nu^\infty_{x,\xi}}  \nonumber\\
	%   & \ge  h^\infty(x,\xi,\dpr{ \id_{\Sbb^{N-1}},\nu^\infty_{x,\xi}}) = h\bigg(x,\xi,\frac{\dd \bbpr{\nu}_x}{\dd \rho_x^\nu}(\xi)\bigg)\,.%\label{eq:estforsingularpart}\\
	%%& \ge  h^\infty\bigg(\xi, \int_{Z} \phi_\delta(\xi- \omega) \dpr{\id_{\Sbb^{N-1}}, \nu^\infty_{x,\omega}}  \dd \rho_x^\nu(\omega)\bigg).\nonumber
	%\end{align*}
	%Since $\frac{\dd \bbpr{\nu}_x}{\dd \rho_x^\nu}(\xi)$ exists for $\rho^\nu_x$-a.e. $\xi \in Z$, this is equivalent to~\eqref{eq:weighted_ineq}.

	\begin{proposition}[at singular points] Let $\bm{\nu}= (\nu,\lambda,\rho,{\nu}^\infty) \in \Y^2(\Omega;E)$ be an $\A$-free two-scale* Young measure and let $g \in \Crm(Z \times \R^N)$ be continuous integrand. Further, assume that $g(\xi,\frarg)$ is convex and positively 1-homogeneous for all $\xi \in Z$. Then, at every $x \in \sing_{\bm\nu}(\Omega)$ it holds that
		\begin{equation}\label{eq:J2s}
		g\bigg(\xi,\frac{\dd \bbpr{{\bm\nu}}_x}{\dd \rho_x}(\xi)\bigg) \le \frac{\dd \bbpr{g,{\bm\nu}}_x}{\dd \rho_x}(\xi) \quad \text{for $\rho_x$-almost every $\xi \in Z$}\,.
		\end{equation}
	\end{proposition}
	\begin{proof}
		The proof can be reproduced by following the exact same ideas in the proof of the second part of the proof of Proposition~\ref{lem:J2r}.
	\end{proof}
	
	We are now in place to prove the \emph{homogeneous} version of Jensen's inequality that involves the $\A$-$\hom$ envelope (of convex integrands). 
	
	\begin{theorem}[homogenized Jensen's inequality]\label{thm:J}
		Let $\bm{\nu}= (\nu,\lambda,\rho,{\nu}^\infty) \in \Y^2(\Omega;E)$ be an $\A$-free two-scale* Young measure and let $h \in \Crm(Z \times E) \cap \Rbf^2(E)$ be an integrand that is convex in its second argument, that is, $h(\xi,\frarg)$ is convex for all $\xi \in Z$. Then,
		\begin{enumerate}
			\item at every regular point $x \in \reg_{\bm\nu}(\Omega)$ it holds that
			\begin{equation}\label{eq:J1rr}
			\begin{split}
			h_{*\Acal}\bigg(\frac{\dd \bpr{{\bm\nu}}}{\dd \Leb^d}(x)\bigg) \leq \frac{\dd \bbpr{h,{\bm\nu}}_x(Z)}{\dd \Leb^d}(x)\,,
			\end{split}
			\end{equation}
			\item and, at every singular point  $x \in \sing_{\bm\nu}(\Omega)$,
			\begin{equation}\label{eq:J1ss}
			(h^\infty)_{*\Acal}\bigg(\frac{\dd \bpr{{\bm\nu}}}{\dd \lambda^s}(x)\bigg) \le  \frac{\dd \bbpr{h,{\bm\nu}}_x(Z)}{\dd \lambda^s}(x)\,.
			\end{equation}
		\end{enumerate}
	\end{theorem}
	\begin{proof}
		The proof is a direct consequence of Propositions~\ref{lem:J2r},~\ref{lem:J2r}, and~\ref{lem:J1}.
	\end{proof}
	
	\subsubsection{Comments on the non-convex case} 
	
	In general, even at singular points $x \in\sing_{\bm{\nu}}\cap \cl \Omega$, one \emph{cannot} expect the second-scale differential inclusion 
	\begin{equation}\label{eq:fail}
	\frac{\dd \bbpr{\bm{\nu}}_{x_0}}{\dd \Leb^d_Z}(\xi) \in \Lambda_{\A} \quad \text{for $\Leb^d_Z$-almost every $\xi \in Z$}
	\end{equation}
	to hold. 
	Instead, we believe the following weaker statement holds under mild non-degeneracy assumptions on $\A$ (for instance if $\A$ satisfies Murat's constant rank condition).
	\begin{conjecture}\label{con:sera} Let $\bm{\nu} \in \Y^2(\Omega;\R^N)$ be an $\A$-free two-scale* Young measure and let $h \in \Crm(Z \times \R^N)$ be a positively 1-homogeneous for all $\xi \in Z$.
		Then, for all $x_0 \in \sing_{\bm{\nu}}$ it holds that
		\begin{equation}\label{eq:conjecture}
		\bbpr{h,\bm\nu}_{x_0}(Z) %= \int_Z \dpr{h(\xi,\frarg),\nu_{x_0,\xi}^\infty} \dd \rho_{x_0}^\nu(\xi)  
		\ge h_{\Ahom}(\bbpr{\bm\nu}_{x_0}(Z))
		\end{equation}
	\end{conjecture}
	This is a \emph{powerful} inclusion, and, in fact, it is the key inequality towards the characterization of the homogenization of \emph{non-convex}  integrals. The proof of this conjecture is relatively simple provided that $\A$ is a first-order operator (see~\cite{matias2015homogenization-}). 
	%In particular, the result holds for gradient two-scale Young measures and the $\curl$-homogeneous envelope of $h$. 
	For operators of general order, the veracity of this conjecture seems to be linked to the \emph{structural}  properties of tangent $\A$-free measures.
	
	\subsubsection{The conjecture in $\BD$} 
	A natural candidate to study the elastic perfectly plastic behavior of materials is the space of
	functions whose linearized strains are measures. Formally, these deformations belong to the space of \emph{functions of  bounded deformation} which is defined as
	\[
	\BD(\Omega) \coloneqq \setb{u \in \Lrm^1(\Omega;\R^d)}{Eu \coloneqq Du + Du^T \in \M(\Omega;\Mbf_{d \times d}^\sym)}.
	\] 
	In the $\A$-free context, a measure $\mu \in \M(\Omega;\Mbf_{d \times d}^\sym)$ is locally a symmetric gradient if and only if it satisfies (see~\cite[Example~3.10(e)]{fonseca1999mathcal-a-quasi}) the second order PDE-constraint 
	\[
	\curl \curl \mu \coloneqq \bigg(\sum_{i=1}^d \partial_{ik} \mu_{i}^j+\partial_{ij} \mu_{i}^k-\partial_{jk} \mu_{i}^i-\partial_{ii} \mu_{j}^k\bigg)_{jk} \qquad j,k=1,\ldots,d\,.
	\] 
	In a forthcoming~\cite{arroyo-rabasa2018rigidity-of-tan} paper we shall give a positive answer to Conjecture~\ref{con:sera} in the case $\bm{\nu}$ is $\BDY^2(\Omega)$ two-scale* Young measure, that is, when
	\[
	Eu_j \toYY \bm{\nu} \qquad (u_j)_{j \in \Nbb} \subset \BD(\Omega).
	\]

	%In a forthcoming paper, the first author gives a positive answer to Conjecture~\ref{con:sera} in the context of linear elasticity where~\eqref{eq:conjecture} involves two-scale Young measures generated by symmetric gradient measures, that is, 
	%\begin{equation}\label{eq:conjecture2}
	%\bbpr{h,\nu}_{x_0}(Z) %= \int_Z \dpr{h(\xi,\frarg),\nu_{x_0,\xi}^\infty} \dd \rho_{x_0}^\nu(\xi)  
	%\ge h_{\curl \curl}(\bbpr{\nu}_{x_0}(Z))
	%\end{equation}
	%where
	%\[
	%Eu_\eps \toYY \nu,
	%\] 
	%and $\A = \curl \curl$ is the second-order operator associated to the St.-Venant compatibility conditions. As mentioned above, the proof of this fact is related to the structural properties of symmetric gradients. Namely, it relies on the following characterization of the tangent set of linearized strain measures. Let $u \in \BD(\Omega)$ be a function of bounded deformation. Then, for $|Eu^s|$-almost every $x_0 \in \Omega$,
	%\[
	%\tau \in \Tan_1(Eu,x_0) \quad \Leftrightarrow \quad \tau = b \cdot h_1'(y \cdot a) + a \cdot  h_2'(y \cdot b)
	%\]
	%for some pair of vectors $a,b \in \R^d$ and functions $h_1, h_2 \in \BV(\R)$. These type of differential inclusions have been considered by Rindler in \cite{Rin11} (see also  \note{XXX}) to establish lower semicontinuity results on $\BD$-spaces. 
	%The characterization of homogenized non-convex integrals on $\BD(\Omega)$ relies on~\eqref{eq:conjecture2} and standard geometric measure theory arguments. 

	\subsection{Proof of the Theorem~\ref{thm:hom}} By the definition of $\Gamma$-limit it is enough to verify the conclusion of the theorem for infinitesimal sequences. Let $\eps_j\searrow 0$ and $\mu_n \toweakstar \mu$ in $\M(\cl \Omega;E) \cap \ker \A$. 
	Observe that up to extracting a subsequence we may assume without loss of generality that
	\[
	\liminf_{j \to \infty} I^{\eps_j}      (\mu_j) = \lim_{j \to \infty} I^{\eps_j}(\mu_j) 
	\]
	and
	\[
	\text{and} \quad \mu_n \toYY \bm{\nu} \quad \text{for some $\bm{\nu}= (\nu,\lambda,\rho,\nu^\infty)\in \Y^2(\Omega;E)$}\,.
	\]
	
	The Young-measure representation and the Radon--Nykodym differentiation give
	\begin{align*}
	\liminf_{j \to \infty} I^{\eps_j}(\mu_j) & = \ddprb{f,\bm{\nu}} \\
	&  = \int_{\cl \Omega} \bbpr{f,\bm{\nu}}_x(Z) \dd(\Leb^d + \lambda^s)(x) \\
	& =  \int_{\Omega} \frac{\bbpr{f,\bm{\nu}}_x(Z)}{\dd \Leb^d}(x) \dd x + \int_{\cl \Omega} \frac{\bbpr{f,\bm{\nu}}_x(Z)}{\dd \lambda^s}(x) \dd \lambda^s(x)\,.
	%& = \int_{\Omega} \int_Z \bigg(\frac{\dd \bbpr{f,\nu}_x}{\dd \Leb^d}(\xi)  + \frac{\dd \bbpr{f,\nu}_x}{\dd (\rho_x^\nu)^s}(\xi) \dd(\rho_x^\nu)^s\bigg) \dd x \\
	%& \qquad  + \int_{\cl \Omega}\int_Z  \frac{\dd \bbpr{f^\infty,\nu}_x}{\dd (\rho_x^\nu)}(\xi) \dd (\lambda^\nu)^s(x)\,.
	\end{align*}
	An $x$-point-wise application of the Jensen inequalities in Theorem~\ref{thm:J} with the family of integrands $\{h(\xi,z)\}_x \coloneqq \{f(x,\xi,z)\}_x$ on the right-hand side above yields (notice that it suffices to perform this at all $x \in \cl \Omega$ where the conclusion of Theorem~\ref{thm:J} holds) 
	\begin{align*}
	\liminf_{j \to \infty} I^{\eps_j}(\mu_j) & \ge \int_{\Omega} f_{*\Acal}\bigg(x,\frac{\dd \mu}{\dd \Leb^d}(x)\bigg) \dd x + \int_{\cl \Omega} (f^\infty)_{*\Acal}\bigg(x,\frac{\dd \bpr{\nu}_x}{\dd (\lambda^\nu)^s}(x)\bigg) \dd \lambda^s(x) \\
	& \ge \int_{\Omega}f_{*\Acal}\bigg(x,\frac{\dd \mu}{\dd \Leb^d}(x)\bigg) \dd x + \int_{\cl \Omega} (f^\infty)_{*\Acal}\bigg(x,\frac{\dd \mu}{\dd |\mu^s|}(x)\bigg) \dd |\mu^s|(x)\,. 
	\end{align*}
	Here, to reach the second inequality we have used that $|\dpr{\id_{\R^N},\nu_x}|(\lambda^\nu)^s \equiv |\mu^s|$ and that $\frac{\dd |\mu^s|}{\dd \lambda^s}(x) = 0$ for $\lambda^*$-a.e. $x \in \cl \Omega$ where $\lambda^*$ is the singular part of $\lambda^s$ with respect to $|\mu^s|$.
	
	Lastly, it follows from Proposition~\ref{prop:exinftyhom} that $(f^\infty)_{*\Acal} \ge (f_{*\Acal})^\infty$ whence we conclude
	\[
	\liminf_{j \to \infty} I^{\eps_j}(\mu_j) \ge\int_{\Omega} f_{*\Acal}\bigg(x,\frac{\dd \mu}{\dd \Leb^d}(x)\bigg) \dd x + \int_{\cl \Omega} (f_{*\Acal})^\infty\bigg(x,\frac{\dd \mu}{\dd |\mu^s|}(x)\bigg) \dd |\mu^s|(x)\,.
	\]
	This finishes the proof.\qed
	
	\appendix
	
	\section{Background on Young measures}
	%Before embarking ourselves with the proofs of the main results, let us first discuss briefly the existing notions of Young measures
	%and their relations to oscillation and concentration phenomena. The expert reader who may wish to skip this is invited to continue reading 
	%Section~\ref{sec:preliminaries}.
	
	\subsubsection{Classical Young measures}
	Let $\Omega$ be an open set in $\R^d$ and let $(u_k)_{k \in \Nbb}$ be a uniformly bounded sequence in $\Lrm^1(\Omega)$. Qualitatively speaking, there 
	are two main reasons why the sequence fails to converge strongly in $\Lrm^1(\Omega)$, namely oscillation and concentration. The first occurs when 
	fluctuations average out. For example, if $d = 1$, the sequence
	\begin{equation}\label{eq:intro_osc}
	u_k = \sin(kx) \quad k = 1,2,\dots
	\end{equation}
	converges weakly to the constant function zero, while its mass remains constant and strictly positive on the interval $[0,2\pi]$. This corresponds to an \emph{oscillation} behavior. 
	The \emph{generalized surface} measures, introduced by L.~C. Young~\cite{young1937generalized-cur,young1942generalized-sur,young1942generalized-surII} and  nowadays known as (classical) Young measures, are 
	powerful measure-theoretic tools to understand oscillations. In an informal manner, one can define {the} Young measure associated to a weakly convergent 
	sequence $(u_k)_{k \in \Nbb} \subset \Lrm^1(\Omega;\R^N)$ as the family of probability measures $\{\nu_x\}$, parameterized by $x \in \Omega$, with the fundamental 
	property that
	\begin{equation}\label{eq:intro_1}
	\int_\Omega f(x,u_k(x)) \dd x \to \int_\Omega \Big(\int_{\R^N} f(x,z) \dd \nu_x(z) \Big) \dd x \quad \text{for all $f \in \Crm_0(\Omega \times \R^N)$}\,.
	\end{equation} 
	Notice that $\nu_x = \delta_{u(x)}$ when $u_k$ converges strongly to $u$ (it suffices to assume convergence in measure). Thus, the total variation measure 
	$|\delta_{u(x)} - \nu_x|$ gives a sense of how rapidly the sequence oscillates around $x$. Moreover, property~\eqref{eq:intro_1} makes of Young measures a 
	natural candidate to represent solutions of variational integral problems, which, may otherwise have no solution in their respective domain of definition.

	Young realized that studying the weak convergence of the \emph{surfaces} \enquote{$\graph u_k$} is the right way to overcome the incompatibility of weak convergence with nonlinear functionals. %Formally, we can explain through the notion of \emph{push-forward} measure. 
	The reasoning behind this claim is the following. The uniformly distributed measure $\Gamma_k$ which is concentrated on the set $\graph u_k \subset \Omega \times \R^N$ is formally expressed by the push-forward measure $(\id, u_k)_\# \Leb^d \restrict \Omega$. Provided   that
	\begin{equation}\label{eq:intro_2}
	\Gamma_k \toweakstar \Gamma \quad \text{as measures in $\Omega \times \R^N$}, 
	\end{equation}
	the convergence in~\eqref{eq:intro_1} can be written as $\dpr{f,\Gamma_k} \to \dpr{f,\Gamma}$ in terms of the duality $(\,\Crm_0(\Omega \times \R^N)\,,\,\M(\Omega \times \R^N)\,)$. The Young measure $\nu_x$ is then nothing else than the \emph{slice} of $\Gamma$ at a point $x \in \Omega$. This reasoning justifies 1) Young's original definition of \emph{generalized surface}, and 2) the \emph{probabilistic interpretation} $\nu_x(A) \approx \lim_{k \to \infty} \Pbb(\setn{u_k(y) \in A}{\text{for $y$ about $x$}})$, expressed rigorously by the limit
	\begin{equation}\label{eq:interpret}
	\nu_x(A) = \; \lim_{\delta \searrow 0}\;\Big( \lim_{k \to \infty} \int_{B_\delta(x)} \chi_A(u_k) \dd y\Big) \qquad A \subset \R^N \,.% \enquote{around} $x$}
	\end{equation}
	%where the limit in the right-hand is to be understood locally rather than pointwise.
	%
	% $\Gamma$ is the weak limit of $\graph u_k$ (as measures in the space $\Omega \times \R^N$), then the \emph{slice} of $\Gamma$ at a point $x \in \Omega$ contains the probability distribution of the event $\{\lim_{k \to \infty} u_k(x)\}$.
	
	%\subsubsection*{Concentration and generalized Young measures}
	
	\subsubsection{Generalized Young measures}
	
	Albeit powerful, the notion of Young measure is somehow unsatisfactory since it relies on the equi-integrability of minimizing sequences.
	There is a second phenomenon  that hinders strong convergence and it corresponds to \emph{concentration} of mass. %, which is precisely when a sequence stops being equi-integrable. 
	The reader may think of a unit mass distribution at the point $x = 0$ (the Dirac mass $\delta_0$ centered at the origin) and let it evolve according to the heat flow, which is highly regularizing and mass preserving. The solution at a time $t > 0$ is given by the gaussian
	\begin{equation}\label{eq:intro_g}
	v_t = \frac{1}{4\pi t}\exp\Big({\frac{|x|^2}{4t}}\Big)\,.
	\end{equation}
	As we go backwards in time, say with the sequence $u_k \coloneqq v_{1/k}$, the sequence $(u_k)_{k \in \Nbb}$ weak-$*$ converges (in the sense of measures) to $\delta_0$. Therefore, $(u_k)_{k \in \Nbb}$ is a sequence of probability measures in $\Lrm^1(\R^d)$ which converges strongly to the zero function in $\R^d \setminus \{0\}$ and nevertheless fails to converge strongly at precisely the point $x = 0$. It is worth to mention that loss of compactness of an $\Lrm^1$-bounded sequence (with respect to the strong topology) does not correspond exclusively to oscillation or concentration, but rather to a combination of both. %Examples of this behavior can be easily constructed by exploiting the linearity of weak-$*$ convergence.
	The understanding of this scenario is part of the seminal work of DiPerna \& Majda~\cite{diperna1987oscillations-an}, which was motivated by evidence pointing that beyond a critical time $T > 0$, the solutions $v_\eps$ of the Navier-Stokes equations (with Reynolds number $\eps^{-1}$) tend to develop wild oscillations as well as concentration effects. Therefore, suggesting that $v_\eps$ weak-$*$ converges (but not strongly) to $v$ a solution of the Euler equation. In their effort to understand the complexity of the flow, they introduced a notion of \emph{measure-valued} solution for the 3-$d$ incompressible Euler equation. 
	To define  \emph{measure-valued} solutions, DiPerna \& Majda extended Young's ideas and introduced \emph{generalized Young measures} (see also~\cite{alibert1997non-uniform-int}). 
	
	Let us discuss briefly their construction and its differences with respect to Young's construction. First, notice that the classical characterization~\eqref{eq:intro_1} fails to deal with concentration of mass. Indeed, taking the reversed heat flow~\eqref{eq:intro_g}, we readily check the limit \enquote{forgets} to distinguish the point $x = 0$ since
	\[
	\int_{B_1} f(x,u_k(x)) \dd x \to \int_{B_1} f(x,0) \dd x \quad \text{for all $f \in \Crm_0(B_1 \times \R^d)$}.
	\]
	This happens because the sequence is tested with bounded integrands. The general idea behind their construction is to test with the \emph{largest} family of functions where one can hope to compute the limit in the left-hand side.
	%\begin{equation}%\label{eq:intro_3}
	%\int_\Omega f(x,\,u_\eps(x)\,).
	%\end{equation}
	Due to the $\Lrm^1$-boundedness, the natural candidates are integrands satisfying a uniform linear-growth condition, that is,  
	\[
	|f(x,z)| \le M(1 + |z|) \quad \text{for some $M > 0$}.
	\]
	However, in spite that $\Gamma_k \toweakstar \Gamma$ as in~\eqref{eq:intro_2}, we cannot ensure $
	\dpr{\Gamma_k, f} \to \dpr{\Gamma, f}$ as before. This owes to the fact that $f \notin \Crm_0(\Omega \times \R^N)$ and hence the attempted pairing above is not in the correct duality. The turn around to this problem rests in the following compactification argument. Since $\Omega$ is open and bounded, %, and $\Tbb^d$ is a compact manifold, 
	the Stone--\v Cech compactification of $X = \Omega \times \R^N$ reduces to
	\[
	\beta X = \cl \Omega \times \beta \R^N,
	\]
	where $\beta \R^N$ is the result of glueing the infinity points at every direction to $\R^N$. In this way $g \in \Crm(\beta X)$ if and only if $g$ is uniformly bounded on $X$, and, it can be continuously extended at every direction of infinity. Since we imposed a linear growth condition on $f$, the function
	\[
	\tilde f(x,z) = \frac{f(x,z)}{1 + |z|} 
	\]
	is uniformly bounded on $X$. To verify $\tilde f \in \Crm(\beta X)$ we require that $\tilde f$ can be extended continuously to $\beta X$, or equivalently, that the \emph{recession function}
	\[
	f^\infty(x,\hat z) \coloneqq \lim_{\substack{x' \to x\\z' \to \hat z\\t \to \infty}} \frac{f(x',tz')}{t} \;\; \text{exists in $\R$ for all $x  \in \cl \Omega$ and $\hat z \in \Bbb^N$},
	\]
	where $\Bbb^N$ is the closed unit ball of $\R^N$. The next step is to balance the additional weight coming from the transformation $f \mapsto \tilde f$ by defining $\tilde \Gamma_k \coloneqq (\id, u_k)_\#[(1 + |u_k|) \Leb^d]$, which is again a uniformly bounded sequence in $\M(\beta X)$. Neglecting the pass to further subsequences, we may assume that $\tilde \Gamma_k$ weak* converges to some Radon measure $\tilde \Gamma$ on $\beta X$.
	Then, the Riesz--Markov--Kakutani representation theorem ensures that $\tilde f$ and $\tilde \Gamma_k$ are in duality and thus
	\[
	\dpr{\Gamma_k,f} = \dpr{\tilde \Gamma_k,\tilde f} \to \dpr{\tilde \Gamma,\tilde f}\,.
	\]
	Let us now consider the canonical projection $\pi : \cl\Omega \times \beta \R^N \to \cl\Omega$ and let $\tilde \lambda = \pi_* \tilde \Gamma$ be the associated push-forward measure of $\tilde \Gamma$. Similarly to Young's construction, a slicing argument and the convergence above yields the limit representation
	\begin{equation*}%\label{eq:intro_4}
	\int_\Omega f(x,u_k(x)) \dd x \to \int_{\cl \Omega} \bigg(\int_{\beta \R^N} \tilde f(x,z) \dd \tilde \nu_x(z) \bigg) \dd \tilde \lambda(x) \quad \text{for all $f \in \Crm_b(\Omega \times \R^N)$}\,.
	\end{equation*} 
	%On the other hand, the left hand side above can be thought of as a sequence of weak* maps on $\Lrm^1(\Omega;\Crm(\beta \R^N))$ by noticing that $f(x,\frarg) \in \Crm(\R^N)$ and the map
	%\[
	%x \mapsto |\dpr{f(x,\frarg),\delta_{u_\eps(x)}}| \le 
	%\]
	%is measurable and essentially bounded on $\Omega$. 
	%The space of all such integrands $f$ conforms a Banach space which we shall denote as $\E(\Omega;\R^N)$.
	%The \emph{generalized} Young measure associated to a sequence $(u_k)_{k \in \Nbb} \subset \Lrm^1(\Omega;\R^N)$
	%is now a triple $(\lambda,\nu,\nu^\infty)$ where $\nu = \{\nu_x\}$ is as before, 
	%\[
	%(1 + |u_\eps|) \Leb^d \restrict \Omega \toweakstar \lambda  \quad \text{as positive measures in $\cl \Omega$},
	%\]
	%and $\nu = \{\nu_x\}$ is a family a probability measures in $\mathrm{Prob}(\partial \Bbb^{N})$, parameterized by $x \in \cl \Omega$ and which is $\lambda$-weak* measurable. The representation in~\eqref{eq:intro_1} extends to %
	%   
	%%from $\cl \Omega$ to 
	%$\mathrm{Prob}(\R^N) \oplus \mathrm{Prob}(\Sbb^{N-1})$ 
	%satisfying
	In fact, one can exploit the topological isomorphism $\beta \R^N \cong \Bbb^N$ and the idea that $f^\infty$ is the trace of $\tilde f$ at infinity to re-write the right-hand side above in the form
	\begin{equation*}
	%\begin{split}
	%\int_\Omega f(x,u_k(x)\,) \to 
	\int_{\cl \Omega} \underbrace{\bigg(\int_{\R^N} f(x,z) \dd \nu_x(z) \bigg) \dd x}_{\textrm{pure oscillation}}
	+ \int_{\cl \Omega} \underbrace{\bigg( \int_{\partial \Bbb^{N}} f(x,z)^\infty \dd \nu_x^\infty(z) \bigg) \dd \lambda(x)}_{\textrm{concentration}}.
	%\end{split},\quad \text{ for all $f \in \Crm(\Omega \times \beta \R^N)$}.
	\end{equation*} %such that
	This construction leads to the following non-rigorous {definition}. The generalized Young measure associated to a sequence $(u_k)_{k \in \Nbb}$ is a triple $(\lambda,\nu,\nu^\infty)$ conformed by a positive measure $\lambda$ satisfying
	\begin{equation}
	(1 + |u_k|) \restrict \Omega \toweakstar \lambda \quad \text{as measures on $\cl\Omega$},
	\end{equation}
	and two families $\nu = \{\nu_x\}$, $\nu^\infty = \{\nu_x^\infty\}$ of probability measures (parameterized by $x \in \cl \Omega$) satisfying the fundamental property that
	\begin{equation}
	\begin{split}
	\int_\Omega f(x,u_k(x)) \dd x \quad \to  \quad & \int_{\cl \Omega} \bigg(\int_{\R^N} f(x,z) \dd \nu_x(z) \bigg) \dd x \\
	& +  \int_{\cl \Omega}\bigg( \int_{\partial \Bbb^{N}} f(x,z)^\infty \dd \nu_x^\infty(z) \bigg) \dd \lambda(x)
	\end{split}
	\end{equation}
	for all $f \in \Crm_b(\Omega \times \R^N)$.
	%\[
	%\tilde f(x,z) \coloneqq \frac{f(x,z)}{1 + |z|} \in \Crm(\Omega \times \beta \R^N). 
	%\]
	%The construction follows a similar principle to~\eqref{eq:intro_1}-\eqref{eq:intro_2} %for a larger set of admissible test functions, 
	%and %which can already be inferred from the example in~\eqref{eq:intro_g} that the compactly supported functions $\Crm_c(\Omega \times \R^N)$ cannot detect measure escaping to infinity. 

	Notice that the correspondent probabilistic interpretation for $\nu_x$ remains the same as in~\eqref{eq:interpret}. The term $\lambda_\nu(\{x\})$ can be interpreted as the amount of mass carried by the sequence $(|u_k|)_{k \in \N}$ about $x$, and $\nu_x^\infty(B) \approx \lim_{k \to \infty} \Pbb(\setn{u_k(y)/|u_k|(y) \in  B}{\text{for $y$ about $x$}})$ or 
	\begin{equation}\label{eq:interpret2}
	\nu_x(B) = \lim_{\delta \searrow 0} \Big(\lim_{k \to \infty} \int_{B_\delta(x)} \chi_{B}(u_k / |u_k|) \dd y\Big) \qquad B \subset \Sbb^{N-1}\,.
	\end{equation}

	\subsubsection{Classical two-scale Young measures}
	
	There is yet another extension of the classical setting which arises from the following question: can we quantify \emph{how fast} or \emph{how often} oscillation occurs with respect to a given parameter? In good part, this is motivated by materials science problems such as the description of \emph{macroscopic} and \emph{microscopic} properties of composite materials. The mathematical approach is that of \enquote{homogenization} to which a particular model corresponds the description of weak (or weak*) limits of sequences of the form
	\begin{equation}\label{eq:intro_hom}
	f(x,x /\eps,u_\eps(x)) \quad \text{with $\eps \searrow 0$}\,.
	\end{equation} 
	Here, the function $f$ is assumed to be $[0,1]^d$-periodic in its second argument. 
	One often refers to $x$ as the \emph{macroscopic} scale and to $x / \eps$ as the \emph{microscopic} one, thence also called \emph{two-scale} analysis. %In our previous examples $\eps$ corresponds to $k^{-1}$. 
	To put in this in context with our previous examples~\eqref{eq:intro_osc} and~\eqref{eq:intro_g} we set $k^{-1}$ to play the role of $\eps$ in~\eqref{eq:intro_hom}. In the first case oscillations are uniformly distributed in space with period $2\pi k^{-1}$, while in the second example we can argue \enquote{most} of the mass carried by $u_k$ is concentrated in a neighborhood of radius $k^{-1}$ around the origin. In general, such information is not recorded by generalized Young measures. This is portrayed by the following $1$-dimensional example. Fix $\alpha > 0$, and consider the purely oscillatory sequence
	\[
	u_{\alpha,\eps} = \sin(\eps^{-\alpha} x) \qquad k = 1,2,\dots.
	\]
	Clearly, each $\alpha$ significantly changes the scale at which oscillations occur as $k$ tends to $\infty$. However, a change of variables shows that regardless of the value of $\alpha$, the associated Young measure to $(u_{\alpha,\eps})_\eps$ is given by the homogeneous family $\{\nu_x\}_{x \in \R}$ where $\nu_x = \overline{\delta_g}$ for every $x \in \R$ and $\overline{\delta_g}$ is the probability measure satisfying
	\[
	\dpr{\overline{\delta_g},\phi} \coloneqq \frac{1}{2\pi}\int_0^{2\pi} \phi(\sin y) \dd y \quad \text{for all $\phi \in \Crm_c(\R)$}.
	\]
	Hence, the search for a Young measure able to distinguish different length-scales of convergence (one that would explicitly depend on $\alpha$ in the previous example). Taking a step towards the solution
	of this problem, 
	%If a sequence $\{u_k\}$ generates a Young measure $\nu$, then every re-parametrization $\{u_{j(k)}\}$ also generates $\nu$. For example the sequences
	%\[
	%u_k \coloneqq \chi_{[0,1/2)}(k x) \quad \text{and} \quad \tilde u_{j(k)} \coloneqq \chi_{[0,1/2)}(j(k)  x)
	%\]
	%generate the same (homogeneous) young measure $\nu_x = \frac 12 \delta_0 + \frac12 \delta_1$.
	Pedregal introduced the notion of \emph{two-scale} Young measure as a way to represent weak limits of \emph{equi-integrable} sequences of the form~\eqref{eq:intro_hom}.  
	Let us briefly recall the the ideas behind Pedregal's construction.   
	First, let $K \subset \R^d$ be a compact set and consider an additional sequence $\{v_\eps : \Omega \to K\}$. Then, \emph{extend} the original equi-integrable sequence $(u_\eps)_\eps$ to the sequence of pairs
	\[
	(\,(\,v_\eps\,,\,u_\eps\,)\,)_\eps \qquad \eps \searrow 0.
	\]
	This sequence is also equi-integrable and therefore it can be fully analyzed within the framework of classical Young measures. 
	For the sake of our discussion, we assume that 
	\begin{align}
	u_\eps \quad & \toY \quad \nu = \{\nu_x\}_{x \in \Omega} &\subset& \; \mathrm{Prob}(\R^N)\\
	v_\eps \quad &  \toY \quad \rho = \{\rho_x\}_{x \in \Omega} &\subset& \; \mathrm{Prob}(K)\\
	(\,u_\eps\,,\,v_\eps\,) \quad & \toY \quad \sigma = \{\sigma_{x}\}_{x\in\Omega} &\subset& \;\mathrm{Prob}(\Omega \times K),
	\end{align}
	where \enquote{\,$\toY$\,} means the sequence in the left hand-side generates the Young measure in the right-hand side in the sense of~\eqref{eq:intro_1}.
	%$\{(v_\eps,u_\eps)\}$ and $\{v_\eps\}$ generate (classical) Young measures $\sigma$ and $\rho$ 
	%= \{\sigma_x\}_{x \in \Omega}$ and $\rho =\{\rho_x\}_{x \in \Omega}$ 
	%respectively; as before $\nu$ will be the associated Young measure to $\{u_\eps\}$. %(we denote by $\nu$ the Young measure generated by $\{u_\eps\}$). %If we write $(\xi,z)$ to denote point in the product space $Q \times \R^N$ (which is the range of the paired sequence) we get
	%\[
	%\pi_\xi(
	%\]
	Since we can recover $v_\eps$ by projecting into the first coordinate of the pair $(v_\eps,u_\eps)$, a well-known disintegration argument yields the existence of a family of probability measures ${\tilde \nu} = \{\tilde \nu_{x,\xi}\}_{x \in \Omega,\xi \in K} \subset \mathrm{Prob}(K)$ such that 
	\begin{equation}\label{eq:intro3}
	\sigma_x = \rho_x \otimes \tilde \nu_{x,\xi}, \quad (x,\xi) \in K \times \Omega.
	%\footnote{The generalized product $\lambda \otimes \nu_x$ between a positive measure $\lambda \in \M(\Omega)$ and a weak-$*$ measurable map $\Omega \to \mathrm{Prob}(\R^N) : x \mapsto \nu_x$}
	\end{equation}
	If the sequences $(v_\eps)_\eps$, $(u_\eps)_\eps$ oscillate at different length-scales, for instance when one oscillates considerably faster than the other (as $\eps \todown 0$), then the semi-product \eqref{eq:intro3} simplifies to the classical product 
	\begin{equation}\label{eq:intro1}
	\sigma_x = \rho_x \otimes \nu_x \quad \text{where \; $\sigma_x(E \times F) = \rho_x(E)\cdot\nu_x(F)$}, 
	\end{equation}
	which \emph{does not} contain any additional information on $\nu$. However, if $v_\eps$ and $u_\eps$ oscillate at a comparable length-scale (for small $\eps$'s), then the \emph{joint} Young measure $\sigma$ will record how \enquote{$u_\eps$ compares to $v_\eps$} at the length-scale where both of their oscillations interact. A priori, one may consider \emph{any} candidate for the test functions $(v_\eps)$. For example, information on the oscillation of $u_\eps$ at $\eps$-length-scale in the direction of a fixed vector $\zeta \in \R^d$ may be captured by considering the indicator functions 
	\[
	v_\eps(x) = \chi_{[0,1/2)}\Big(\areaB{\frac{(x\cdot \zeta)}{\eps} + a}\Big), \quad a \in \R^d,
	\]
	where for a point $x \in \Omega \subset \R^d$ we have denoted by $\area{x}$ its representative class in $K = \Tbb^d$ (the flat $d$-dimensional torus). Motivated by problems in homogenization theory, in particular the behavior of sequences of the form~\eqref{eq:intro_hom}, Pedregal studied (classical) Young measures generated by sequences of pairs 
	\[
	w_\eps(x) \coloneqq \Big\{\areaB{\frac{x}{\eps}}\,,\,u_\eps(x) \Big\},
	\]
	and called the resulting family $\tilde{\nu} = \{\tilde \nu_{x,\xi}\}_{x \in \Omega, \xi \in \Tbb^d}$ in~\eqref{eq:intro3}, the \emph{two-scale} Young measure generated by $\{u_\eps\}$. %In general, one may speak of \emph{multi}-scale Young measures by pairing additional test functions of the form $\{x \mapsto x/\delta\}$ with $\delta \ll \eps$.
	Since the Young measure generated by the sequence $(\area{\frarg/\eps})_\eps$ is precisely the Lebesgue measure restricted to $\Tbb^d$ (as the measure smears evenly throughout the torus), then~\eqref{eq:intro3} reads 
	\begin{equation}\label{eq:intro11}
	\sigma_x = (\Leb^d \restrict \Tbb^d) \otimes \tilde \nu_{\xi,x}.
	\end{equation}
	On the other hand $w_\eps \toY \sigma$, and hence the identity above leads to the limit representation %A rigorous deduction based on this identity leads (recall that $\{w_\eps\}$ generates the coupled Young measure $\bm{\nu}$) to the two-scale Young measure representation 
	\begin{align*}
	\int_\Omega f(x,{x}/{\eps},u_\eps(x)) \dd x & \to  \int_{\Omega} \int_{\Tbb^d \times \R^N} f(x,\xi,z) \dd \sigma_{x}(\xi,z) \\
	& =  \int_\Omega \bigg(\int_{\Tbb^d} \bigg(\int_{\R^N} f(x,\xi,z) \dd \tilde \nu_{x,\xi}(z)\bigg) \dd \xi \bigg)\dd x,
	\end{align*}
	which holds for all continuous integrands $f : \Omega \times \Tbb^d \times \R^N \to \R$ with uniform $\R^N$-linear growth in its third component. This comprises the construction of two-scale Young measures.

	\bibliographystyle{plain}
	\bibliography{zotero}
\end{document}